\documentclass{article}
\usepackage[utf8]{inputenc}
\usepackage{amsmath,amssymb,amsthm}
\usepackage{hyperref}
\usepackage{enumitem}   
\usepackage{setspace}
\usepackage{multirow}
\usepackage{outlines}
\usepackage[ margin=1in]{geometry}
\usepackage[affil-it]{authblk}

\usepackage{graphicx}
\graphicspath{{../Dissertation/Dissertation1/}}

\usepackage{float}
\usepackage{caption}
\usepackage{subcaption}

\newtheorem{theorem}{Theorem}[section]
\newtheorem{lemma}[theorem]{Lemma}

\usepackage[numbers]{natbib}
\let\cite=\citep

\begin{document}

\title{Stability and Hopf bifurcation analysis of a two state delay differential equation modeling the human respiratory system }

\author{Nirjal Sapkota%
  \thanks{Electronic address: \texttt{nxs167030@utdallas.edu}; Corresponding author}}
\affil{Department of Mathematical Sciences,\\ The University of Texas at Dallas\\ Richardson, TX, 75080, USA}

\author{Janos Turi%
  \thanks{Electronic address: \texttt{turi@utdallas.edu}}}
\affil{Department of Mathematical Sciences,\\ The University of Texas at Dallas\\ Richardson, TX, 75080, USA}

\date{}

	\maketitle

	\section*{Abstract}

	We study the two state model which describes the balance equation for carbon dioxide and oxygen. These are nonlinear parameter dependent and because of the transport delay in the respiratory control system, they are modeled with delay differential equation. So, the dynamics of a two state one delay model are investigated. By choosing the delay as a parameter, the stability and Hopf bifurcation conditions are obtained.  We notice that as the delay passes through its critical value, the positive equilibrium loses its stability and Hopf bifurcation occurs. The  stable region of the system with delay against the other parameters and bifurcation diagrams are also plotted. The three dimensional stability chart of the two state model is  constructed. We find that the  delay parameter has effect on the stability but not on the equilibrium state. 	The explicit derivation of the direction of  Hopf  	bifurcation and the stability of the bifurcation periodic solutions are determined with the help of normal form theory and center manifold theorem to delay differential equations. 	Finally, some numerical example and simulations are carried out to confirm the analytical findings. The numerical simulations verify the theoretical results.

	\section{Introduction}
	\label{ch:intro}

	In human respiratory system, the goal is to exchange the unwanted byproduct such as carbon dioxide for oxygen. The carbon dioxide is exchanged for oxygen by passive diffusion. Alveoli is the tiny  air sacs in the lungs where the exchange of oxygen and carbon dioxide takes place. The respiratory control system changes the rate of ventilation in response to the levels of oxygen and carbon dioxide in the body. The time delay is due to the physical distance where carbon dioxide and oxygen level information is transported to the sensory control system before the ventilatory response can be adapted.

	Understanding the human respiratory system is important for many medical conditions. The human respiratory and its control mechanics have been studied for more than hundred years.   This system has important medical implications some of which are listed below \cite{berry2012rules, leung2001sleep, khoo1982factors, eckert2007central}.
	\begin{outline}
		\1 Periodic Breathing: Periodic Breathing is define as three or more episodes of central apnea lasting at least 3 seconds, separated by no more than 20 seconds of normal breathing.  . 
		\1 Sleep Apnea: Sleep Apnea is a disorder in which breathing repeatedly stops and stars again. It is classified in two ways.
		\2 Central Sleep Apnea is a disorder which is characterized by a lack of drive to breathe.
		\2 Obstructive Sleep Apnea is a disorder which is characterized by episodes of partial or complete physical obstruction of the airflow.
		\1 Cheyne-Stokes Respiration: Cheyne-Stokes Respiration is a disorder which is characterized by gradual increase in breathing followed by decrease or absence of breathing. 
	\end{outline} 
	
	These disorders have been associated with a number of medical conditions such as hypertension, heart failures, diabetes and others \cite{lanfranchi1999prognostic, javaheri1998sleep, altevogt2006sleep}. 
	

	\subsection{Components of the Model}
	The carbon dioxide and oxygen levels are monitored at two respiratory centers in the body. These are called central  and peripheral chemoreceptors.

	\subsubsection{Central chemoreceptors}
	
	Central chemoreceptors are located at the ventral surface of the medulla in the brain. These respond to the changes in the partial pressure of carbon dioxide in the brain.
	
	\subsubsection{Peripheral chemoreceptors}
	
	Peripheral chemoreceptors are located in the  carotid bodies at the junction of the common carotid arteries and also at the aortic bodies. These respond to the changes in the partial pressure of both carbon dioxide and oxygen in arterial blood.
	
	Since these respiratory centers are located at a distance from the lungs where the  levels of the carbon dioxide and oxygen are regulated, there will be some delay (two transport delays)  in the process. This regulation is modeled with the ventilation function.
	
	\subsubsection{Ventilation function}
	We assume some conditions for the ventilation functions $V(x,y)$ to be biologically realistic model.
	
	\begin{itemize}
		\item $V(x,y) \geq 0$ and $ V(0,0) = 0$
		\item $V(x,y) $ is differentiable 
		\item $V(x,y)$ is an increasing function in both $x$ and $y$
		\item $ \frac{\partial V(x,y)}{\partial x} > 0 $ and  $ \frac{\partial V(x,y)}{\partial y} > 0$
	\end{itemize}
	
	\subsection{Model equation}
	
	Although a five-state model involving three compartments and two control loops with multiple delays is investigated in   ~\cite{khoo1982factors}, here we will  study a two state model with one time delay discussed in \cite{cooke1994stability} and \cite{kollar2005numerical}. A block diagram of the respiratory system is shown in Figure \ref{fig:RCS1}. The controller adjusts to inputs from the state (i.e., sleep, wakefulness)  and the chemoreceptors  which respond to the change in carbon dioxide and oxygen concentration.

	\begin{figure}[H]
		\centering
		\includegraphics[width=1\textwidth]{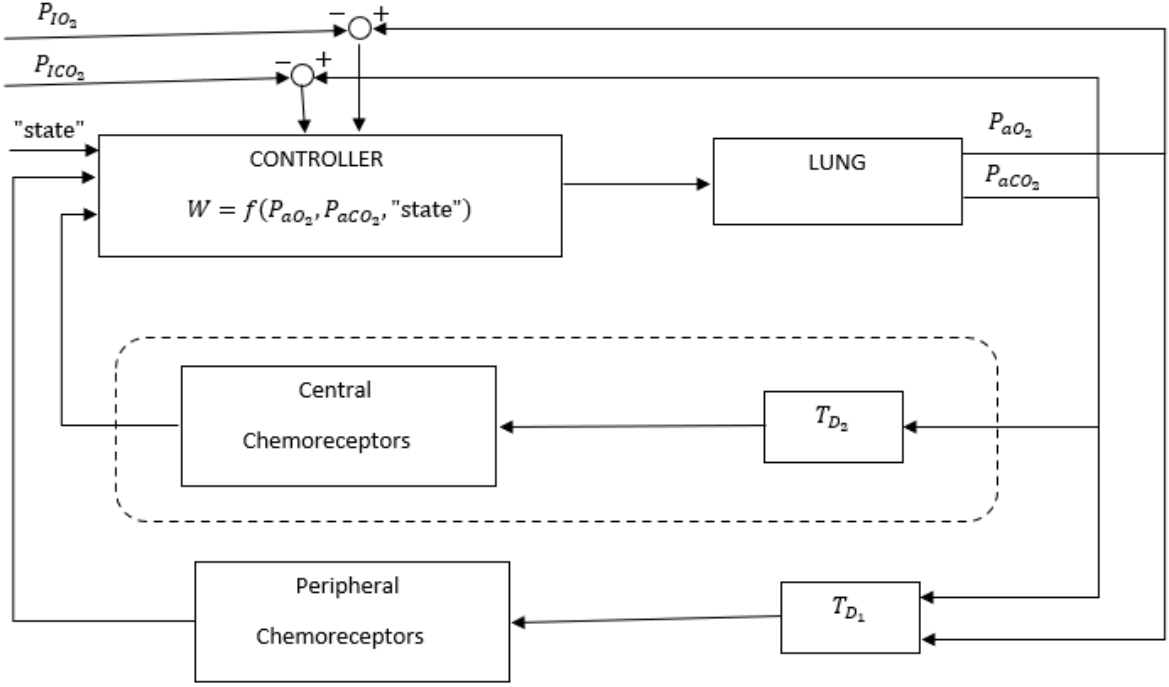}	
		\caption{Block diagram of respiratory control system}	
		\label{fig:RCS1}
	\end{figure}

	We consider the following two state model  for studying stability and bifurcation of a human respiratory system.
	
	\begin{equation} \label{eq:sys0}
		\begin{aligned}
			\frac{\mathrm{d} \tilde{x}}{\mathrm{d} t} &= p - \alpha  \,  W(\tilde{x}(t - \tau), \tilde{y}(t - \tau)) (\tilde{x}(t) -x_I)\\
			\frac{\mathrm{d} \tilde{y}}{\mathrm{d} t} &= -\sigma + \beta  \,  W(\tilde{x}(t - \tau), \tilde{y}(t - \tau)) (y_I - \tilde{y}(t))
		\end{aligned}
	\end{equation}
	where
	
	\begin{itemize}
		\item $ \tilde{x}(\cdot) $, $ \tilde{y}(\cdot) $  represent the arterial carbon dioxide and oxygen concentration
		\item $ W(\cdot,\cdot)  $ is the ventilation function which represents the volume of gas moved by the respiratory system
		\item $ \tau  $ is the transport delay ($\tau >0$ and $\tau = T_{D_{1}}$ in Figure \ref{fig:RCS1})
		\item $x_I$, $y_I$ are inspired carbon dioxide and oxygen concentration
		\item $p$ is the carbon dioxide production rate
		\item $\sigma$ is the oxygen consumption rate
		\item $\alpha$, $\beta$ are positive constants associated with the diffusibility of carbon dioxide and oxygen respectively
	\end{itemize}

	For studying the stability analysis with a more convenient system, we convert the system (\ref{eq:sys0}) using
	\begin{equation}
		\begin{aligned}
			x(t) &= a(\tilde{x}(t) - x_I)\\
			y(t) &= b(y_I - \tilde{y}(t))
		\end{aligned}
	\end{equation}
	Solving for $\tilde{x}(t-\tau)$ and $\tilde{y}(t-\tau)$, we get,
	
	\begin{equation} \label{eq:xytau1}
		\begin{aligned}
			\tilde{x}(t-\tau) &= x_I + \frac{1}{a}x(t-\tau)\\
			\tilde{y}(t-\tau) &= y_I - \frac{1}{b}y(t-\tau)\\
		\end{aligned}
	\end{equation}
	
	Using Equation (\ref{eq:xytau1}) in Equation (\ref{eq:sys0}), we obtain
	
	\begin{equation} \label{eq:sys0p5}
		\begin{aligned}
			\frac{\mathrm{d} x}{\mathrm{d} t} &= a \frac{\mathrm{d} \tilde{x}}{\mathrm{d} t}= a p - a \alpha W\left(x_I + \frac{1}{a}x(t-\tau), y_I - \frac{1}{b}y(t-\tau)\right)\frac{x(t)}{a}\\
			\frac{\mathrm{d} y}{\mathrm{d} t} &= -b \frac{\mathrm{d} \tilde{y}}{\mathrm{d} t} = b \sigma-  b \beta W\left(x_I + \frac{1}{a}x(t-\tau), y_I - \frac{1}{b}y(t-\tau)\right)\frac{y(t)}{a} \\
		\end{aligned}
	\end{equation}
	Setting $a = 1/p$ and $b = 1/\sigma$, we obtain the  equations
	\begin{equation} \label{eq:sys1}
		\begin{aligned}
			\frac{\mathrm{d} x}{\mathrm{d} t} &= 1 - \alpha \,  V(x(t-\tau), y(t-\tau))\, x(t)\\
			\frac{\mathrm{d} y}{\mathrm{d} t} &= 1 - \beta \, V(x(t-\tau), y(t-\tau))\, y(t)\\	
		\end{aligned}
	\end{equation}
	where the ventilation function is given by 
	\begin{equation} \label{eq:vent0}
		V(x(t-\tau), y(t-\tau)) = W(\tilde{x}(t-\tau), \tilde{y}(t-\tau))
	\end{equation}
	We will study the system (\ref{eq:sys1}) with 
	\begin{equation} \label{eq:vent1}
		V(x(t-\tau), y(t-\tau)) = 0.14\,e^{-0.05(100-y(t-\tau))} \, x(t-\tau)
	\end{equation}
	The state variables are concentrations in our model.
	
	
	\section{Stability and Hopf Bifurcation} 
	\label{stability}
	
	In recent years,  a lot of delay differential equations modeling various chemical, biological, ecological systems have been studied \cite{bi2013bifurcations, du2010hopf, epstein1998introduction, su2009hopf, murray2002mathematical, yafia2007hopf, belair1994stability, roose2007continuation}. With the outbreak of COVID-19 pandemic, many  mathematical models using delay differential equations have been proposed \cite{menendez2020elementary, rihan2021dynamics, shayak2020delay, paul2021distribution, dell2020solvable, guglielmi2022delay, liu2020covid}.
	
Li and Zhang \cite{li2021dynamic}, B{\i}lazero{\u{g}}lu \cite{bilazerouglu2022hopf}, studied the dynamic analysis and Hopf bifurcation of a Lengyel-Epstein system with two delays. Li \cite{li1999stability} studied a class of delay differential equations with two delays. Kumar et al \cite{ kumar2020hopf} proposed a multiple delayed innovation diffusion model with Holling II functional response. Delayed predator-prey system have been investigated by many researchers  (\cite{song2005local, yan2006hopf, bairagi2011stability, song2004stability, khellaf2010boundedness, zhang2013hopf, wang2012hopf, sun2007analysis, celik2015stability, ccelik2008stability, tang2007stability, zhou2012hopf}). Ghosh et al \cite{ghosh2021dynamics} studied the rumor spread mechanism and the influential factors using epidemic like model. Several researchers have analyzed the U{\c{c}}ar  prototype system \cite{uccar2002prototype, bhalekar2016stability, li2004hopf, li2019bifurcation}. Wei \cite{wei2007bifurcation} discussed the dynamics of a scalar delay differential equation. Gilsinn \cite{gilsinn2002estimating} estimates the bifurcation parameter of delay differential equation with application to machine tool chatter. There has been a focus on studying stability and Hopf bifurcation by choosing the delay as a parameter of the system with the linear stability methods.

	The complex system modeling the human respiratory control system have been studied for several decades. Mackey and Glass \cite{macke1977oscillation}, Khoo et al \cite{khoo1982factors}, Batzel et al \cite{batzel2000stability, batzel2007cardiovascular} have investigated stability analysis.

	
	\subsection{Equilibrium point}

	\begin{lemma}
		\label{posequi}
		There is a unique positive equilibrium point $E_{*}(x_{*}, y_{*})$ of system (\ref{eq:sys1}).
		
	\end{lemma}
	\begin{proof}
		The equilibrium point  $E_{*}(x_{*}, y_{*})$ is obtained by solving
		\begin{equation}
			\begin{aligned}
				1 -\alpha\, 0.14\,e^{-0.05(100-y_{*})}x_{*}^2 & = 0 \\
				1 -\beta\, 0.14\,e^{-0.05(100-y_{*})}y_{*} & = 0\\
			\end{aligned}
		\end{equation}

		We also notice that
		$$x_{*} \neq 0, y_{*} \neq 0,  \text{ and }  x_{*} = \frac{\beta}{\alpha} y_{*}$$
		Then, rewriting as exact fractions and solving the equation
		\begin{equation}
			1 - \alpha \left( \frac{14}{100} \right) \,e^{-\frac{5}{100}(100-y_{*})} \frac{\beta^2 y_{*}^2}{\alpha^2} = 0
		\end{equation}	
		we get,
		\begin{equation}  \label{eq:ystar}
			y_{*} = 40\, W\left(\frac{e^{5/2} \sqrt{\frac{\alpha }{\beta ^2}}}{4 \sqrt{14}}\right)
		\end{equation}
		and 	
		\begin{equation} \label{eq:xstar}
			x_{*} = 40\, \left(\frac{\beta}{\alpha}\right) W\left(\frac{e^{5/2} \sqrt{\frac{\alpha }{\beta ^2}}}{4 \sqrt{14}}\right)
		\end{equation}
		where $W$ represents the Lambert $W$-function.\\
		Since $V(0,0) = 0$  and $$  V\left(\frac{\beta y_{*}}{\alpha} , y^{*}\right) =  \left( \frac{14}{100} \right) \,e^{-\frac{5}{100}(100-y_{*})} \frac{\beta y_{*}}{\alpha}$$ is increasing in $y_{*},$ there is a unique positive solution $y_{*}$. \\
		For the default values of $\alpha = 0.5$ and $\beta = 0.8$, we get $(x_{*}, y_{*}) \approx (29.1842, 18.2401).$ 
		
	\end{proof}
	
	We plot the Equations (\ref{eq:xstar}) and (\ref{eq:ystar}) as a function of $\alpha$ and $\beta$ in Figure \ref{fig:equi_xyvsalpha_beta0p8} and \ref{fig:equi_xyvsbeta_alpha0p5}.
	
	\begin{figure}[H]
		\centering
		\begin{subfigure}[b]{0.45\textwidth}
			\centering
			\includegraphics[width=1\textwidth]{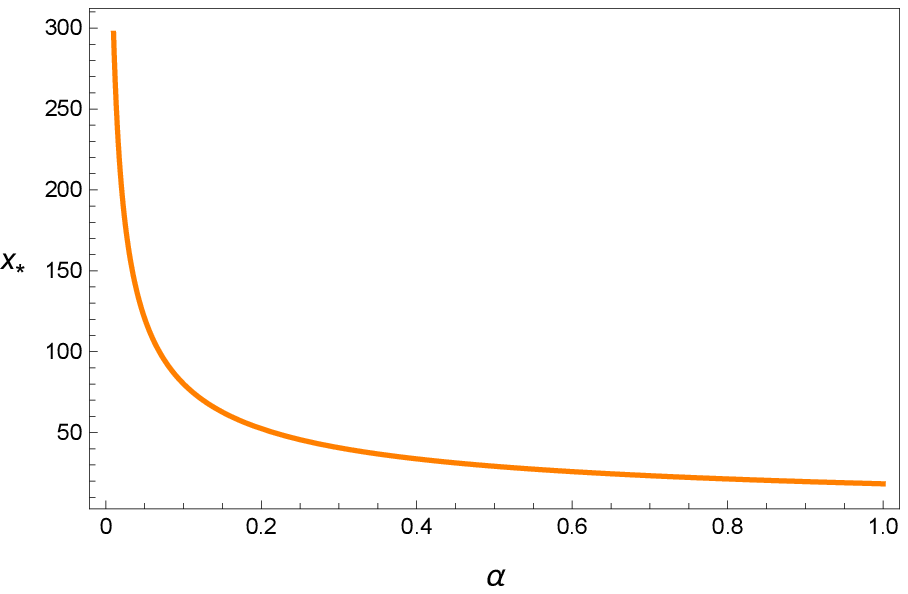}		
			\label{fig:xstarvsalpha_beta0p8}
		\end{subfigure}
		\hfill
		\begin{subfigure}[b]{0.45\textwidth}
			\centering
			\includegraphics[width=1\textwidth]{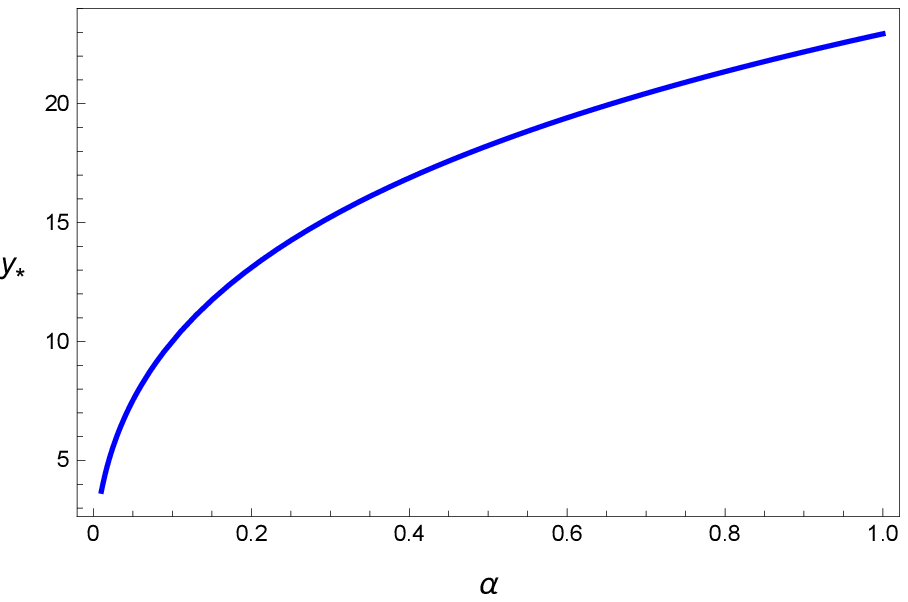}		
			\label{fig:ystarvsalpha_beta0p8}
		\end{subfigure}
		\caption{Equilibrium points as a function of $\alpha$ with $\beta = 0.8$.}
		\label{fig:equi_xyvsalpha_beta0p8}
	\end{figure}
	
	\begin{figure}[H]
		\centering
		\begin{subfigure}[b]{0.45\textwidth}
			\centering
			\includegraphics[width=1\textwidth]{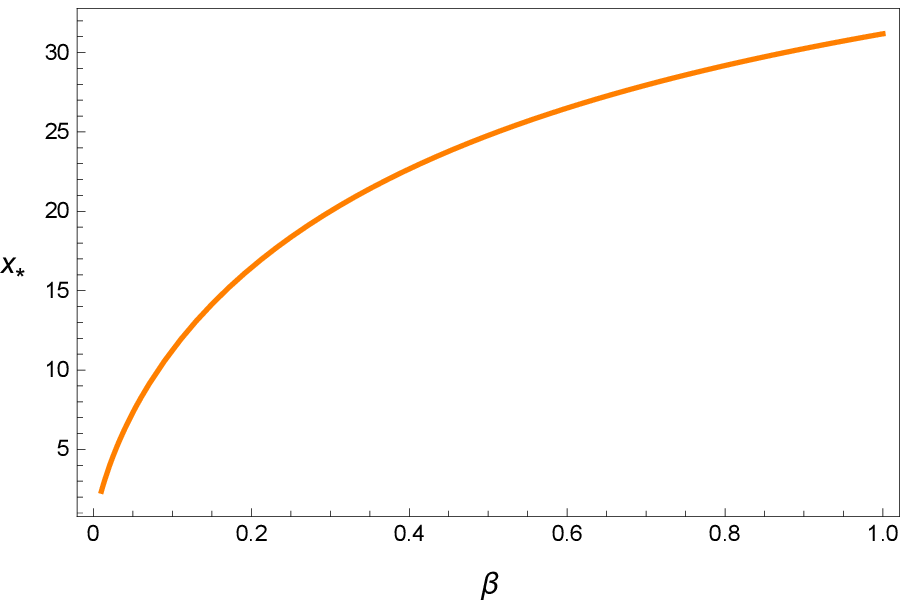}		
			\label{fig:xstarvsbeta_alpha0p5}
		\end{subfigure}
		\hfill
		\begin{subfigure}[b]{0.45\textwidth}
			\centering
			\includegraphics[width=1\textwidth]{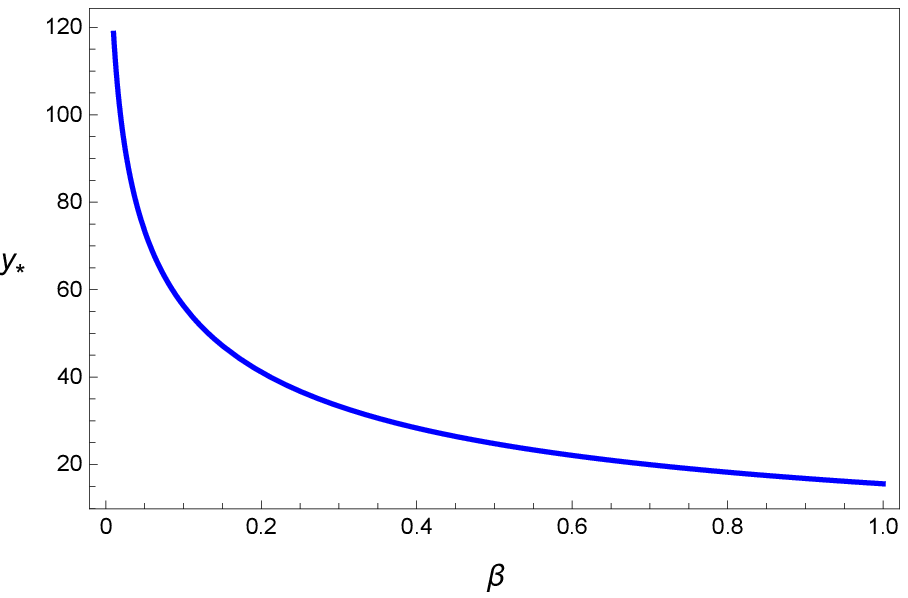}		
			\label{fig:ystarvsbeta_alpha0p5}
		\end{subfigure}
		\caption{Equilibrium points as a function of $\beta$ with $\alpha = 0.5$.}
		\label{fig:equi_xyvsbeta_alpha0p5}
	\end{figure}

	\subsection{Stability of the Equilibrium point}
	Let $u(t) = x(t) - x_{*},\, v(t) = y(t) - y_{*}.$ Then, the linearized system of (\ref{eq:sys1}) is given as follows:
	
	\begin{equation}  \label{eq:linearizeda}
		\begin{aligned}
			\frac{\mathrm{d} u(t)}{\mathrm{d} t} &=  -\alpha V(x_{*}, y_{*}) u(t) - \alpha x_{*} V_{x}(x_{*}, y_{*}) u(t - \tau) -  \alpha x_{*} V_{y}(x_{*}, y_{*}) v(t - \tau)  \\
			\frac{\mathrm{d} v(t)}{\mathrm{d} t} &=  -\beta V(x_{*}, y_{*}) v(t) - \beta y_{*} V_{x}(x_{*}, y_{*}) v(t - \tau) -  \beta y_{*} V_{y}(x_{*}, y_{*}) v(t - \tau) 			
		\end{aligned}	
	\end{equation}
	
	This could be written in the form as
	\begin{equation} \label{eq:linearized}
		\frac{\mathrm{d} }{\mathrm{d} t} \begin{pmatrix}
			u(t)  \\ v(t)
		\end{pmatrix}  + A_1  \begin{pmatrix} 
			u(t)  \\ v(t)
		\end{pmatrix} + B_1 \begin{pmatrix}
			u(t - \tau)  \\ v(t - \tau) \end{pmatrix}  = 
		\begin{pmatrix}
			0  \\0
		\end{pmatrix}		
	\end{equation}
	where 
	\begin{equation*} A_1 =  
		\begin{pmatrix}
			\alpha V(x_{*},y_{*}) & 0 \\0  & \beta V(x_{*},y_{*}) 
		\end{pmatrix} = 		
		\left(
		\begin{array}{cc}
			\frac{7}{50} \alpha  x_* e^{\frac{1}{20} \left(y_*-100\right)} & 0 \\
			0 & \frac{7}{50} \beta  x_* e^{\frac{1}{20} \left(y_*-100\right)} \\
		\end{array}
		\right)
	\end{equation*}	and 	
	\begin{equation*}			
		B_1 =  
		\begin{pmatrix}
			\alpha x_{*} V_{x}(x_{*},y_{*}) & 	\alpha x_{*} V_{y}(x_{*},y_{*}) \\ 	\beta y_{*} V_{x}(x_{*},y_{*})  &\beta y_{*} V_{y}(x_{*},y_{*})
		\end{pmatrix} =
		\left(
		\begin{array}{cc}
			\frac{7}{50} \alpha  x_* e^{\frac{1}{20} \left(y_*-100\right)} & \frac{7 \alpha  x_*^2 e^{\frac{1}{20} \left(y_*-100\right)}}{1000} \\
			\frac{7}{50} \beta  e^{\frac{1}{20} \left(y_*-100\right)} y_* & \frac{7 \beta  x_* e^{\frac{1}{20} \left(y_*-100\right)} y_*}{1000} \\
		\end{array}
		\right)	
	\end{equation*}
	
	The associated  characteristic equation of the linear system (\ref{eq:linearized}) is  
	\begin{equation} \label{eq:char0}
		\Delta(\lambda, \tau) = \det(\lambda I + A_1 + B_1 e^{-\tau \lambda}) = 0
	\end{equation}	
	That is,
	\begin{equation} \label{eq:chareqn1}
		\lambda ^2+\frac{7 \lambda  x_* e^{-\lambda  \tau +\frac{y_*}{20}-5} \left(20 (\alpha +\beta ) e^{\lambda  \tau }+20 \alpha +\beta  y_*\right)}{1000}+\frac{49 \alpha  \beta  x_*^2 e^{-\lambda  \tau +\frac{y_*}{10}-10} \left(20 e^{\lambda  \tau }+y_*+20\right)}{50000} = 0
	\end{equation}
	
	This could be also written as
	\begin{equation} \label{eq:charABCD1}
		\Delta(\lambda, \tau) = \lambda ^2 +  \left(A+B e^{-\lambda  \tau }\right) \lambda + \left(C + D e^{-\lambda  \tau }\right) = 0 
	\end{equation}
	where
	\begin{equation} \label{eq:ABCD1}
		\begin{aligned}
			A &= \frac{7}{50} \alpha  x_* e^{\frac{y_*}{20}-5}+\frac{7}{50} \beta  x_* e^{\frac{y_*}{20}-5}\\
			B &= \frac{7}{50} \alpha  x_* e^{\frac{y_*}{20}-5}+\frac{7 \beta  x_* e^{\frac{y_*}{20}-5} y_*}{1000} \\
			C &= \frac{49 \alpha  \beta  x_*^2 e^{\frac{y_*}{10}-10}}{2500}\\
			D &= \frac{49 \alpha  \beta  x_*^2 e^{\frac{y_*}{10}-10}}{2500}+\frac{49 \alpha  \beta  x_*^2 e^{\frac{y_*}{10}-10} y_*}{50000}		
		\end{aligned}
	\end{equation}

	\begin{itemize}
		\item For $\tau = 0$
	\end{itemize}
	The characteristic equation (\ref{eq:chareqn1}) simplifies to
	\begin{equation} \label{eq:}
		\lambda ^2+\frac{7 \lambda  x_* e^{\frac{y_*}{20}-5} \left(40 \alpha +\beta  \left(y_*+20\right)\right)}{1000}+\frac{49 \alpha  \beta  x_*^2 e^{\frac{y_*}{10}-10} \left(y_*+40\right)}{50000} = 0
	\end{equation}
	
	The coefficients in this equation are postivie, and therefore the roots have negative real parts.
	
	\begin{itemize}
		\item For $\tau > 0$
	\end{itemize}
	
	The change in stability of eigenvalue $\lambda$ can occur if $\text{Re}(\lambda) = 0.$ Let $\lambda = i \omega$ and characteristic equation (\ref{eq:charABCD1}) takes the form
	\begin{equation} \label{eq:chareqnw}
		-\omega ^2 + i \omega  \left(A+B e^{-i \tau  \omega }\right)+  \left(C+D e^{-i \tau  \omega }\right) = 0
	\end{equation}
	
	Solving for the real and imaginary parts of both sides, we get
	\begin{equation} \label{eq:reim1}
		\begin{aligned}
			C -\omega ^2 	&= - B \omega  \sin (\tau  \omega ) - D \cos (\tau  \omega )\\
			A \omega &=  D \sin (\tau  \omega ) - B \omega  \cos (\tau  \omega )
		\end{aligned}
	\end{equation}

	Squaring and adding (\ref{eq:reim1}), we get the relation
	\begin{equation} \label{eq:fomega1}
		f(\omega) = -\omega ^4 + \left(-A^2+B^2+2 C\right) \omega ^2  + \left(- C^2+D^2\right) = 0 
	\end{equation}

	Let 
	\begin{equation}
		M = -A^2+B^2+2 C = 	\frac{49 \beta  x_*^2 e^{\frac{y_*}{10}-10} \left(-400 \beta +40 \alpha  y_*+\beta  y_*^2\right)}{1000000}
	\end{equation}
	
	Then $M$ is positive if $40 \alpha  y_*+\beta  y_*^2 > 400 \beta$

	and let 
	\begin{equation}
		N = - C^2 + D^2 = \frac{2401 \alpha ^2 \beta ^2 x_*^4 e^{\frac{y_*}{5}-20} y_* \left(y_*+40\right)}{2500000000} > 0	
	\end{equation}
	
	Assuming that $Y= \omega^2$, we can write (\ref{eq:fomega1}) as
	\begin{equation} \label{eq:Phiy}
		\Phi(Y) = -Y^2 + M Y + N = 0 	
	\end{equation}

	This means that the Equation (\ref{eq:Phiy}) has one positive root. Solving for $\tau$ from (\ref{eq:reim1}), we have the critical curves given by

	
	\begin{multline} \label{eq:taustar}
		\tau_*(n) = \\
		\frac{2000 \sqrt{2} e^5 \pi  n}{7 \sqrt{\sqrt{N_3}+N_4}}\\
		\pm  
		\frac{1000 i \sqrt{2} e^5 \log \left(\frac{e^{-\frac{y_*}{20}} \left(\beta  N_2 x_*^2 e^{\frac{y_*}{10}} \left(40 \alpha +\beta  N_1\right)-20 i \sqrt{2} \sqrt{\sqrt{N_3}+N_4} x_* e^{\frac{y_*}{20}} (\alpha +\beta )+\sqrt{N_3}\right)}{x_* \left(40 \alpha  \beta  N_1 x_* e^{\frac{y_*}{20}}+i \sqrt{2} \sqrt{\sqrt{N_3}+N_4} \left(20 \alpha +\beta  y_*\right)\right)}\right)}{7 \sqrt{\sqrt{N_3}+N_4}}\\
		(n = 0, \pm1,\pm2,...)	
	\end{multline}
	where
	\begin{equation}
		\begin{aligned}
			N_1 &= (y_* + 20)\\
			N_2 &= (y_* - 20)\\
			N_3 &= \beta ^2 x_*^4 e^{\frac{y_*}{5}} \left(y_*+20\right) \left(3200 \alpha ^2 y_*+80 \alpha  \beta  \left(y_*-20\right) y_*+\beta ^2 \left(y_*-20\right){}^2 \left(y_*+20\right)\right)\\
			N_4 &= 	\beta  x_*^2 e^{\frac{y_*}{10}} \left(40 \alpha  y_*+\beta  \left(y_*^2-400\right)\right)	
		\end{aligned}
	\end{equation}
	
	Consequently, we can count that the stability region are restricted between the set of two curves
	\begin{multline} \label{eq:tau1}
		\tau_1(n) = \\
		\frac{2000 \sqrt{2} e^5 \pi  n}{7 \sqrt{\sqrt{N_3}+N_4}}\\
		+  
		\frac{1000 i \sqrt{2} e^5 \log \left(\frac{e^{-\frac{y_*}{20}} \left(\beta  N_2 x_*^2 e^{\frac{y_*}{10}} \left(40 \alpha +\beta  N_1\right)-20 i \sqrt{2} \sqrt{\sqrt{N_3}+N_4} x_* e^{\frac{y_*}{20}} (\alpha +\beta )+\sqrt{N_3}\right)}{x_* \left(40 \alpha  \beta  N_1 x_* e^{\frac{y_*}{20}}+i \sqrt{2} \sqrt{\sqrt{N_3}+N_4} \left(20 \alpha +\beta  y_*\right)\right)}\right)}{7 \sqrt{\sqrt{N_3}+N_4}}\\
		(n = 0, 1, 2,...)	
	\end{multline}
	
	\begin{multline} \label{eq:tau2}
		\tau_2(n) = \\
		\frac{2000 \sqrt{2} e^5 \pi  n}{7 \sqrt{\sqrt{N_3}+N_4}}\\
		-  
		\frac{1000 i \sqrt{2} e^5 \log \left(\frac{e^{-\frac{y_*}{20}} \left(\beta  N_2 x_*^2 e^{\frac{y_*}{10}} \left(40 \alpha +\beta  N_1\right)-20 i \sqrt{2} \sqrt{\sqrt{N_3}+N_4} x_* e^{\frac{y_*}{20}} (\alpha +\beta )+\sqrt{N_3}\right)}{x_* \left(40 \alpha  \beta  N_1 x_* e^{\frac{y_*}{20}}+i \sqrt{2} \sqrt{\sqrt{N_3}+N_4} \left(20 \alpha +\beta  y_*\right)\right)}\right)}{7 \sqrt{\sqrt{N_3}+N_4}}\\
		(n =  1, 2,...)	
	\end{multline}
	
	Notice that Equation (\ref{eq:tau1}) starts with $n = 0, 1,2,...$ and Equation (\ref{eq:tau2}) starts with $n = 1, 2, ...$ for the pair of curves to have positive values of $\tau.$ If $\mathrm{Re}\left(d\lambda/d\tau\right)$ have different sign on any two consecutive critical curves, then the stability region is confined between these two curves in the $(\tau, \alpha, \beta)$ parameter space \cite{lakshmanan2011dynamics}.

	Now we look to verify the transversality condition:
	\begin{equation}
		\text{Re} \left( \frac{\mathrm{d} \lambda}{\mathrm{d} \tau} \right) > 0
	\end{equation}
	
	at $\tau = \tau_* $ with $ n = 0.$
	
	Differentiating characteristic equation (\ref{eq:charABCD1}) with respect to $\tau$, we get
	\begin{equation}
		2 \lambda  \frac{\mathrm{d} \lambda}{\mathrm{d} \tau} + \left(A + B e^{-\lambda  \tau }\right) \frac{\mathrm{d} \lambda}{\mathrm{d} \tau} + B e^{-\lambda  \tau } \lambda \left(-\lambda - \tau \frac{\mathrm{d} \lambda}{\mathrm{d} \tau}\right) +  D e^{-\lambda  \tau }  \left(-\lambda - \tau \frac{\mathrm{d} \lambda}{\mathrm{d} \tau}\right) = 0
	\end{equation}
	Solving for $	\left[\mathrm{d} \lambda / \mathrm{d} \tau \right]^{-1}, $ we get
	\begin{equation}
		\begin{aligned}
			\left[ \frac{\mathrm{d} \lambda}{\mathrm{d} \tau}  \right]^{-1} &= 	\frac{(A+2 \lambda ) e^{\lambda  \tau }}{\lambda  (B \lambda +D)}-\frac{B \lambda  \tau -B+D \tau }{\lambda  (B \lambda +D)}\\
			&= -\frac{A D-B C+B \lambda ^2+2 D \lambda }{\lambda  \left(A \lambda +C+\lambda ^2\right) (B \lambda +D)}-\frac{\tau }{\lambda }
		\end{aligned}
	\end{equation}
	
	On critical curves i.e with $\tau = \tau_*$ and $\lambda = i \omega$,  and solving for the real part, we have
	
	\begin{equation}
		\begin{aligned}
			\text{Re} \left(\left[ \frac{\mathrm{d} \lambda}{\mathrm{d} \tau}  \right]^{-1}\right) &= \frac{A^2-2 C+2 \omega ^2}{A^2 \omega ^2+\left(C-\omega ^2\right)^2}-\frac{B^2}{B^2 \omega ^2+D^2}\\
			&= \frac{A^2-2 C+2 \omega ^2}{B^2 \omega ^2+D^2}-\frac{B^2}{B^2 \omega ^2+D^2}\\
			&= \frac{A^2-B^2-2 C+2 \omega ^2}{B^2 \omega ^2+D^2} \\
			&= \frac{-C^2+D^2+\omega ^4}{\omega ^2 \left(B^2 \omega ^2+D^2\right)}\\
			&= \frac{N+\omega ^4}{\omega ^2 \left(B^2 \omega ^2+D^2\right)} > 0, \quad \textrm{since } N > 0
		\end{aligned}
	\end{equation}

	Since $\mathrm{Re}\left(d\lambda/d\tau\right) > 0 $ for all the critical curves (\ref{eq:tau1}) and (\ref{eq:tau2}), the corresponding slopes have positive values on all the stability determining critical curves. Thus, there are no  eigenvalues with negative real part across the critical curves. Further, we know that for $\tau = 0$ the equilibrium points $(x_*, y_*)$ are stable. Therefore, there can be only one stable region in the $(\tau, \alpha)$
	or $(\tau, \beta)$ plane enclosed between the line $\tau = 0$ and the curve $\tau_1(0).$

	The critical curves for various $n$ as a function of $\alpha$ and $\beta$ are shown in Figures \ref{fig:cc_beta0p8} and \ref{fig:cc_alpha0p5}.
	\begin{figure}[H]
		\centering
		\includegraphics[width=0.8\linewidth]{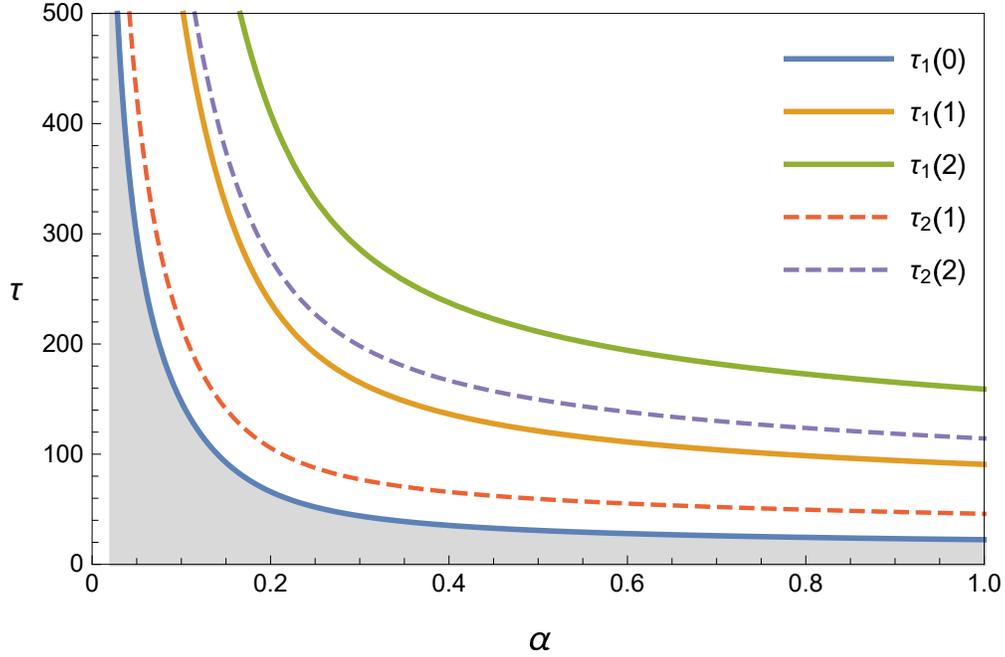}
		\caption{Critical curves for  equilibrium  $ (x_*,y_*) $ with $\beta = 0.8. $ The solid curves represents $\tau_1$ for $n = 0, +1, +2$ and dashed curves represent $\tau_2$ for $n = +1, +2.$ The region enclosed (shaded region) between the line $\tau = 0$ and the curve $\tau = \tau_1(0)$ is the only stable region.}
		\label{fig:cc_beta0p8}
	\end{figure}
	
	\begin{figure}[H]
		\centering
		\includegraphics[width=0.8\linewidth]{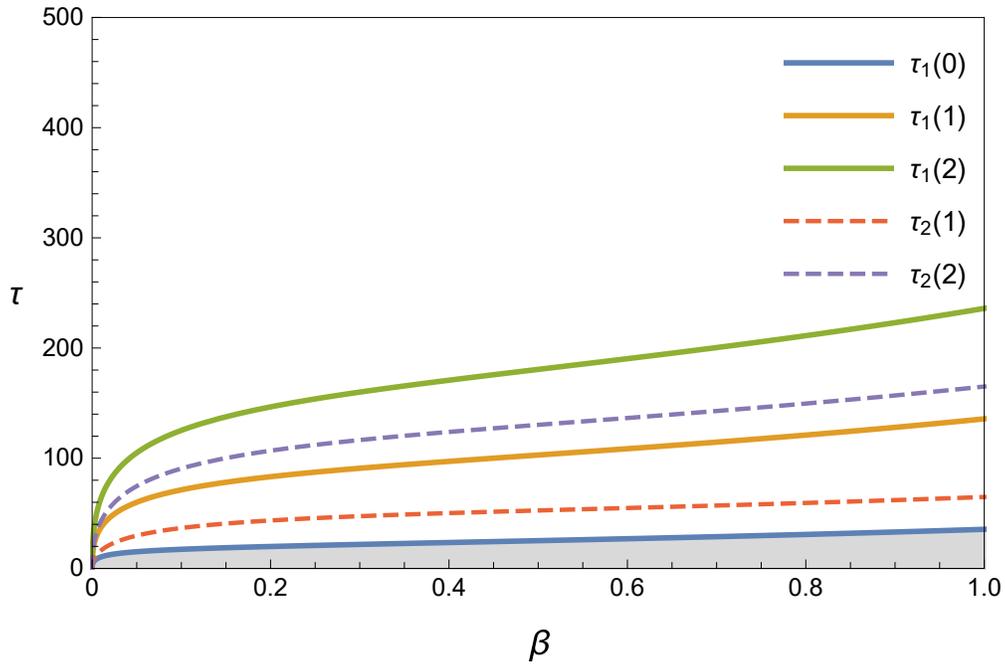}
		\caption{Critical curves for  equilibrium  $ (x_*,y_*) $ with $\alpha = 0.5. $ The solid curves represents $\tau_1$ for $n = 0, +1, +2$ and dashed curves represent $\tau_2$ for $n = +1, +2.$ The region enclosed (shaded region) between the line $\tau = 0$ and the curve $\tau = \tau_1(0)$ is the only stable region.}
		\label{fig:cc_alpha0p5}
	\end{figure}	
	
	The critical surfaces  in the parameter space $(\alpha, \beta, \tau)$ that encompass the stable region is shown in Figure \ref{fig:ccs_tauvsalphavsbeta}.
	
	\begin{figure}[H]
		\centering
		\includegraphics[width=0.7\linewidth]{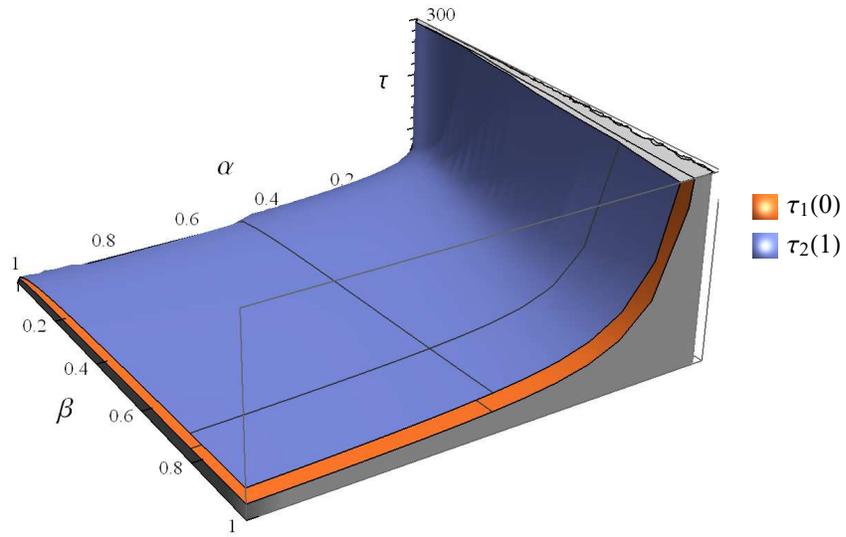}
		\caption{Critical surfaces of the two state human respiratory system.}
		\label{fig:ccs_tauvsalphavsbeta}
	\end{figure}
	\begin{figure}[H]
		\centering
		\includegraphics[width=0.7\linewidth]{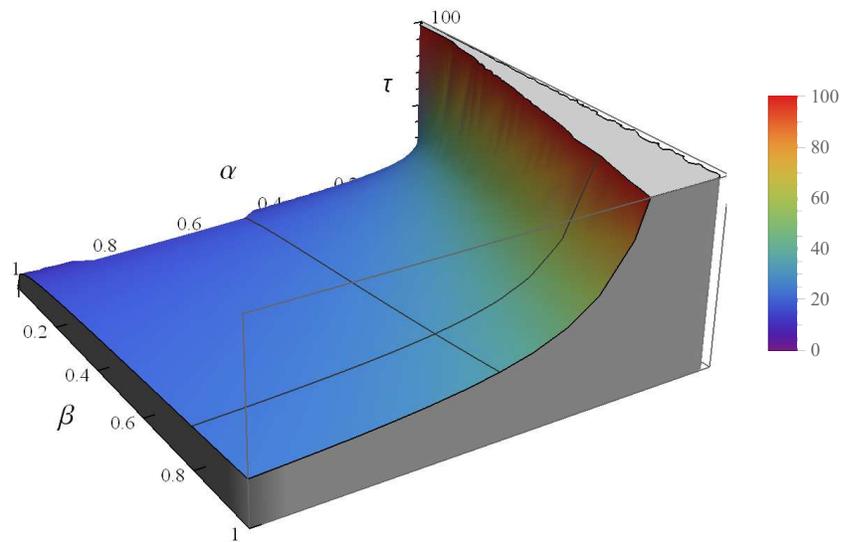}
		\caption{Stability chart of the two state human respiratory system.}
		\label{fig:sr_tauvsalphavsbeta}
	\end{figure}
	
		The 3 dimensional stability chart of the two  state model of a human respiratory system in the parameter space $(\alpha, \beta, \tau)$ is shown in Figure \ref{fig:sr_tauvsalphavsbeta}.
	
	%

	For a more general case, the following theorem was proved in \cite{cooke1994stability}.
	\begin{theorem}
		Let $V_{*} = V(x_*, y_*), $ $V_{x*} = V_{x}(x_*, y_*) $ and  $V_{y*} = V_{y}(x_*, y_*). $ 
		\begin{enumerate}
			\item If $V_{*} \geq x_{*}V_{x*} +  y_{*}V_{y*}, $ then the equilibrium $(x_*, y_*)$ is asymptotically stable for all delay $\tau \geq 0.$
			\item If $V_{*} < x_{*}V_{x*} +  y_{*}V_{y*}, $ then there exists $\tau_{*} > 0 $ such that the  equilibrium $(x_*, y_*)$ is asymptotically stable if $0 \leq \tau < \tau_* $ and unstable if $\tau > \tau_{*}.$ 
		\end{enumerate}
	\end{theorem}
	
	Therefore, from the above discussions, the following results can be directly deduced for our case. 
	
	Let  $\tau_*$ be defined by (\ref{eq:taustar}). Then,
	\begin{enumerate}[label=(\roman*)]
		\item The positive equilibrium  $E_{*}(x_{*}, y_{*})$ of system (\ref{eq:sys1}) is asymptotically stable for $0 \leq \tau < \tau_*$
		\item The positive equilibrium  $E_{*}(x_{*}, y_{*})$ of system (\ref{eq:sys1}) is  unstable for $ \tau > \tau_*$
		\item  System (\ref{eq:sys1}) undergoes Hopf bifurcation at the positive equilibrium  $E_{*}(x_{*}, y_{*})$ for $\tau = \tau_*$
	\end{enumerate}

	\subsection{Critical Delay and Bifurcation}
	
	If we consider $\tau$ as a parameter, then as $\tau$ passes through its critical value $\tau_*$, the positive equilibrium $E_{*}(x_{*}, y_{*})$ loses its stability. The maximum value of the real part of the characteristic equation is computed for several values of $\tau.$ This is shown in figure \ref{fig:maxroots_alpha0p5_beta0p8}. For the default values of $\alpha = 0.5$, $\beta = 0.8$, the equilibrium is $E_{*}(x_{*}, y_{*}) \approx (29.1842, 18.2401 )$ and   the critical delay is $\tau_* \approx 30.8017.$ 
	
	\begin{figure}[H]
		\centering
		\includegraphics[width=1\textwidth]{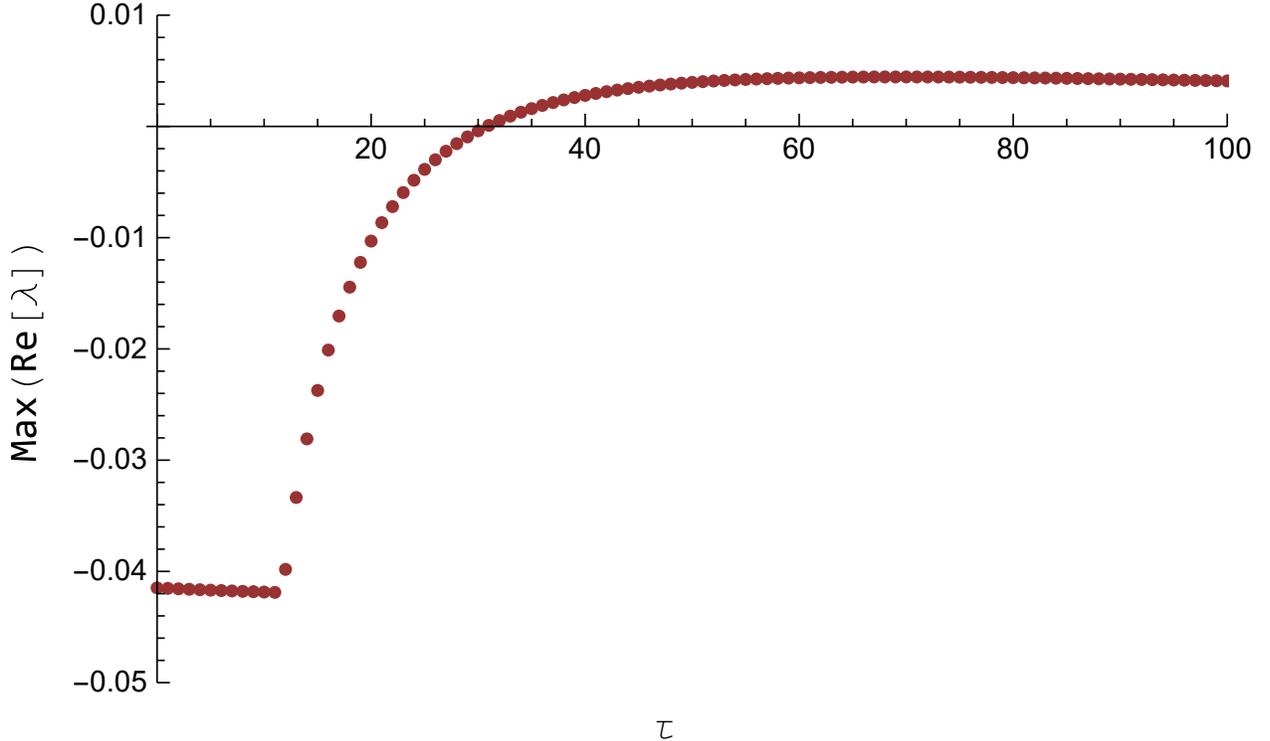}
		\caption{The maximum value of the real part of the eigenvalue with  $\alpha = 0.5, \; \beta = 0.8. $ }
		\label{fig:maxroots_alpha0p5_beta0p8}
	\end{figure}

	Table \ref{table:maxreal1} lists the largest real part of the characteristic root computed for several values of $\tau$ near the critical delay with  $\alpha = 0.5$ and $\beta = 0.8.$
	
	\begin{table}[H]
		\begin{center}
			\caption{The largest real part of the eigenvalues with $\alpha = 0.5$ and $\beta = 0.8$}
			\label{table:maxreal1}
			\begin{tabular}{cr}
				\hline
				$\tau$ & Max(Re[$\lambda$])\\
				\hline
				25 & -0.00386067 \\
				26 & -0.0029966 \\
				27 & -0.00222993 \\
				28 & -0.00154774 \\
				29 & -0.000939186 \\
				30 & -0.000395051 \\
				31 & 0.0000925033 \\
				32 & 0.000530187 \\
				33 & 0.000923769 \\
				34 & 0.00127823 \\
				35 & 0.00159789 \\			
				\hline
			\end{tabular}
		\end{center}
	\end{table}
	
	Figure \ref{fig:max_roots_parameters_vary} shows the maximum value of the real part of the characteristic equation for varying the parameters $\alpha$ and $\beta.$

	\begin{figure}[H]
		\centering
		\begin{subfigure}[b]{0.3\textwidth}
			\centering
			\includegraphics[width=1\textwidth]{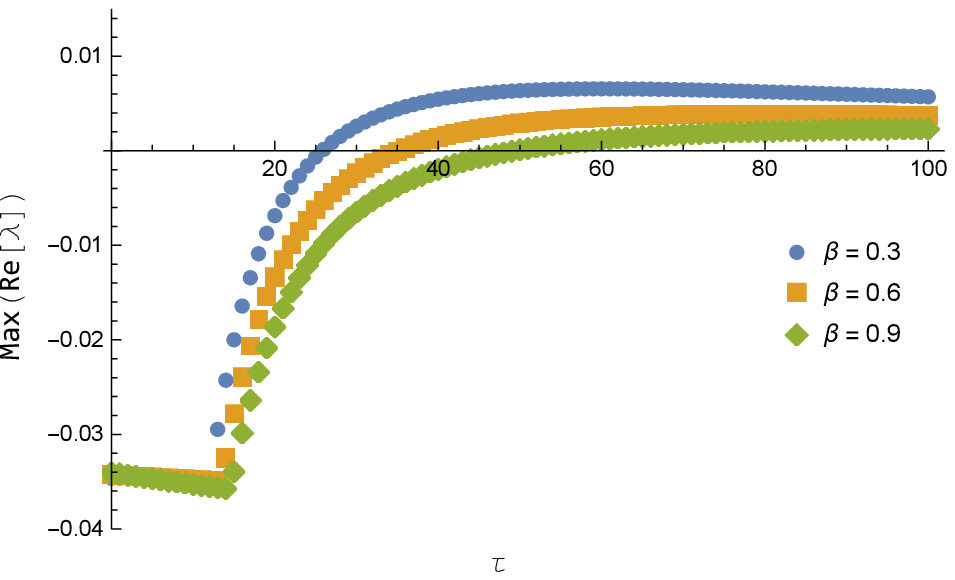}
			\caption{$\alpha=0.3$}
			\label{fig:maxroots_betaVary_alpha0p2}
		\end{subfigure}
		\hfill
		\begin{subfigure}[b]{0.3\textwidth}
			\centering
			\includegraphics[width=1\textwidth]{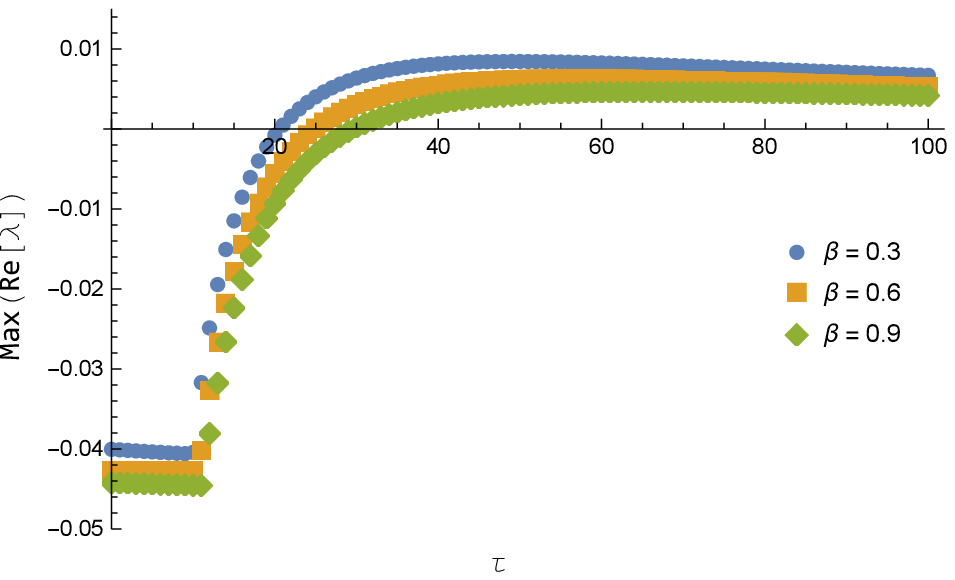}
			\caption{$\alpha=0.6$}
			\label{fig:maxroots_betaVary_alpha0p5}
		\end{subfigure}
		\hfill
		\begin{subfigure}[b]{0.3\textwidth}
			\centering
			\includegraphics[width=1\textwidth]{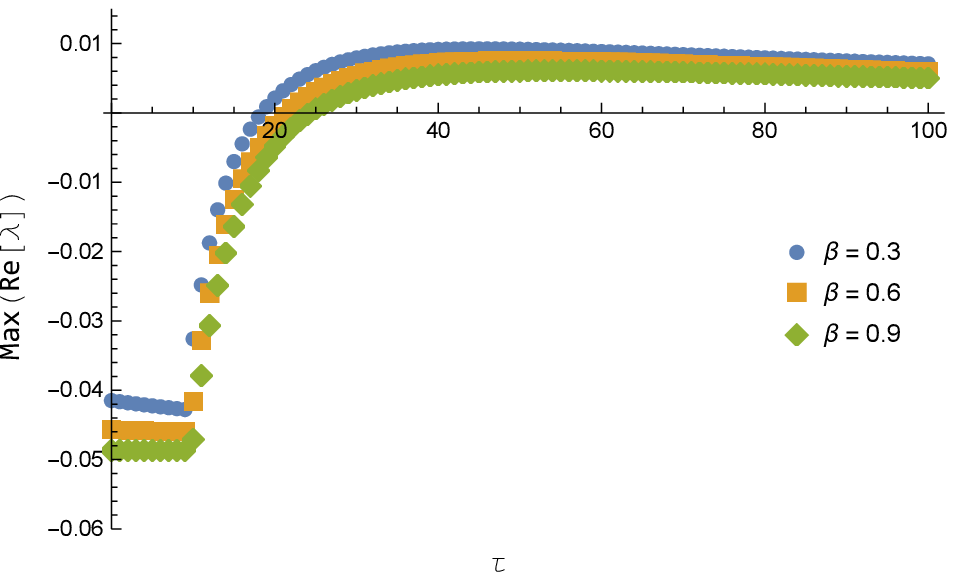}
			\caption{$\alpha=0.9$}
			\label{fig:maxroots_betaVary_alpha0p8}
		\end{subfigure}
		\caption{The maximum value of the real part of the eigenvalue with various values of $\alpha$ and $\beta.$}
		\label{fig:max_roots_parameters_vary}
	\end{figure}

	The following tables list out the largest part of the eigenvalues with various combination of $\alpha$ and $\beta$ near the critical delay.
	
	\begin{table}[H]
		\begin{center}
			\caption{The largest real part of the eigenvalues with $\alpha = 0.3$}
			\label{table:maxreal_alpha0p3}
			\begin{tabular}{ccrccrccr}
				\hline
				$\beta$ & $\tau$ & Max(Re[$\lambda$]) &  $\beta$ & $\tau$ & Max(Re[$\lambda$]) & $\beta$ & $\tau$ & Max(Re[$\lambda$])\\
				\hline
				0.3 & 20 & -0.00686887 & 0.6 & 30 & -0.00237351 & 0.9 & 45 & -0.00076911 \\ 
				0.3 & 21 & -0.00526380 & 0.6 & 31 & -0.00181421 & 0.9 & 46 & -0.00057352 \\ 
				0.3 & 22 & -0.00387139 & 0.6 & 32 & -0.00130903 & 0.9 & 47 & -0.00039104 \\ 
				0.3 & 23 & -0.00265816 & 0.6 & 33 & -0.00085181 & 0.9 & 48 & -0.00022063 \\ 
				0.3 & 24 & -0.00159690 & 0.6 & 34 & -0.00043725 & 0.9 & 49 & -0.00006136 \\ 
				0.3 & 25 & -0.00066533 & 0.6 & 35 & -0.00006074 & 0.9 & 50 & 0.00008762 \\ 
				0.3 & 26 & 0.00015497 & 0.6 & 36 & 0.00028174 & 0.9 & 51 & 0.00022707 \\ 
				0.3 & 27 & 0.00087928 & 0.6 & 37 & 0.00059370 & 0.9 & 52 & 0.00035769 \\
				0.3 & 28 & 0.00152043 & 0.6 & 38 & 0.00087822 & 0.9 & 53 & 0.00048012 \\
				0.3 & 29 & 0.00208920 & 0.6 & 39 & 0.00113802 & 0.9 & 54 & 0.00059495 \\ 
				0.3 & 30 & 0.00259473 & 0.6 & 40 & 0.00137550 & 0.9 & 55 & 0.00070271 \\ 
				\hline
			\end{tabular}
		\end{center}
	\end{table}

	\begin{table}[H]
		\begin{center}
			\caption{The largest real part of the eigenvalues with $\alpha = 0.6$}
			\label{table:maxreal_alpha0p6}
			\begin{tabular}{ccrccrccr}
				\hline
				$\beta$ & $\tau$ & Max(Re[$\lambda$]) &  $\beta$ & $\tau$ & Max(Re[$\lambda$]) & $\beta$ & $\tau$ & Max(Re[$\lambda$])\\
				\hline
				0.3 & 15 & -0.01148690 & 0.6 & 20 & -0.00560191 & 0.9 & 25 & -0.00312474 \\ 
				0.3 & 16 & -0.00853185 & 0.6 & 21 & -0.00412443 & 0.9 & 26 & -0.00229615 \\ 
				0.3 & 17 & -0.00607124 & 0.6 & 22 & -0.00284708 & 0.9 & 27 & -0.00156234 \\
				0.3 & 18 & -0.00400648 & 0.6 & 23 & -0.00173803 & 0.9 & 28 & -0.00091068 \\ 
				0.3 & 19 & -0.00226226 & 0.6 & 24 & -0.00077144 & 0.9 & 29 & -0.00033055 \\ 
				0.3 & 20 & -0.00078019 & 0.6 & 25 & 0.00007382 & 0.9 & 30 & 0.00018705 \\ 
				0.3 & 21 & 0.00048556 & 0.6 & 26 & 0.00081518 & 0.9 & 31 & 0.00064978 \\ 
				0.3 & 22 & 0.00157136 & 0.6 & 27 & 0.00146711 & 0.9 & 32 & 0.00106420 \\ 
				0.3 & 23 & 0.00250636 & 0.6 & 28 & 0.00204171 & 0.9 & 33 & 0.00143594 \\ 
				0.3 & 24 & 0.00331419 & 0.6 & 29 & 0.00254916 & 0.9 & 34 & 0.00176985 \\ 
				0.3 & 25 & 0.00401410 & 0.6 & 30 & 0.00299806 & 0.9 & 35 & 0.00207014 \\
				\hline
			\end{tabular}
		\end{center}
	\end{table}

	\begin{table}[H]
		\begin{center}
			\caption{The largest real part of the eigenvalues with $\alpha = 0.9$}
			\label{table:maxreal_alpha0p9}
			\begin{tabular}{ccrccrccr}
				\hline
				$\beta$ & $\tau$ & Max(Re[$\lambda$]) &  $\beta$ & $\tau$ & Max(Re[$\lambda$]) & $\beta$ & $\tau$ & Max(Re[$\lambda$])\\
				\hline
				0.3 & 13 & -0.01396730 & 0.6 & 16 & -0.00957134 & 0.9 & 20 & -0.00478992 \\
				0.3 & 14 & -0.01012840 & 0.6 & 17 & -0.00709898 & 0.9 & 21 & -0.00339803 \\ 
				0.3 & 15 & -0.00701248 & 0.6 & 18 & -0.00502401 & 0.9 & 22 & -0.00219773 \\ 
				0.3 & 16 & -0.00445748 & 0.6 & 19 & -0.00327084 & 0.9 & 23 & -0.00115831 \\ 
				0.3 & 17 & -0.00234397 & 0.6 & 20 & -0.00178089 & 0.9 & 24 & -0.00025488 \\ 
				0.3 & 18 & -0.00058240 & 0.6 & 21 & -0.00050815 & 0.9 & 25 & 0.00053292 \\ 
				0.3 & 19 & 0.00089540 & 0.6 & 22 & 0.00058388 & 0.9 & 26 & 0.00122183 \\ 
				0.3 & 20 & 0.00214212 & 0.6 & 23 & 0.00152449 & 0.9 & 27 & 0.00182576 \\ 
				0.3 & 21 & 0.00319896 & 0.6 & 24 & 0.00233739 & 0.9 & 28 & 0.00235633 \\ 
				0.3 & 22 & 0.00409852 & 0.6 & 25 & 0.00304190 & 0.9 & 29 & 0.00282329 \\ 
				0.3 & 23 & 0.00486685 & 0.6 & 26 & 0.00365396 & 0.9 & 30 & 0.00323487 \\ 
				\hline
			\end{tabular}
		\end{center}
	\end{table}

	We now list the equilibrium point   $(x_*,y_*)$ and critical delay $\tau_*$ for these combinations of $\alpha$ and $\beta.$

	\begin{table}[H]
		\begin{center}
			\caption{The equilibrium point and critical delay for various values of  $\alpha$ and $\beta$}
			\label{table:eq_delay}
			\begin{tabular}{cccc}
				\hline
				$\alpha$ & $\beta$ & $(x_*, y_*)$ & $\tau_*$ \\
				\hline
				0.3 & 0.3 & (28.8782, 28.8782) & 25.8012 \\ 
				0.3 & 0.6 & (37.2949, 18.6474) & 35.1706 \\ 
				0.3 & 0.9 & (41.9183, 13.9728) & 49.4039 \\ 
				0.6 & 0.3 & (17.5118, 35.0237) & 20.5978 \\ 
				0.6 & 0.6 & (23.4108,  23.4108) & 24.9072 \\ 
				0.6 & 0.9 & (26.8631,  17.9087) & 29.6255 \\ 
				0.9 & 0.3 & (12.9723,  38.9169) & 18.3737 \\ 
				0.9 & 0.6 & (17.6832, 26.5248) & 21.4466 \\ 
				0.9 & 0.9 & (20.5381, 20.5381) & 24.3089 \\ 
				\hline
			\end{tabular}
		\end{center}
	\end{table}

	In figure \ref{fig:lambda_tauVary}, the characteristic roots of the smallest modulus are  shown for the default values of $\alpha = 0.5,$ and $\beta = 0.8$ while varying the parameter $\tau$ near the critical delay.

	\begin{figure}[H]
		\centering
		\begin{subfigure}[b]{0.3\textwidth}
			\centering
			\includegraphics[width=1\textwidth]{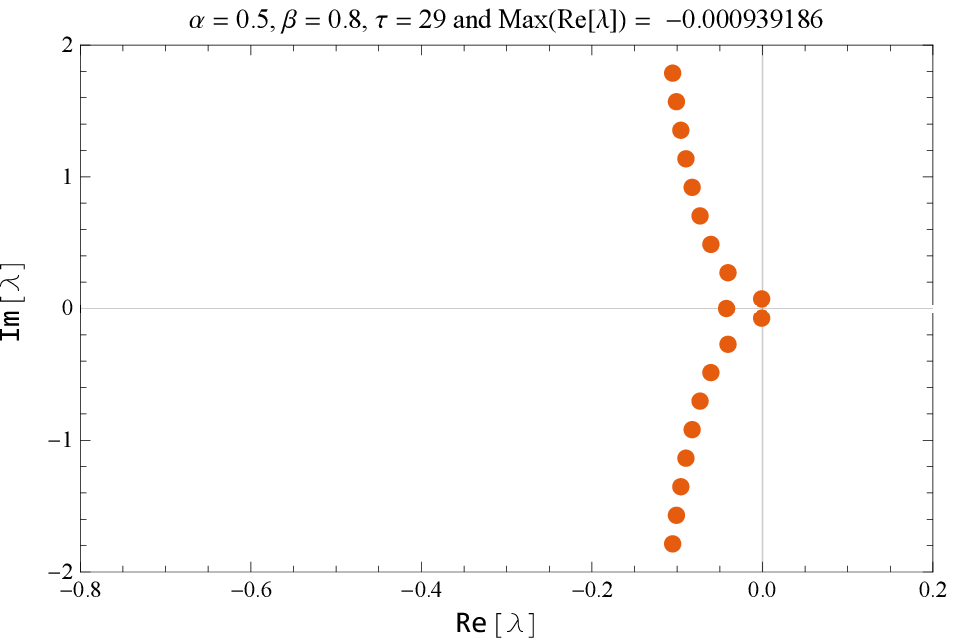}		
			\label{fig:lambda_a0p5b0p8t29}
		\end{subfigure}
		\hfill
		\begin{subfigure}[b]{0.3\textwidth}
			\centering
			\includegraphics[width=1\textwidth]{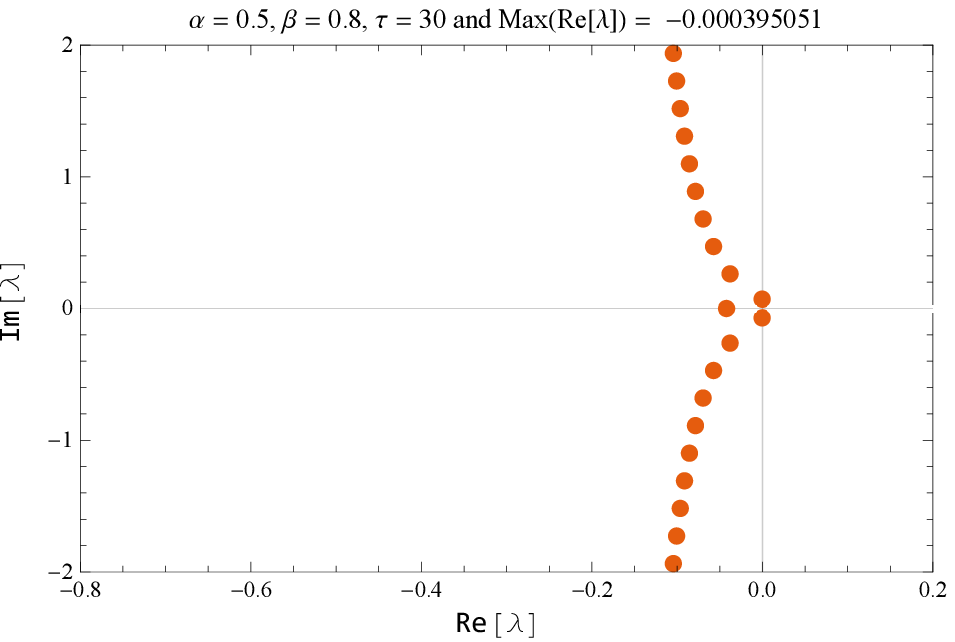}	
			\label{fig:lambda_a0p5b0p8t30}
		\end{subfigure}
		\hfill
		\begin{subfigure}[b]{0.3\textwidth}
			\centering
			\includegraphics[width=1\textwidth]{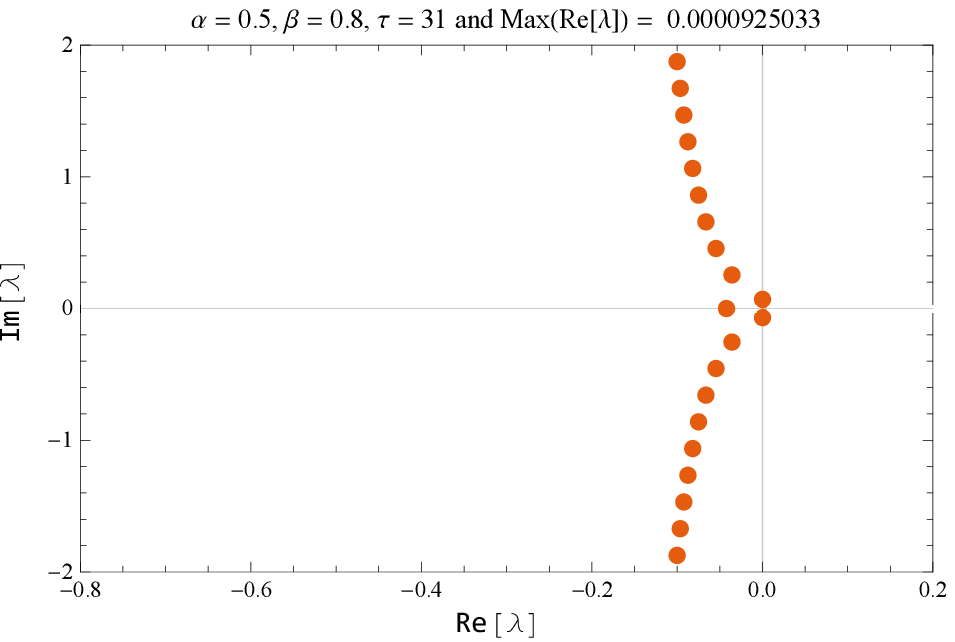}
			\label{fig:lambda_a0p5b0p8t31}
		\end{subfigure}
		
		\caption{The characteristic roots of the smallest modulus with various values of  $\alpha, $  $\beta, $ and $\tau$ near the critical delay.}
		\label{fig:lambda_tauVary}
	\end{figure}

	Figure \ref{fig:lambda_parameters_vary} shows the characteristic roots of the smallest modulus for various parameters of $\alpha,$  $\beta$ and $\tau.$

	\begin{figure}[H]
		\centering
		\begin{subfigure}[b]{0.3\textwidth}
			\centering
			\includegraphics[width=1\textwidth]{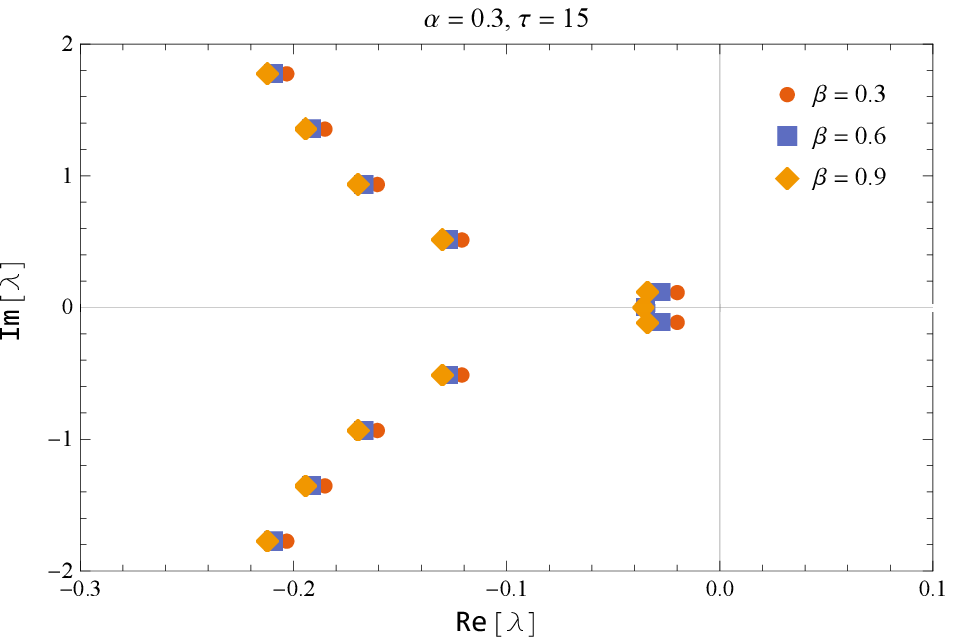}		
			\label{fig:lambda_a0p3t15_betaVary}
		\end{subfigure}
		\hfill
		\begin{subfigure}[b]{0.3\textwidth}
			\centering
			\includegraphics[width=1\textwidth]{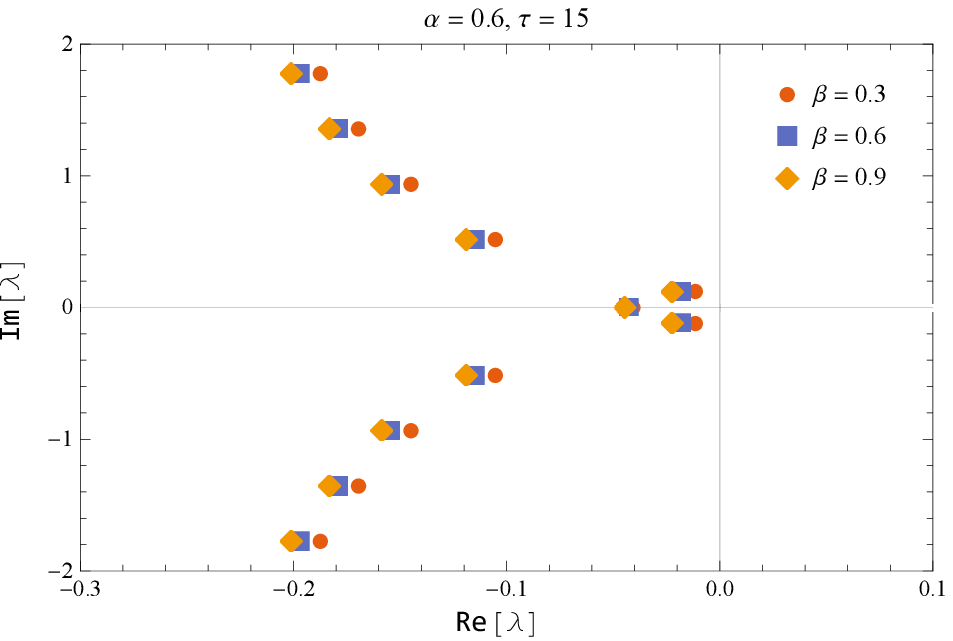}	
			\label{fig:lambda_a0p6t15_betaVary}
		\end{subfigure}
		\hfill
		\begin{subfigure}[b]{0.3\textwidth}
			\centering
			\includegraphics[width=1\textwidth]{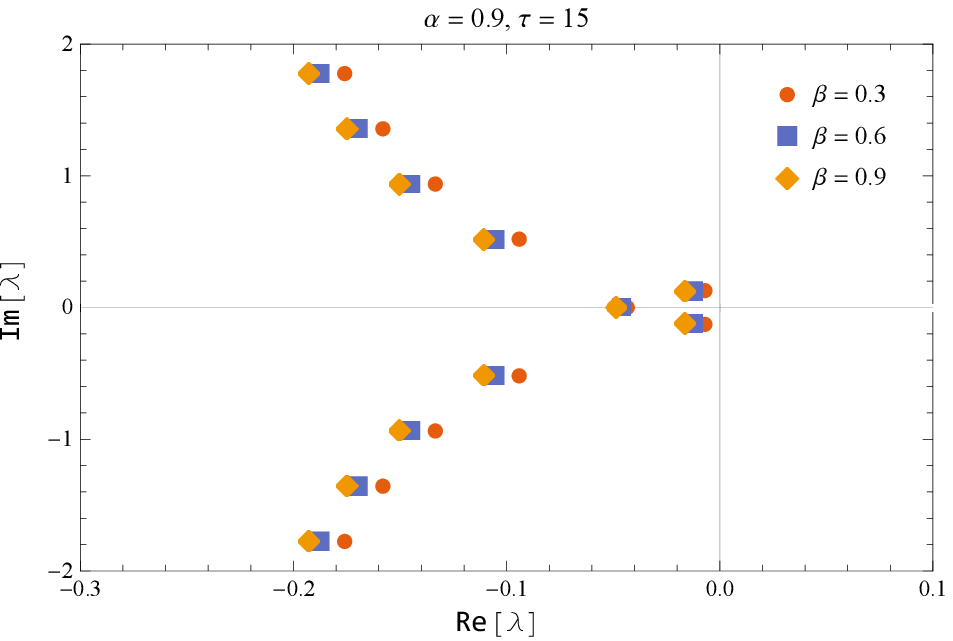}
			\label{fig:lambda_a0p9t15_betaVary}
		\end{subfigure}
		\begin{subfigure}[b]{0.3\textwidth}
			\centering
			\includegraphics[width=1\textwidth]{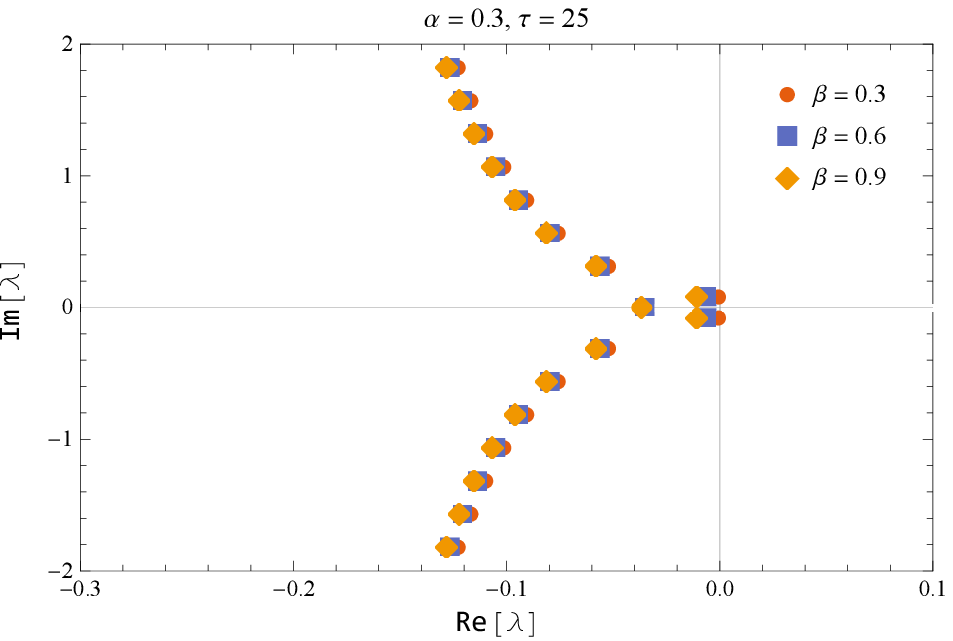}		
			\label{fig:lambda_a0p3t25_betaVary}
		\end{subfigure}
		\hfill
		\begin{subfigure}[b]{0.3\textwidth}
			\centering
			\includegraphics[width=1\textwidth]{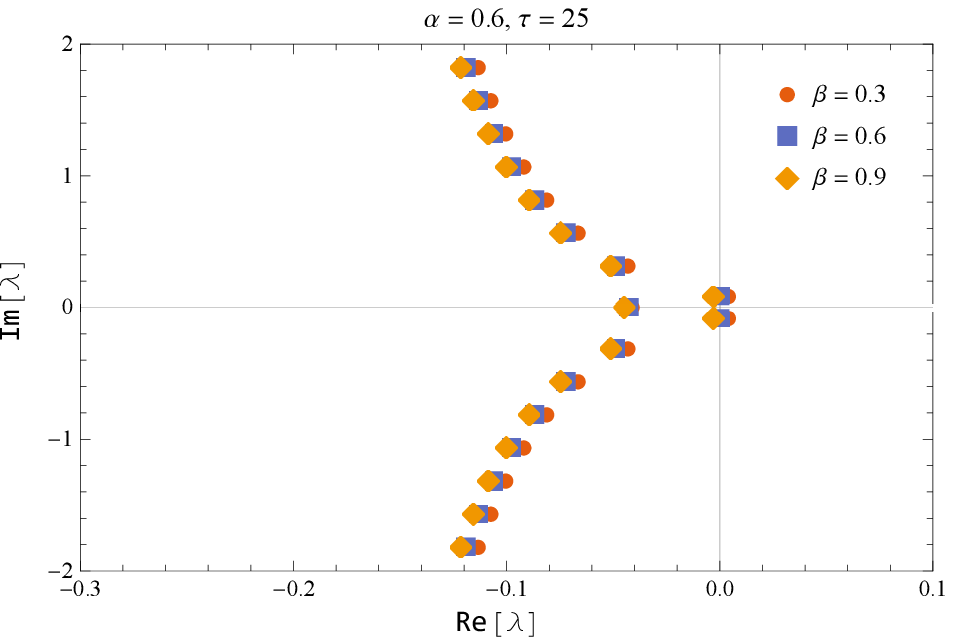}	
			\label{fig:lambda_a0p6t25_betaVary}
		\end{subfigure}
		\hfill
		\begin{subfigure}[b]{0.3\textwidth}
			\centering
			\includegraphics[width=1\textwidth]{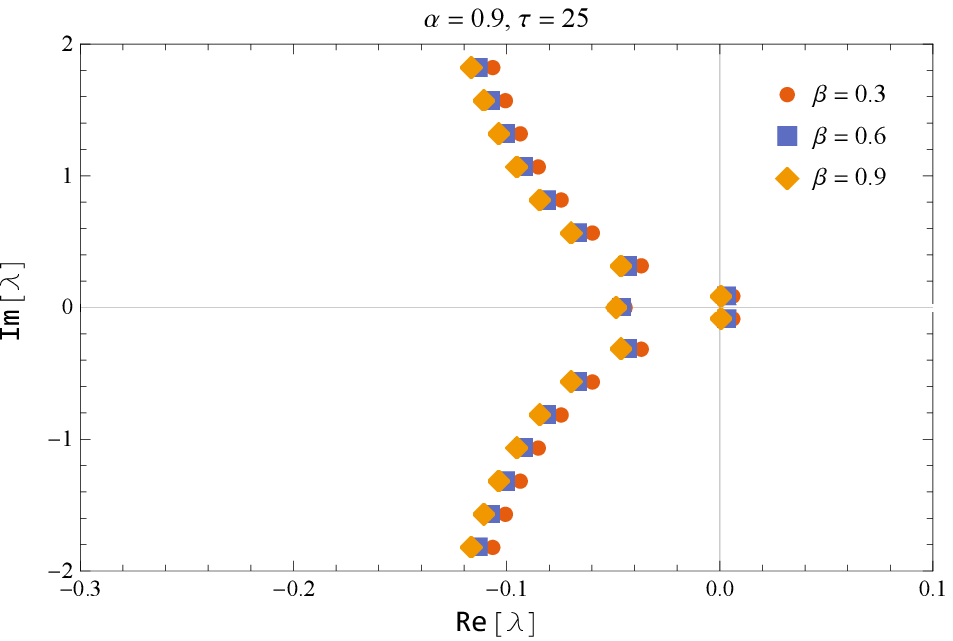}
			\label{fig:lambda_a0p9t25_betaVary}
		\end{subfigure}
		\begin{subfigure}[b]{0.3\textwidth}
			\centering
			\includegraphics[width=1\textwidth]{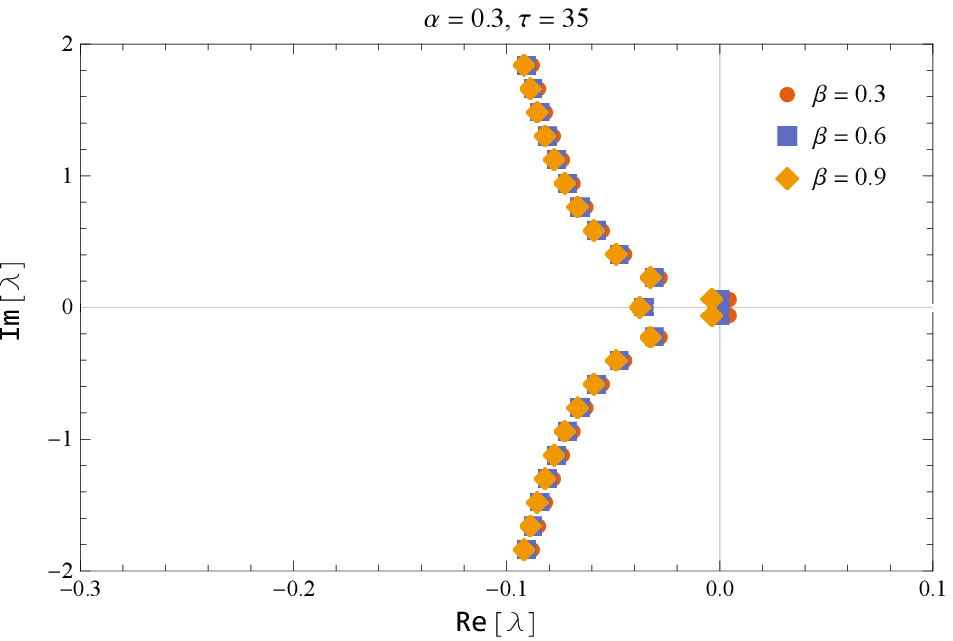}		
			\label{fig:lambda_a0p3t35_betaVary}
		\end{subfigure}
		\hfill
		\begin{subfigure}[b]{0.3\textwidth}
			\centering
			\includegraphics[width=1\textwidth]{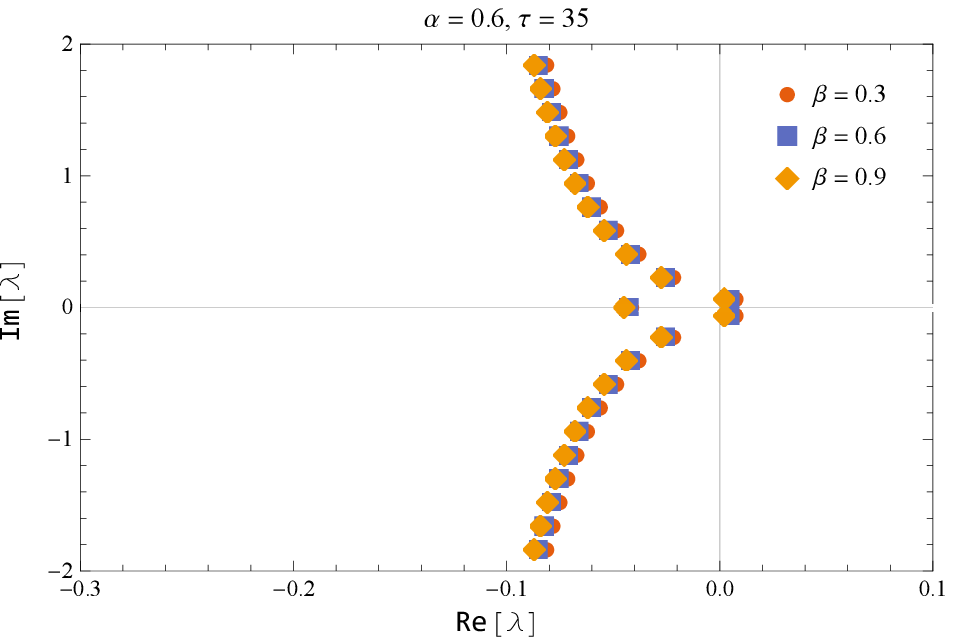}	
			\label{fig:lambda_a0p6t35_betaVary}
		\end{subfigure}
		\hfill
		\begin{subfigure}[b]{0.3\textwidth}
			\centering
			\includegraphics[width=1\textwidth]{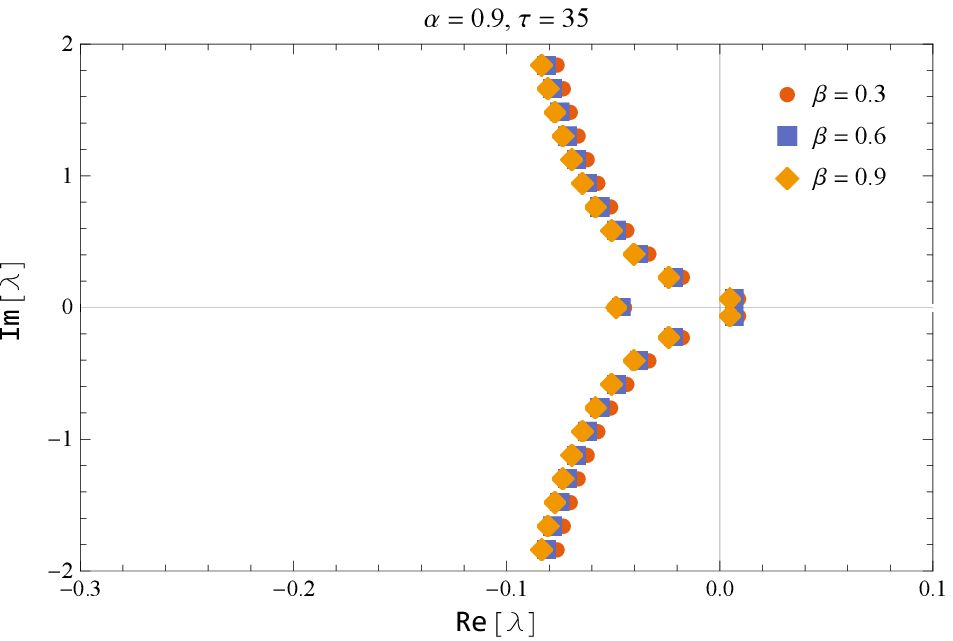}
			\label{fig:lambda_a0p9t35_betaVary}
		\end{subfigure}	
		\caption{The characteristic roots of the smallest modulus with various values of  $\alpha, $  $\beta, $ and $\tau.$}
		\label{fig:lambda_parameters_vary}
	\end{figure}

	Thus, we deduced that the (\ref{eq:sys1}) is asymptotically stable for   $0 \leq \tau < \tau_*$ and unstable for  $ \tau > \tau_*$, and undergoes a Hopf bifurcation at the   positive equilibrium  $E_{*}(x_{*}, y_{*})$ for $\tau = \tau_*(n)$ for $n = 0, 1, 2, \ldots$

	\subsection{Bifurcation Diagram}
	We vary the time delay $\tau$ with	 $\alpha = 0.5$ and $\beta = 0.8.$ When $\tau \geq 30.8017,$ the fixed point loses its stability through Hopf bifurcation as discussed in the previous section.  The system  (\ref{eq:sys1}) is run for a time span of [0 20,000] and we consider only the last one-fourth part to eliminate the possible transient responses. The bifurcation diagram is shown in Figure \ref{fig:bif_xyvstau}.
	
	\begin{figure}[H]
		\centering
		\begin{subfigure}[b]{0.45\textwidth}
			\centering
			\includegraphics[width=1\textwidth]{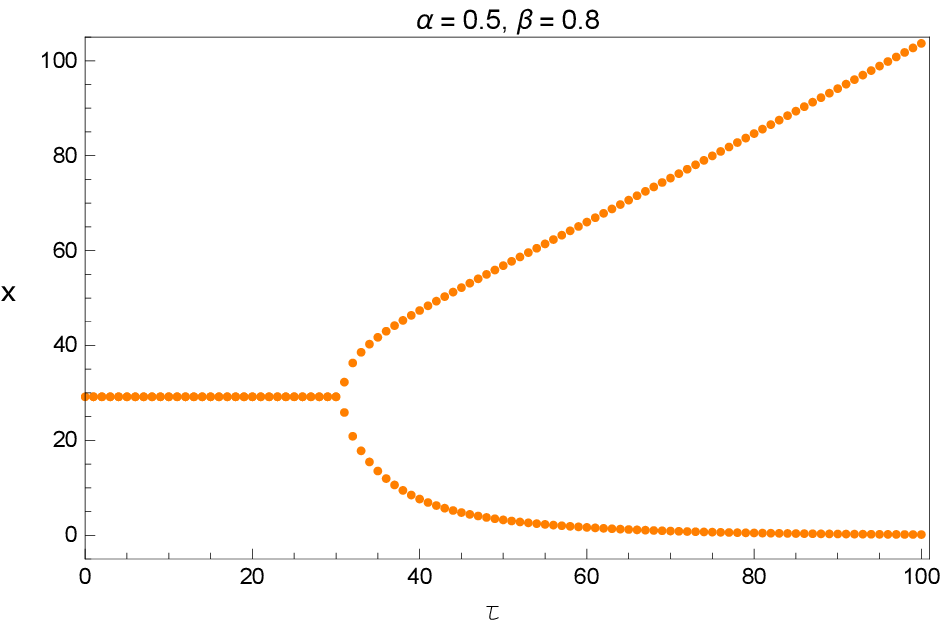}		
			\label{fig:xvstau_alpha0p5beta0p8}
		\end{subfigure}
		\hfill
		\begin{subfigure}[b]{0.45\textwidth}
			\centering
			\includegraphics[width=1\textwidth]{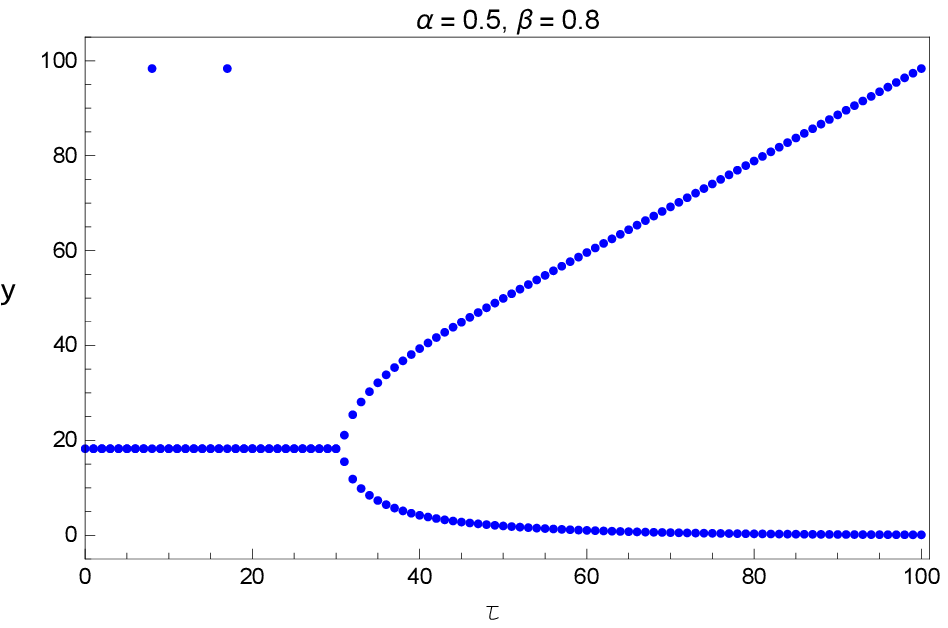}		
			\label{fig:yvstau_alpha0p5beta0p8}
		\end{subfigure}
		\caption{Bifurcation diagram with respect to time delay.}
		\label{fig:bif_xyvstau}
	\end{figure}

	We also plot the bifurcation diagram of the long term values of the system with the parameters $\alpha$ and $\beta$ listed in Table \ref{table:eq_delay}. Since these diagrams are plotted by a different method, this provides us another opportunity to verify the values we found in that table.  This is shown in Figures \ref{fig:bif_xyvstau_a0p3} - \ref{fig:bif_xyvstau_a0p9}.
	
	\begin{figure}[H]
		\centering
		\begin{subfigure}[b]{0.45\textwidth}
			\centering
			\includegraphics[width=1\textwidth]{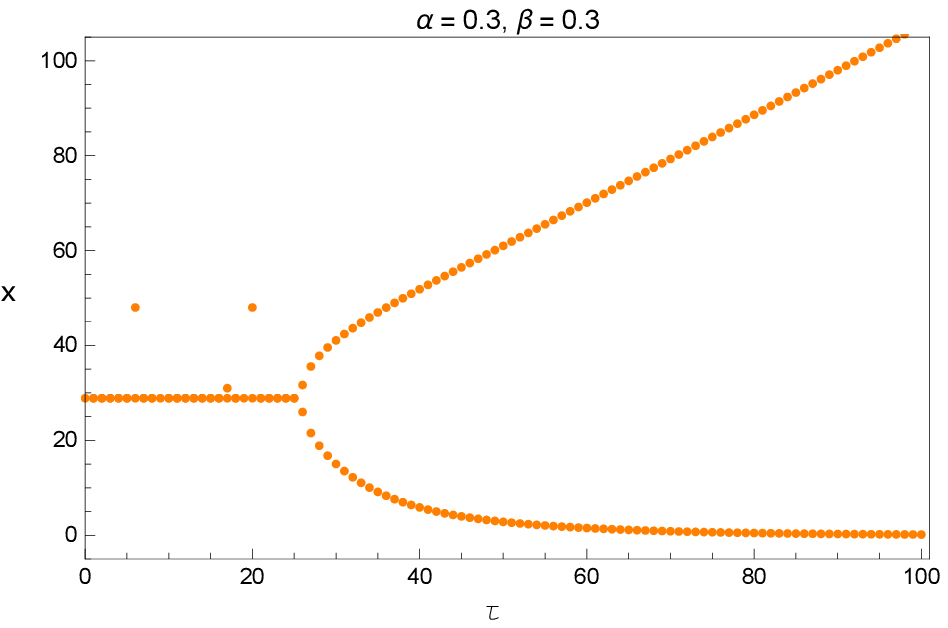}		
			\label{fig:xbif_a0p3b0p3}
		\end{subfigure}
		\hfill
		\begin{subfigure}[b]{0.45\textwidth}
			\centering
			\includegraphics[width=1\textwidth]{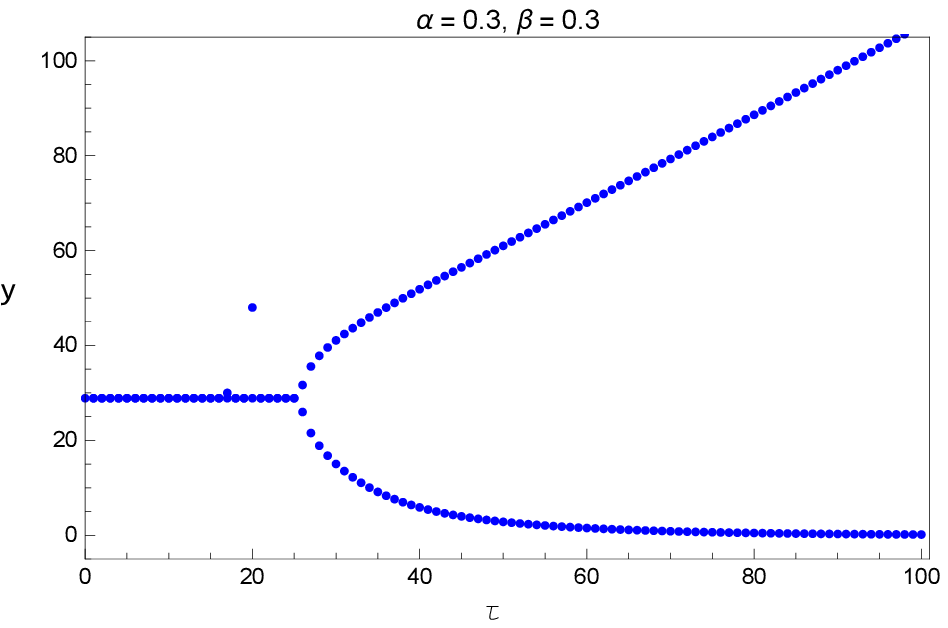}	
			\label{fig:ybif_a0p3b0p3}
		\end{subfigure}
		\hfill
		\begin{subfigure}[b]{0.45\textwidth}
			\centering
			\includegraphics[width=1\textwidth]{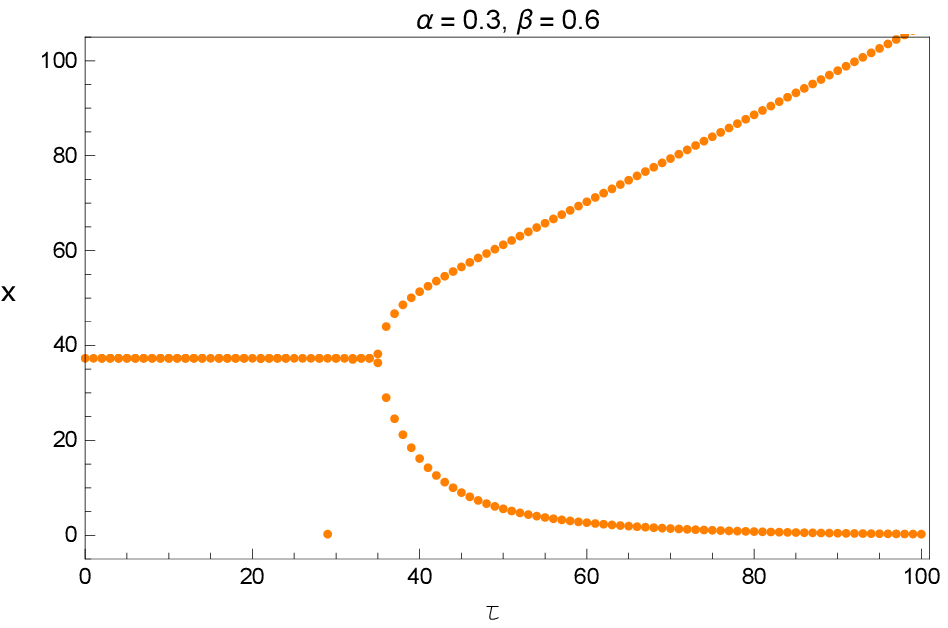}		
			\label{fig:xbif_a0p3b0p6}
		\end{subfigure}
		\hfill
		\begin{subfigure}[b]{0.45\textwidth}
			\centering
			\includegraphics[width=1\textwidth]{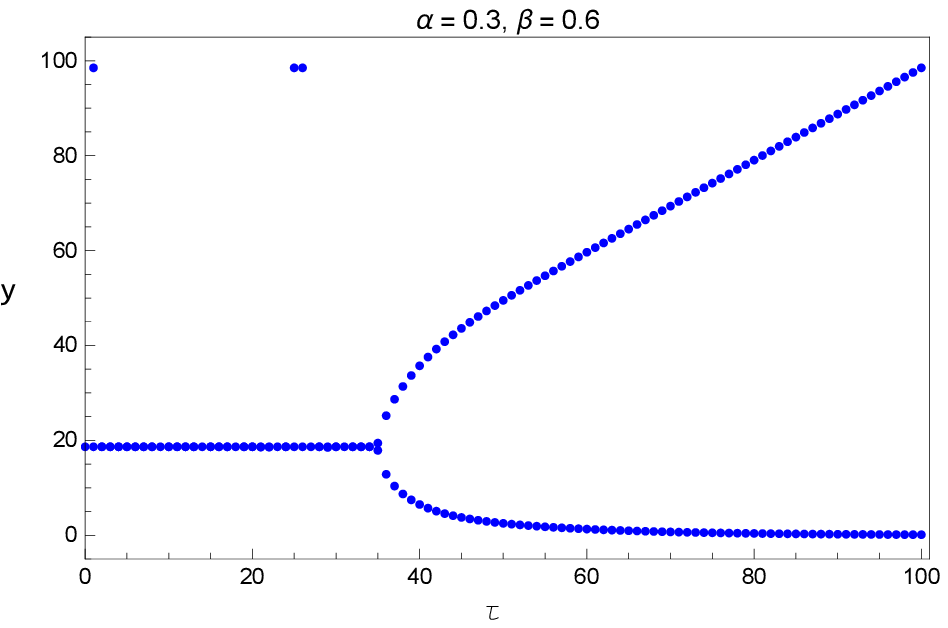}	
			\label{fig:ybif_a0p3b0p6}
		\end{subfigure}
		\hfill
		\begin{subfigure}[b]{0.45\textwidth}
			\centering
			\includegraphics[width=1\textwidth]{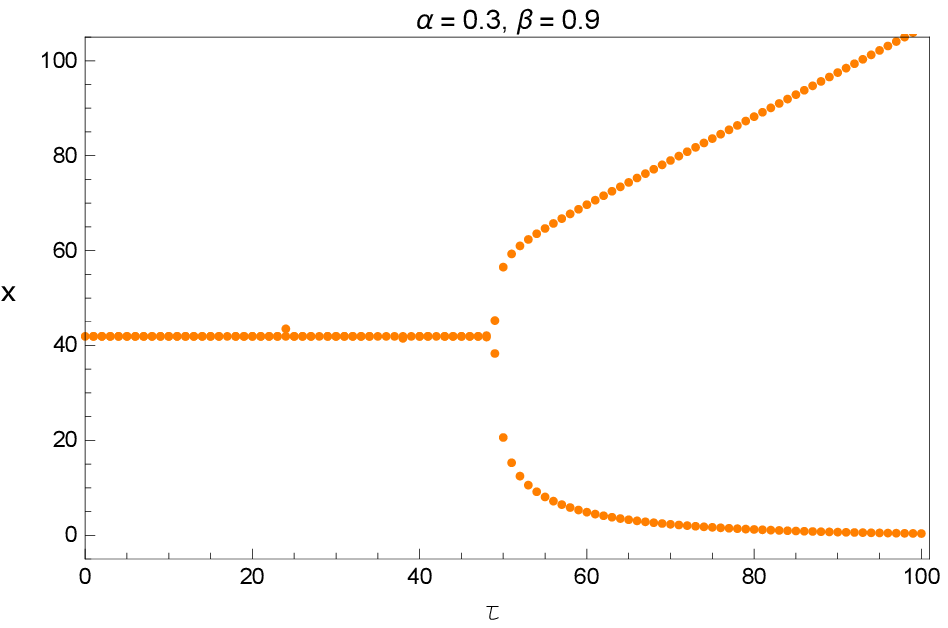}		
			\label{fig:xbif_a0p3b0p9}
		\end{subfigure}
		\hfill
		\begin{subfigure}[b]{0.45\textwidth}
			\centering
			\includegraphics[width=1\textwidth]{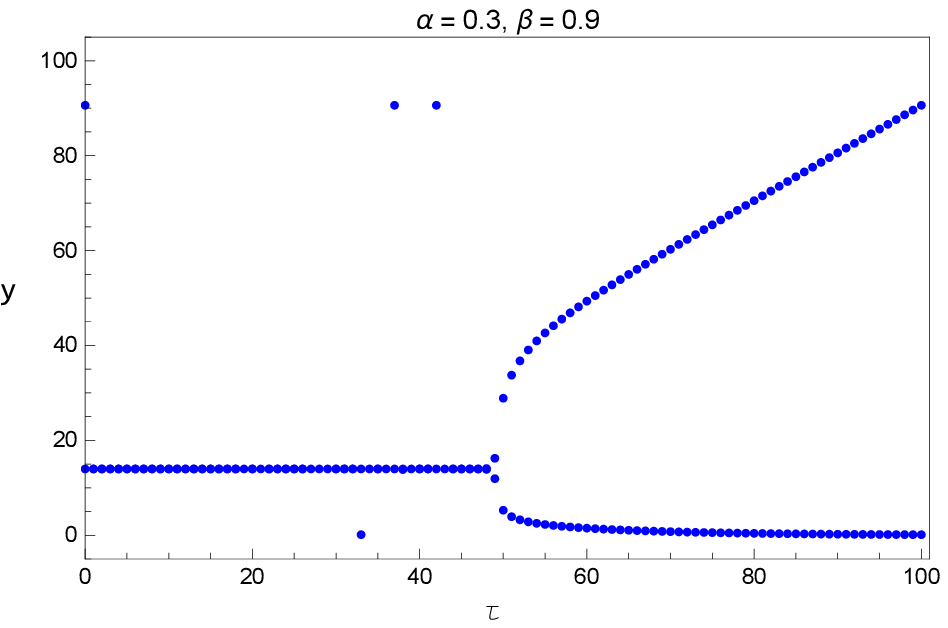}	
			\label{fig:ybif_a0p3b0p9}
		\end{subfigure}
		\caption{Bifurcation diagram with  $\alpha = 0.3 $ and various $\beta. $}
		\label{fig:bif_xyvstau_a0p3}
	\end{figure}

	\begin{figure}[H]
		\centering
		\begin{subfigure}[b]{0.45\textwidth}
			\centering
			\includegraphics[width=1\textwidth]{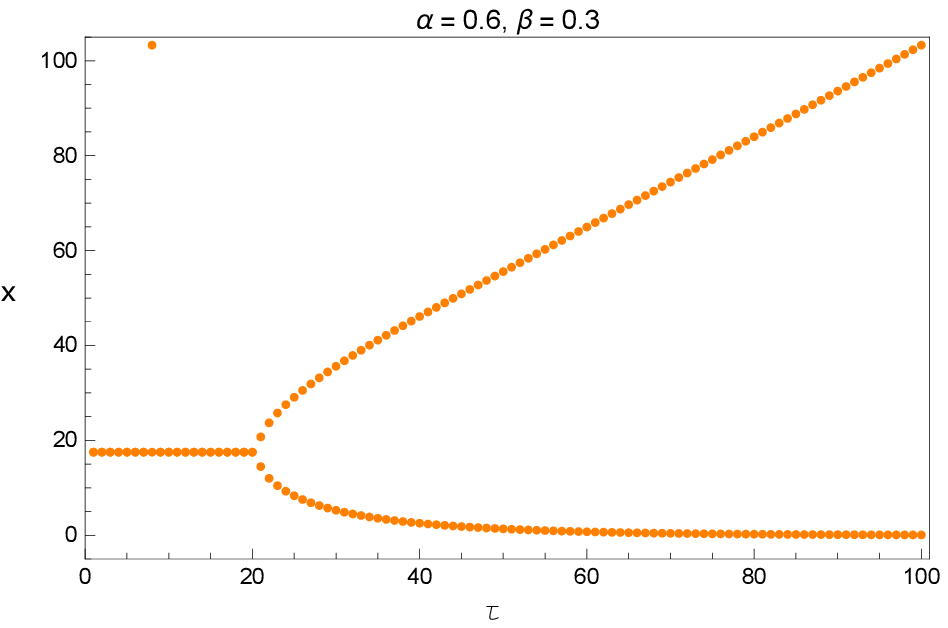}		
			\label{fig:xbif_a0p6b0p3}
		\end{subfigure}
		\hfill
		\begin{subfigure}[b]{0.45\textwidth}
			\centering
			\includegraphics[width=1\textwidth]{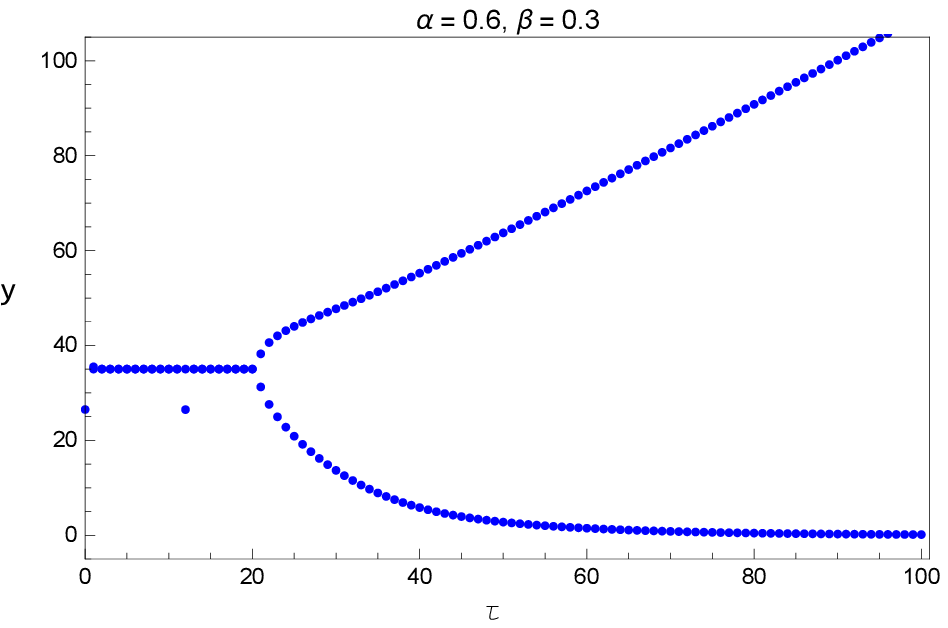}	
			\label{fig:ybif_a0p6b0p3}
		\end{subfigure}
		\hfill
		\begin{subfigure}[b]{0.45\textwidth}
			\centering
			\includegraphics[width=1\textwidth]{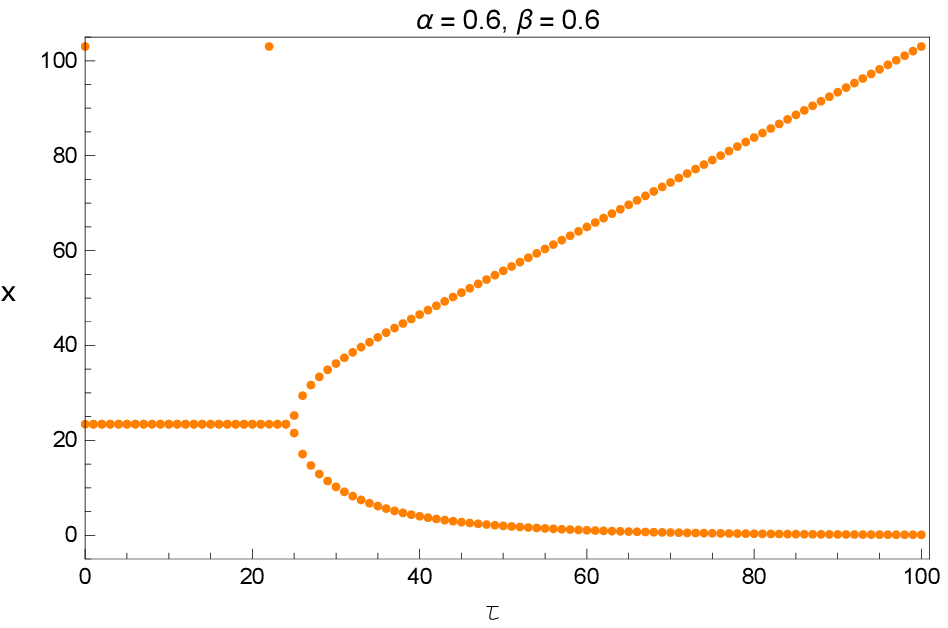}		
			\label{fig:xbif_a0p6b0p6}
		\end{subfigure}
		\hfill
		\begin{subfigure}[b]{0.45\textwidth}
			\centering
			\includegraphics[width=1\textwidth]{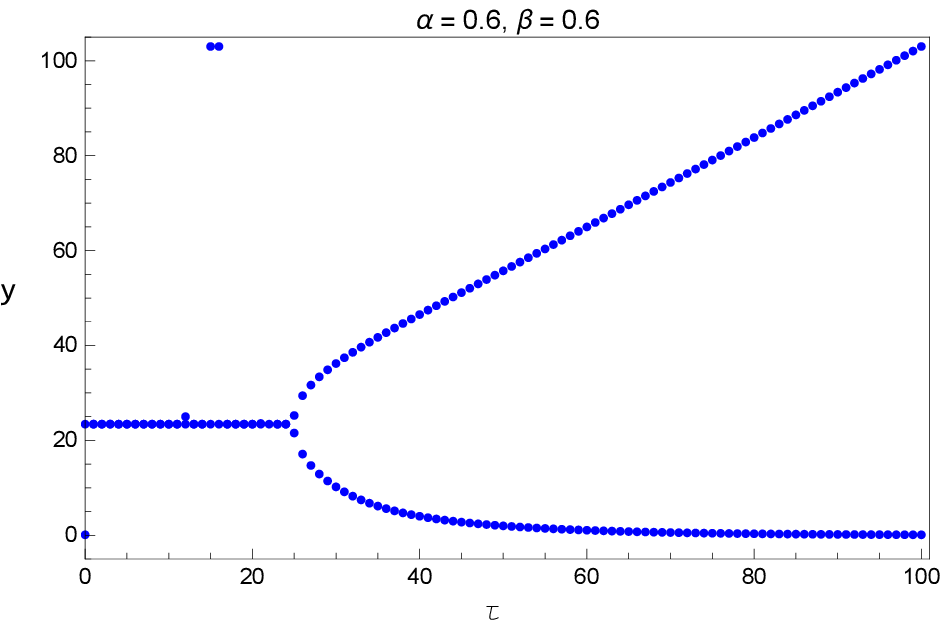}	
			\label{fig:ybif_a0p6b0p6}
		\end{subfigure}
		\hfill
		\begin{subfigure}[b]{0.45\textwidth}
			\centering
			\includegraphics[width=1\textwidth]{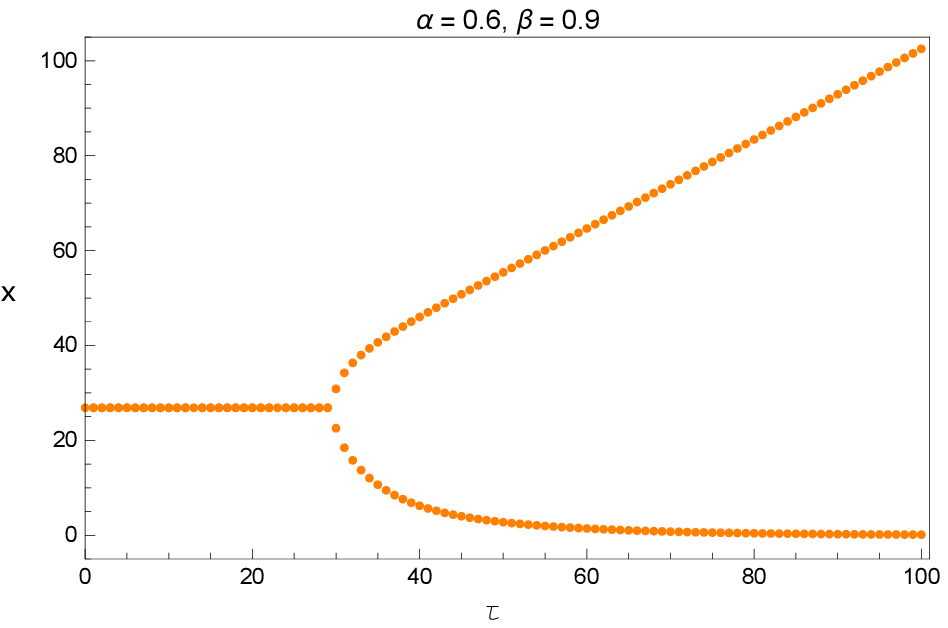}		
			\label{fig:xbif_a0p6b0p9}
		\end{subfigure}
		\hfill
		\begin{subfigure}[b]{0.45\textwidth}
			\centering
			\includegraphics[width=1\textwidth]{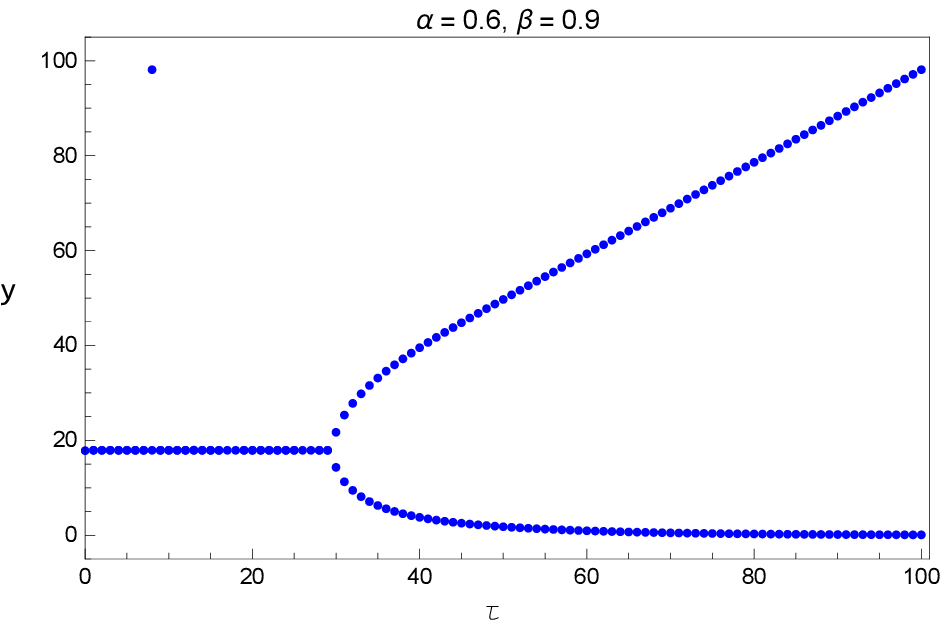}	
			\label{fig:ybif_a0p6b0p9}
		\end{subfigure}
		\caption{Bifurcation diagram with  $\alpha = 0.6 $ and various $\beta. $}
		\label{fig:bif_xyvstau_a0p6}
	\end{figure}
	
	\begin{figure}[H]
		\centering
		\begin{subfigure}[b]{0.45\textwidth}
			\centering
			\includegraphics[width=1\textwidth]{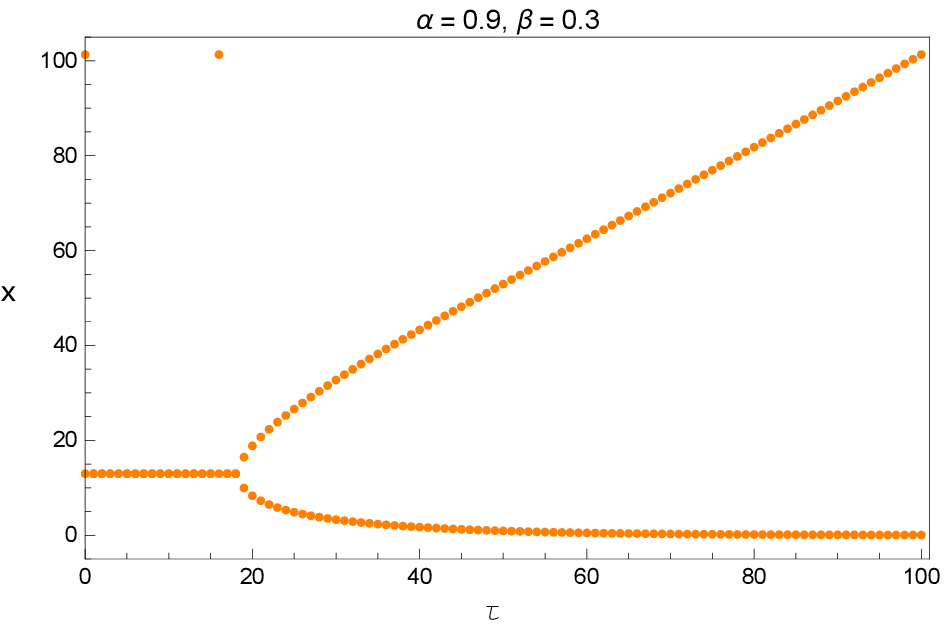}		
			\label{fig:xbif_a0p9b0p3}
		\end{subfigure}
		\hfill
		\begin{subfigure}[b]{0.45\textwidth}
			\centering
			\includegraphics[width=1\textwidth]{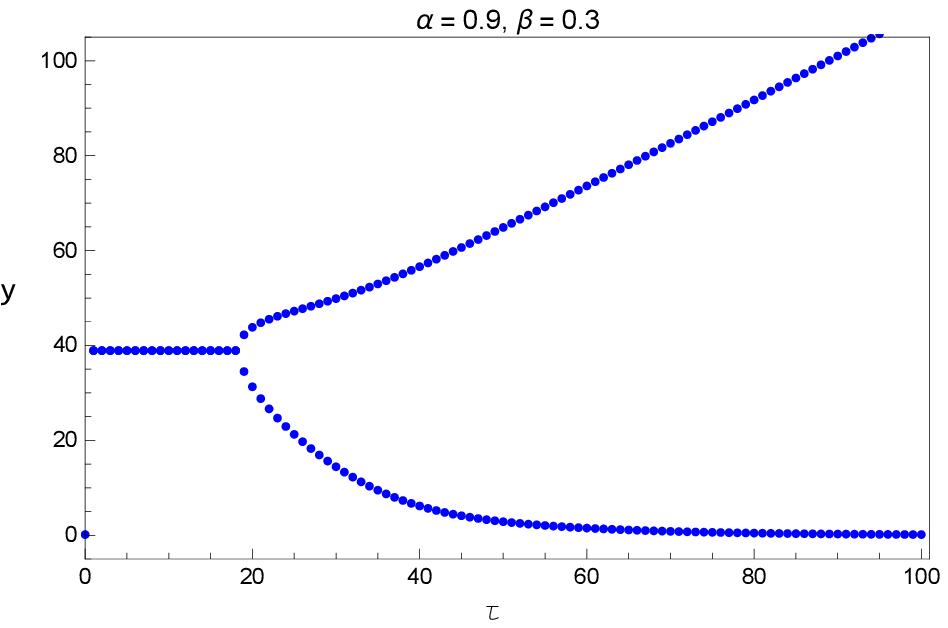}	
			\label{fig:ybif_a0p9b0p3}
		\end{subfigure}
		\hfill
		\begin{subfigure}[b]{0.45\textwidth}
			\centering
			\includegraphics[width=1\textwidth]{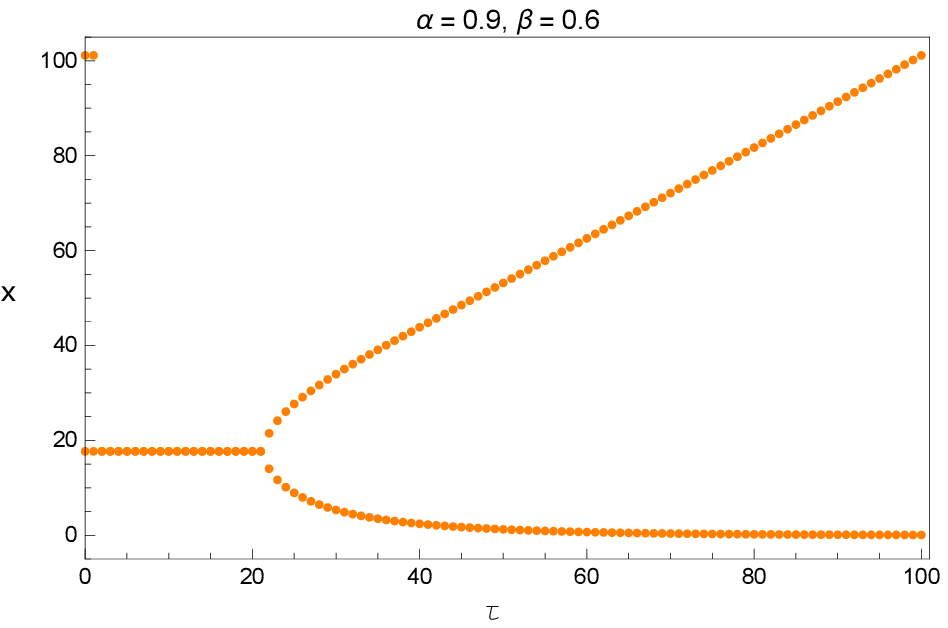}		
			\label{fig:xbif_a0p9b0p6}
		\end{subfigure}
		\hfill
		\begin{subfigure}[b]{0.45\textwidth}
			\centering
			\includegraphics[width=1\textwidth]{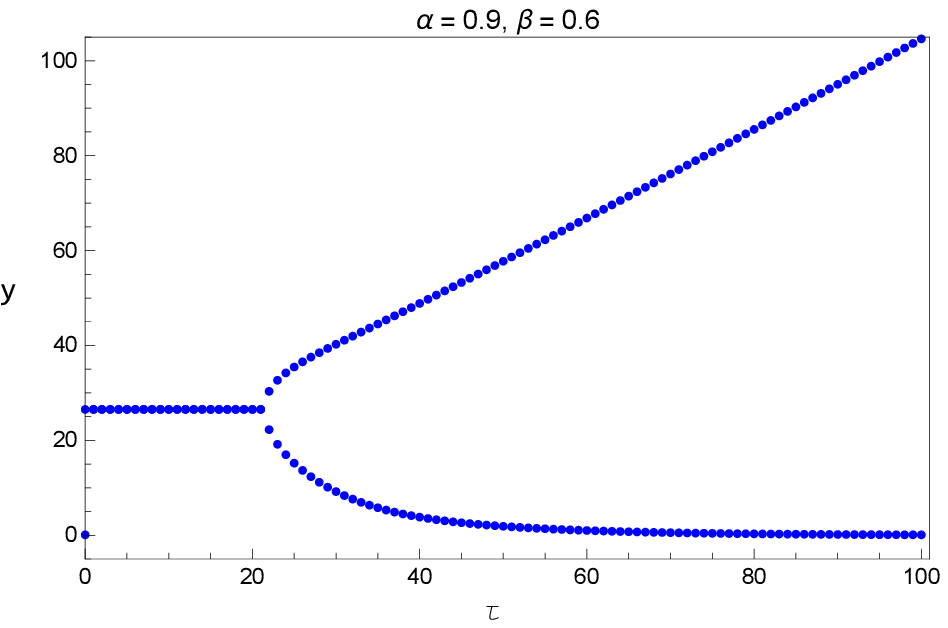}	
			\label{fig:ybif_a0p9b0p6}
		\end{subfigure}
		\hfill
		\begin{subfigure}[b]{0.45\textwidth}
			\centering
			\includegraphics[width=1\textwidth]{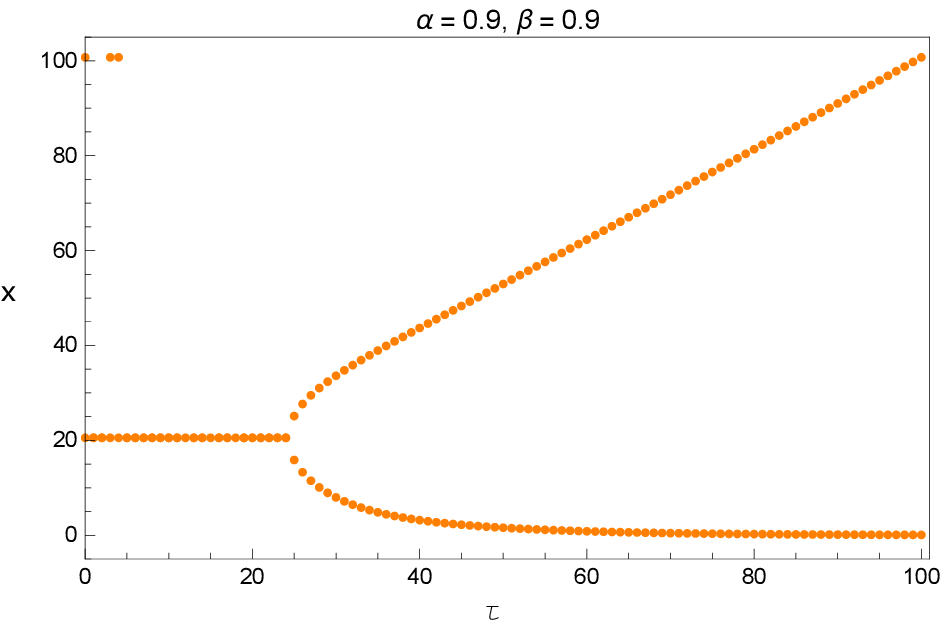}		
			\label{fig:xbif_a0p9b0p9}
		\end{subfigure}
		\hfill
		\begin{subfigure}[b]{0.45\textwidth}
			\centering
			\includegraphics[width=1\textwidth]{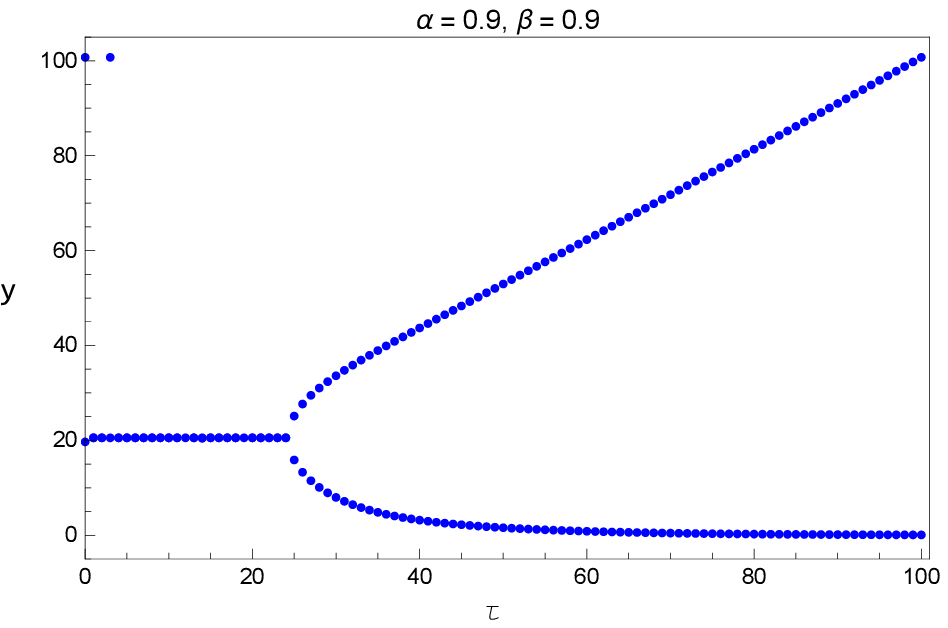}	
			\label{fig:ybif_a0p9b0p9}
		\end{subfigure}
		\caption{Bifurcation diagram with  $\alpha = 0.9 $ and various $\beta. $}
		\label{fig:bif_xyvstau_a0p9}
	\end{figure}
	
	\section{Direction and Stability of the Hopf Bifurcation}	
	In this section, we apply the normal form theory and the center manifold theorem of Hassard et al \cite{hassard1981theory} to get some properties of the Hopf bifurcation. In order to determine the direction and the stability of the Hopf bifurcation, we consider the following system whose equilibrium is shifted to the origin. Let $\tau = \tau_* + \mu$, then $\mu = 0 $ is the Hopf bifurcation value of system (\ref{eq:sys1}) at the positive equilibrium $E(x_*, y_*)$ and  $\pm i \omega_*$ is the corresponding purely imaginary roots of the characteristic equation. 
	
	For the convenience of discussion, let $x_1 = x - x_*$, $x_2 = y - y_*$. The system (\ref{eq:sys1}) can be regarded as FDE in  $C = C([-1,0], \mathbb{R}^2)$ as
	
	\begin{equation} \label{eq:FDE1}
		\dot{x} = L_{\mu}(x_t) + f(\mu, x_t)	
	\end{equation}
	where $x(t) = (x_1, x_2)^T \in \mathbb{R}^2 $, and $L_{\mu} : C \to \mathbb{R}$, $f: \mathbb{R} \times C \to \mathbb{R}$ are given respectively by
	
	\begin{equation}
		\begin{split}
			L_{\mu}(\phi) &= (\tau + \mu) 	
			\left(
			\begin{array}{cc}
				\frac{1}{25} (-7) \beta  e^{\frac{\gamma }{20}-5} \gamma  & -\frac{7 \beta ^2 e^{\frac{\gamma }{20}-5} \gamma ^2}{1000 \alpha } \\
				\frac{1}{50} (-7) \beta  e^{\frac{\gamma }{20}-5} \gamma  & -\frac{7 \beta ^2 e^{\frac{\gamma }{20}-5} \gamma  (\gamma +20)}{1000 \alpha } \\
			\end{array}
			\right)
			\begin{pmatrix}
				\phi_1(0) \\
				\phi_2(0) \\
			\end{pmatrix}\\	&+(\tau + \mu) 	
			\left(
			\begin{array}{cc}
				\frac{1}{50} (-7) \alpha  e^{\frac{\gamma }{20}-5} x_* & -\frac{7 \beta  e^{\frac{\gamma }{20}-5} \gamma  x_*}{1000} \\
				\frac{1}{50} (-7) \beta  e^{\frac{\gamma }{20}-5} y_* & -\frac{7 \beta ^2 e^{\frac{\gamma }{20}-5} \gamma  y_*}{1000 \alpha } \\
			\end{array}
			\right)
			\begin{pmatrix}
				\phi_1(-1) \\
				\phi_2(-1) \\
			\end{pmatrix}
		\end{split}
	\end{equation}
	and 
	\begin{equation}  \label{eq:f_mu}
		f(\mu, \phi) = (\tau + \mu) 
		\begin{pmatrix}
			-\frac{7 \alpha }{50 e^5} \phi_1^2(0) \\
			-\frac{7 \beta  }{50 e^5} \phi_1(-1) \phi_2(-1)   \\
		\end{pmatrix}		
	\end{equation}
	
	According to the Riesz representation theorem, there exists a matrix function $\eta(\vartheta, \mu),$ $\vartheta \in [-1,0]$ with bounded variation components such that
	\begin{equation} \label{eq:L_mu}
		L_{\mu}(\phi)	= 	\int _{-1}^0d\eta (\vartheta ,0)\phi(\vartheta) \quad \text{for} \quad \phi \in C
	\end{equation}
	Actually we can take
	\begin{equation} \label{eq:eta1}
		\begin{split}
			\eta(\vartheta, \mu) &= (\tau + \mu) 	
			\left(
			\begin{array}{cc}
				\frac{1}{25} (-7) \beta  e^{\frac{\gamma }{20}-5} \gamma  & -\frac{7 \beta ^2 e^{\frac{\gamma }{20}-5} \gamma ^2}{1000 \alpha } \\
				\frac{1}{50} (-7) \beta  e^{\frac{\gamma }{20}-5} \gamma  & -\frac{7 \beta ^2 e^{\frac{\gamma }{20}-5} \gamma  (\gamma +20)}{1000 \alpha } \\
			\end{array}
			\right)
			\delta(\vartheta)\\	&-(\tau + \mu) 	
			\left(
			\begin{array}{cc}
				\frac{1}{50} (-7) \alpha  e^{\frac{\gamma }{20}-5} x_* & -\frac{7 \beta  e^{\frac{\gamma }{20}-5} \gamma  x_*}{1000} \\
				\frac{1}{50} (-7) \beta  e^{\frac{\gamma }{20}-5} y_* & -\frac{7 \beta ^2 e^{\frac{\gamma }{20}-5} \gamma  y_*}{1000 \alpha } \\
			\end{array}
			\right)
			\delta(\vartheta + 1)
		\end{split}
	\end{equation}	
	where $\delta$ is the Dirac Delta function. For  $\phi \in C^1([-1,0], \mathbb{R}^2)$, define
	\begin{equation}
		A(\mu)\phi = \left\{\begin{array}{ll}\frac{d\phi(\vartheta)}{d\vartheta}, & \vartheta \in \lbrack-1,0), \\ 
			
			\int_1^0 d\eta(\mu,s)\phi(s) & \vartheta = 0,\end{array}\right.
	\end{equation}
	and
	\begin{equation}
		R(\mu)\phi = \left\{\begin{array}{ll} 0, & \vartheta \in \lbrack-1,0), \\
			f(\mu,\phi) & \vartheta = 0,\end{array}\right.
	\end{equation}
	The system 	(\ref{eq:FDE1}) can be represented as
	\begin{equation} \label{eq:FDE2}
		\dot{x} = A_{\mu}(x_t) + R(\mu) x_t	
	\end{equation}
	where $x_t(\vartheta) = x(t + \vartheta)$ for $\vartheta \in [-1,0].$
	
	For $\psi \in C^1([0,1], (\mathbb{R}^2)^*),$ define 	
	\begin{equation}
		A^*\psi(s) = \left\{\begin{array}{ll}-\frac{d \psi(s)}{d s}, & s \in (0,1\rbrack, \\ 
			
			\int_1^0 d\eta^T(t,0)\psi(-t) & s = 0,\end{array}\right.
	\end{equation}	
	Furthermore, for  $\phi \in C^1([-1,0], (\mathbb{R}^2)^*),$ and  $\psi \in C^1([0,1], (\mathbb{R}^2)^*),$ we give the bilinear inner product as
	\begin{equation} \label{eq:bip1}
		\left< \psi(s), \phi(\vartheta) \right> = \bar{\psi}(0)\phi(0) - \int_{-1}^{0}\int_{\xi =0}^{\vartheta}\bar{\psi}(\xi - \vartheta)d\eta(\vartheta)\phi(\xi)d\xi	
	\end{equation}
	where $\eta(\vartheta) = \eta(\vartheta,0).$ Then $A(0)$ and $A^*$ are adjoint operators.  From previous section, we have that $\pm i \omega_* \tau_* $ are eigenvalues of $A(0).$ It is evident that they are also the eigenvalues of the linear operator $A^*.$ We need to compute the eigenvectors of $A(0)$ and $A^*$ corresponding to $i \omega_* \tau_*$ and $- i \omega_* \tau_*.$

	Assume that
	\begin{equation}
		q(\vartheta) = 
		\begin{pmatrix}
			1   \\
			c   \\
		\end{pmatrix} e^{i \omega_* \tau_* \vartheta}
	\end{equation}
	is the eigenvector of $A(0)$ corresponding to $i\omega_* \tau_*$
	and when $\vartheta = 0,$ we have
	\begin{equation}
		q(0) = 
		\begin{pmatrix}
			1   \\
			c   \\
		\end{pmatrix}
	\end{equation}	
	Then
	\begin{equation}
		A(0) q(\vartheta)  = i \omega_* \tau_* q(\vartheta)			
	\end{equation}
	From the definition of $A(0)$ and (\ref{eq:FDE1}), (\ref{eq:L_mu}) and (\ref{eq:eta1}), we get
	\begin{equation}
		A(0) q(0) = 	\begin{pmatrix}
			i \omega_* \tau_*   \\
			i \omega_* \tau_*	c   \\
		\end{pmatrix}	
	\end{equation}
	or,
	\begin{equation}
		\tau_* \left(
		\begin{array}{cc}
			\frac{7}{25} \beta  e^{\frac{\gamma }{20}-5} \gamma +\frac{7}{50} \alpha  x e^{\frac{\gamma }{20}-i \tau  \omega -5}+i \omega  & \frac{7 \beta ^2 e^{\frac{\gamma }{20}-5} \gamma ^2}{1000 \alpha }+\frac{7 \beta  \gamma  x e^{\frac{\gamma }{20}-i \tau  \omega -5}}{1000} \\
			\frac{7}{50} \beta  e^{\frac{\gamma }{20}-5} \gamma +\frac{7}{50} \beta  y e^{\frac{\gamma }{20}-i \tau  \omega -5} & \frac{7 \beta ^2 e^{\frac{\gamma }{20}-5} \gamma  (\gamma +20)}{1000 \alpha }+\frac{7 \beta ^2 \gamma  y e^{\frac{\gamma }{20}-i \tau  \omega -5}}{1000 \alpha }+i \omega  \\
		\end{array}
		\right).
	\end{equation}
	\begin{equation*}
		q(0) = \left(
		\begin{array}{c}
			0 \\
			0 \\
		\end{array}
		\right)
	\end{equation*}
	Thus we obtain,
	\begin{equation}
		c = -\frac{140 \alpha  \beta  e^{\gamma /20} \left(y+\gamma  e^{i \tau  \omega }\right)}{1000 i \alpha  \omega  e^{5+i \tau  \omega }+7 \beta ^2 \gamma ^2 e^{\frac{\gamma }{20}+i \tau  \omega }+140 \beta ^2 \gamma  e^{\frac{\gamma }{20}+i \tau  \omega }+7 \beta ^2 e^{\gamma /20} \gamma  y}	
	\end{equation}
	
	Similarly, we can get the eigenvector 
	\begin{equation}
		q^*(s) = D 
		\begin{pmatrix}
			1   \\
			c^*   \\
		\end{pmatrix} e^{i \omega_* \tau_* s} 
	\end{equation}
	of $A^*$ corresponding to $-i \omega_* \tau_*,$ 
	where 
	\begin{equation}
		c^* = \frac{50 i \omega  e^{-\frac{\gamma }{20}+i \tau  \omega +5}-7 \left(\alpha  x+2 \beta  \gamma  e^{i \tau  \omega }\right)}{7 \beta  \left(y+\gamma  e^{i \tau  \omega }\right)}	
	\end{equation}
	
	Now we evaluate the value of $D$ such that $\left<q^*(s), q(\vartheta) \right> = 1.$ From the bilinear inner product of (\ref{eq:bip1}), it follows that 
	
	\begin{equation}
		\begin{aligned}
			1 &=	\left<q^*(s), q(\vartheta) \right>\\
			&= \bar{q}^*(0)q(0) - \int_{-1}^{0}\int_{\xi =0}^{\vartheta}\bar{q}^*(\xi - \vartheta)d\eta(\vartheta)q(\xi)d\xi \\
			&= \bar{D}(1, \bar{c}^*)(1, c)^T - \int_{-1}^{0}\int_{\xi =0}^{\vartheta}\bar{D}(1, \bar{c}^*) e^{-i \omega_* \tau_* (\xi - \vartheta)} d\eta(\vartheta) (1, c)^T  e^{i \omega_* \tau_* \xi}   d\xi\\
			&= \bar{D}(1, \bar{c}^*)(1, c)^T - \bar{D}(1, \bar{c}^*) \int_{-1}^{0}\int_{\xi =0}^{\vartheta}  d\eta(\vartheta) (1, c)^T  e^{i \omega_* \tau_* \vartheta}   d\xi\\
			&= \bar{D}(1, \bar{c}^*)(1, c)^T - \bar{D}(1, \bar{c}^*)  \int_{-1}^{0} \vartheta e^{i \omega_* \tau_* \vartheta}  d\eta(\vartheta) (1, c)^T     \\
			&= \bar{D}\left\{ 1 + c \bar{c}^* +\bar{c}^* \tau_* \left(  -\frac{7 e^{\frac{\gamma }{20}-5} (20 \alpha +\beta  c \gamma ) \left(\beta  c^* y_*+\alpha  x_*\right)}{1000 \alpha } \right)   e^{-i \omega_* \tau_*} \right\}
		\end{aligned}
	\end{equation}
	
	Thus, we have
	\begin{equation}
		\bar{D} = \frac{1}{ 1 + c \bar{c}^* +\bar{c}^* \tau_* \left(  -\frac{7 e^{\frac{\gamma }{20}-5} (20 \alpha +\beta  c \gamma ) \left(\beta  c^* y_*+\alpha  x_*\right)}{1000 \alpha } \right)   e^{-i \omega_* \tau_*} }	
	\end{equation}
	
	
	In addition, from 	$ \left< \psi, A \phi \right> =  \left< A^*\psi,  \phi \right> $ and $ A \bar{q}(\vartheta) = -i \omega_* \tau_* \bar{q}(\vartheta),	$ we can obtain
	\begin{equation}
		\begin{aligned}
			-i \omega_* \tau_* \left<q^*, \bar{q}\right> & = \left<q^*, A\bar{q}\right> \\
			&= 		\left<A^*q^*, \bar{q}\right> \\
			&= 		\left<A^*q^*, \bar{q}\right> \\
			&= 		i \omega_* \tau_* \left<q^*, \bar{q}\right>\\
		\end{aligned}
	\end{equation}
	
	Hence  $\left<q^*(s), \bar{q}(\vartheta)\right> = 0.$
	
	In rest of the section, we calculate the coordinates to describe the center manifold $C_0$ at $\mu = 0$ by the method used in Hassard paper. Let $x_t$ be the solution of equation (\ref{eq:FDE2}) and define $z(t) = \left<q^*,x_t\right>;$ then
	\begin{equation}
		\begin{aligned}
			\dot{z}(t) &= \left<q^*, \dot{x}_t \right>	= 	\left<q^*, A(0)\dot{x}_t + R(0)x_t\right>\\
			&= 	\left<q^*, A(0)\dot{x}_t \right> + 	\left<q^*, R(0)x_t\right> \\
			&= 	\left<A^*(0)q^*, \dot{x}_t \right> + 	\bar{q}^*(0)f_0(z,\bar{z}) \\
			&= 	i \omega_* \tau_* z + 	g(z,\bar{z}) \\
		\end{aligned}
	\end{equation}
	where
	\begin{equation}  \label{eq:qz}
		g(z, \bar{z}) = \bar{q}^*(0)f_0(z,\bar{z})	= g_{20}\frac{z^2}{2} + g_{11}z \bar{z} + g_{02} \frac{\bar{z}^2}{2} 
		+  g_{21} \frac{z^2\bar{z}}{2} + \ldots
	\end{equation}
	Let
	\begin{equation}
		\begin{aligned} \label{eq:W1}
			W(t,\vartheta) &= x_t(\vartheta) -z(t)q(\vartheta) - \bar{z}(t) \bar{g}(\vartheta)	\\
			&= 	x_t(\vartheta) - 2 \text{Re} \left\{z(t)q(\vartheta)\right\} \\
		\end{aligned}
	\end{equation}
	On the center manifold $C_0,$ we have
	\begin{equation}
		W(t,\vartheta) = W(z(t),\bar{z}(t), \vartheta),	
	\end{equation}
	where
	\begin{equation} \label{eq:W2}
		W(z, \bar{z}, \vartheta)  =  W_{20}(\vartheta)\frac{z^2}{2} + W_{11} (\vartheta)z \bar{z} + W_{02} (\vartheta) \frac{\bar{z}^2}{2} +  W_{30} (\vartheta) \frac{z^3 }{6} + \ldots	
	\end{equation}
	$z$ and $\bar{z}$ are local coordinates for center manifold $C_0$ in the direction of $q^*$ and $\bar{q}^*.$ Note that $W$ is real if $x_t$ is real. We only consider real solutions. 
	
	It follows from (\ref{eq:W1}) and (\ref{eq:W2}) that
	\begin{equation}
		\begin{aligned}
			x_t(\vartheta) &= (	x_{1t}(\vartheta), 	x_{2t}(\vartheta))^T\\
			&=  W(t, \vartheta) + 2 \text{Re}  \left\{z(t)q(\vartheta)\right\}\\
			&= W(t, \vartheta) +   z(t)q(\vartheta) +  \bar{z}(t)\bar{q}(\vartheta)\\
			&= W_{20}(\vartheta)\frac{z^2}{2} + W_{11} (\vartheta)z \bar{z} + W_{02} (\vartheta) \frac{\bar{z}^2}{2}\\
			& \quad + (1, c)^T e^{i \omega_* \tau_* \vartheta}z + (1, \bar{c})^T e^{-i \omega_* \tau_* \vartheta}\bar{z} + \ldots\\ 
		\end{aligned} 	
	\end{equation}

	From (\ref{eq:f_mu}), it follows that
	\begin{equation}
		\begin{aligned}
			q(z,\bar{z}) &= \bar{q}^*(0)f(0,x_t) \\
			&= \tau_* \bar{D}(1,\bar{c}^*) 		\begin{pmatrix}
				-\frac{7 \alpha }{50 e^5} x_{1t}^2(0) \\
				-\frac{7 \beta  }{50 e^5} x_{1t}(-1) x_{2t}(-1)   \\	
			\end{pmatrix}\\
		\end{aligned}
	\end{equation}
	
	and with
	\begin{equation}
		\begin{aligned}
			x_{1t}(0) &= z + \bar{z} + W_{20}^{1}(0)\frac{z^2}{2}	 + W_{11}^{1}(0)z \bar{z} + W_{02}^{1} (0) \frac{\bar{z}^2}{2} + \dots\\
			x_{1t}(-1) &= z e^{-i \tau _* \omega _*} + \bar{z} e^{i \tau _* \omega _*} + W_{20}^{1}(0)\frac{z^2}{2}	 + W_{11}^{1}(0)z \bar{z} + W_{02}^{1} (0) \frac{\bar{z}^2}{2} + \dots\\
			x_{2t}(-1) &= z c e^{-i \tau _* \omega _*} + \bar{z} \bar{c} e^{i \tau _* \omega _*} + W_{20}^{2}(0)\frac{z^2}{2}	 + W_{11}^{2}(0)z \bar{z} + W_{02}^{2} (0) \frac{\bar{z}^2}{2} + \dots\\
		\end{aligned}
	\end{equation}
	we get,
	\begin{equation}
		\begin{aligned}
			q(z,\bar{z}) &= 		 \tau_* \bar{D}    \left\lbrace 		-\frac{7 \alpha }{50 e^5} \left(z + \bar{z} + W_{20}^{1}(0)\frac{z^2}{2}	 +  W_{11}^{1}(0)z \bar{z} + W_{02}^{1} (0) \frac{\bar{z}^2}{2} + \dots \right)^2 \right.\\ 
			& \quad  -\frac{7 \beta \bar{c}^* }{50 e^5} \left( z e^{-i \tau _* \omega _*} + \bar{z} e^{i \tau _* \omega _*} + W_{20}^{1}(0)\frac{z^2}{2}	 + W_{11}^{1}(0)z \bar{z} + W_{02}^{1} (0) \frac{\bar{z}^2}{2} + \dots \right)\\
			& \quad \times \left. \left( z c e^{-i \tau _* \omega _*} + \bar{z} \bar{c} e^{i \tau _* \omega _*} + W_{20}^{2}(0)\frac{z^2}{2}	 + W_{11}^{2}(0)z \bar{z} + W_{02}^{2} (0) \frac{\bar{z}^2}{2} + \dots \right)   \right\rbrace \\
		\end{aligned}
	\end{equation}

	Comparing the coefficients with (\ref{eq:qz}), we have
	\begin{equation} \label{eq:gcoeff}
		\begin{aligned}
			g_{20} &= 2 \tau_* \bar{D} \left\{ -\frac{7 \alpha }{50 e^5}-\frac{7}{50} \beta  c \bar{c}^* e^{-5-2 i \tau _* \omega _*} \right\} \\
			g_{11} &= 2 \tau_* \bar{D} \left\{  -\frac{7 \alpha }{25 e^5} -\frac{7 \beta  \bar{c}^*  \text{Re}\{c\}}{50 e^5} \right\} \\
			g_{02} &= 2 \tau_* \bar{D} \left\{ -\frac{7 \alpha }{50 e^5}-\frac{7}{50} \beta  \bar{c} \bar{c}^* e^{-5+2 i \tau _* \omega _*} \right\}\\
			g_{21} &=  \tau_* \bar{D} \left\lbrace  -\frac{7 \alpha  \left(4 W_{11}^1(0)+ 2W_{20}^1(0)\right)}{50 e^5} - \frac{7}{50} \beta  \bar{c}^* e^{-5-i \tau _* \omega _*} \times \right.\\
			& \quad    \left(e^{2 i \tau _* \omega _*} \left(W_{20}^1(-1) \bar{c}+W_{20}^2(-1)\right)+2 c W_{11}^1(-1)+2 W_{11}^2(-1)\right) \bigg\}\\
		\end{aligned}
	\end{equation}

	Since we have $W_{20}(\vartheta)$ and $W_{11}(\vartheta)$ in $g_{21},$ we still need to calculate these terms. From (\ref{eq:FDE2}) and (\ref{eq:W1}), we get
	\begin{equation}  \label{eq:Wdot}
		\begin{aligned}
			\dot{W} = \dot{x_t} - \dot{z} - \dot{\bar{z}}\bar{q} &= 	 \left\{\begin{array}{ll} A W - 2\text{Re}\left\{\bar{q}^*(0) f_{0q}(\vartheta) \right\}, & \vartheta \in \lbrack-1,0), \\
				A W - 2\text{Re}\left\{\bar{q}^*(0) f_{0q}(\vartheta) \right\} + f_{0},  & \vartheta = 0,\end{array}\right.	\\
			& \stackrel{\text{def}}{=} A W + H(z, \bar{z}, \vartheta),
		\end{aligned}
	\end{equation}
	
	where
	\begin{equation} \label{eq:Hz1}
		H(z, \bar{z}, \vartheta) =  H_{20}(\vartheta)\frac{z^2}{2} + H_{11} (\vartheta)z \bar{z} + H_{02} (\vartheta) \frac{\bar{z}^2}{2} + \ldots
	\end{equation}
	Thus, we have
	\begin{equation}
		AW(t, \vartheta) - \dot{W}	=  -H(z, \bar{z}, \vartheta)= - H_{20}(\vartheta)\frac{z^2}{2} - H_{11} (\vartheta)z \bar{z} - H_{02} (\vartheta) \frac{\bar{z}^2}{2} + \ldots
	\end{equation}
	
	From (\ref{eq:W2}), we obtain

	\begin{equation}
		\begin{aligned}
			AW(t,\vartheta) &=  A	W_{20}(\vartheta)\frac{z^2}{2} + A W_{11} (\vartheta)z \bar{z} + A W_{02} (\vartheta) \frac{\bar{z}^2}{2} +  AW_{30} (\vartheta) \frac{z^3 }{6} + \ldots\\
			\dot{W} &= W_{z}\dot{z} + W_{\bar{z}}\dot{\bar{z}} = W_{20}(\vartheta)z\dot{z} + W_{11}(\vartheta)(\dot{z}\bar{z} + z\dot{\bar{z}}) + \ldots \\
			&= 2 i \omega_* \tau_* 	W_{20}(\vartheta)\frac{z^2}{2} + \ldots 
		\end{aligned}	
	\end{equation}

	Thus we have,
	\begin{equation}  \label{eq:A1}
		\begin{aligned}
			(A - 2i\omega_* \tau_*)W_{20}(\vartheta) &= - H_{20}(\vartheta),\\
			AW_{11}(\vartheta) &= - H_{11}(\vartheta) 		\\
		\end{aligned}
	\end{equation}

	For $ \vartheta \in [-1,0), $
	
	\begin{equation}
		H(z, \bar{z}, \vartheta) = 	- \bar{q}^*(0) f_{0q}(\vartheta) 	- \bar{q}^*(0) f_{0}(\vartheta) = 	- g(z,\bar{z})q(\vartheta) 
		-     	 \bar{g}(z,\bar{z})\bar{q}(\vartheta)
	\end{equation}

	Comparing the coefficients with (\ref{eq:Hz1}), we get
	\begin{equation}  \label{eq:Hs}
		\begin{aligned}
			H_{20}(\vartheta) &= -g_{20}q(\vartheta) - \bar{g}_{02}\bar{q}(\vartheta)\\
			H_{11}(\vartheta) &= -g_{11}q(\vartheta) - \bar{g}_{11}\bar{q}(\vartheta)\\
		\end{aligned}
	\end{equation}

	From (\ref{eq:A1}) and (\ref{eq:Hs}) and the definition of $ A, $ it follows that
	\begin{equation}
		\dot{W}_{20}(\vartheta) = 2i \omega_* \tau_* W_{20} (\vartheta) + g_{20}q(\vartheta) + \bar{g}_{02}\bar{q}(\vartheta)	
	\end{equation}
	
	Notice that 
	\begin{equation}
		q(\vartheta) = 
		\begin{pmatrix}
			1   \\
			c   \\
		\end{pmatrix} e^{i \omega_* \tau_* \vartheta},
	\end{equation}
	so
	\begin{equation}  \label{eq:W20}
		W_{20}(\vartheta) = \frac{ig_{20}}{\omega_* \tau_*} q(0)	e^{i \omega_* \tau_* \vartheta} +
		\frac{i\bar{g}_{02}}{3\omega_* \tau_*} \bar{q}(0)	e^{-i \omega_* \tau_* \vartheta} + E_1 	e^{2i \omega_* \tau_* \vartheta}
	\end{equation}
	where $ E_1 = \left(	E_1^{(1)}, 	E_1^{(2)} \right)  \in \mathbb{R}^2 $  is a two-dimensional constant vector.
	
	Similarly from (\ref{eq:A1}) and (\ref{eq:Hs}), we obtain
	
	\begin{equation} \label{eq:W11}
		W_{11}(\vartheta) = - \frac{ig_{11}}{\omega_* \tau_*} q(0)	e^{i \omega_* \tau_* \vartheta} +
		\frac{i\bar{g}_{11}}{\omega_* \tau_*} \bar{q}(0)	e^{-i \omega_* \tau_* \vartheta} + E_2
	\end{equation}
	where $ E_2 = \left(	E_2^{(1)}, 	E_2^{(2)} \right)  \in \mathbb{R}^2 $  is a two-dimensional constant vector.
	
	Next, we need to compute $ E_1 $ and $ E_2. $ From the definition of $ A $ and (\ref{eq:A1}), we obtain
	\begin{equation*}
		\int_{-1}^{0} d\eta(\vartheta) W_{20}(\vartheta) = 2i\omega_* \tau_* W_{20}(0) - H_{20}(0)	
	\end{equation*}
	and 
	\begin{equation} \label{eq:inteta1}
		\int_{-1}^{0} d\eta(\vartheta) W_{11}(\vartheta) =  - H_{11}(0)	
	\end{equation}
	where $ \eta(\vartheta) = \eta(0, \vartheta). $
	
	By (\ref{eq:Wdot}), we have
	
	\begin{equation}  \label{eq:H20}
		H_{20}(0) = -g_{20}q(0) - \bar{g}_{02}\bar{q}(0) + 2\tau_* 
		\begin{pmatrix}
			-7 \alpha /\left(50 e^5\right)  \\
			\frac{-7}{50 e^5}  \beta  c  e^{-2 i \tau _* \omega _*}  \\
		\end{pmatrix}
	\end{equation}
	
	and
	\begin{equation}  \label{eq:H11}
		H_{11}(0) = -g_{11}q(0) - \bar{g}_{11}\bar{q}(0) + 2\tau_* 
		\begin{pmatrix}
			-7 \alpha /\left(25 e^5\right)   \\
			-7 \beta    \text{Re}\{c\} /\left(50 e^5\right) \\
		\end{pmatrix}
	\end{equation}

	Substituting (\ref{eq:W20}) and (\ref{eq:H20}) into (\ref{eq:inteta1}), and noticing that
	
	\begin{equation}
		\left( i \omega_* \tau_* I - \int_{-1}^{0} 	e^{i \omega_* \tau_* \vartheta} d\eta(\vartheta) \right) q(0) = 0
	\end{equation}
	
	and
	
	\begin{equation}
		\left(- i \omega_* \tau_* I - \int_{-1}^{0} 	e^{-i \omega_* \tau_* \vartheta} d\eta(\vartheta) \right) \bar{q}(0) = 0
	\end{equation}
	
	we obtain,
	\begin{equation}
		\left(2 i \omega_* \tau_* I - \int_{-1}^{0} 	e^{2i \omega_* \tau_* \vartheta} d\eta(\vartheta) \right) E_1 = 2\tau_* 
		\begin{pmatrix}
			-7 \alpha /\left(50 e^5\right)  \\
			\frac{-7}{50 e^5}  \beta  c  e^{-2 i \tau _* \omega _*}  \\
		\end{pmatrix}
	\end{equation}
	
	This leads to 
	\begin{equation*}
		\begin{pmatrix}
			\frac{7}{25} \beta  e^{\frac{\gamma }{20}-5} \gamma +\frac{7}{50} \alpha  x_* e^{\frac{\gamma }{20}-2i \tau _* \omega _*-5}+2i \omega _* & \frac{7 \beta ^2 e^{\frac{\gamma }{20}-5} \gamma ^2}{1000 \alpha }+\frac{7 \beta  \gamma  x_* e^{\frac{\gamma }{20}-2i \tau _* \omega _*-5}}{1000} \\
			\frac{7}{50} \beta  e^{\frac{\gamma }{20}-5} \gamma +\frac{7}{50} \beta  y_* e^{\frac{\gamma }{20}-2i \tau _* \omega _*-5} & \frac{7 \beta ^2 e^{\frac{\gamma }{20}-5} \gamma  (\gamma +20)}{1000 \alpha }+\frac{7 \beta ^2 \gamma  y_* e^{\frac{\gamma }{20}-2i \tau _* \omega _*-5}}{1000 \alpha }+2i \omega _* \\
		\end{pmatrix}
	\end{equation*}
	\begin{equation}
		\times E_1 = 2
		\begin{pmatrix}
			-7 \alpha /\left(50 e^5\right)  \\
			\frac{-7}{50 e^5}  \beta  c  e^{-2 i \tau _* \omega _*}  \\
		\end{pmatrix}
	\end{equation}
	
	It follows that 
	\begin{equation}
		E_1^{(1)} = \frac{2}{A}	\left| 
		\begin{array}{cc}
			-7 \alpha /\left(50 e^5\right) &  \frac{7 \beta ^2 e^{\frac{\gamma }{20}-5} \gamma ^2}{1000 \alpha }+\frac{7 \beta  \gamma  x_* e^{\frac{\gamma }{20}-2i \tau _* \omega _*-5}}{1000} \\
			\frac{-7}{50 e^5}  \beta  c  e^{-2 i \tau _* \omega _*} & \frac{7 \beta ^2 e^{\frac{\gamma }{20}-5} \gamma  (\gamma +20)}{1000 \alpha }+\frac{7 \beta ^2 \gamma  y_* e^{\frac{\gamma }{20}-2i \tau _* \omega _*-5}}{1000 \alpha }+2i \omega _* \\
		\end{array}
		\right| 
	\end{equation}
	
	and 
	
	\begin{equation}
		E_1^{(2)} = \frac{2}{A}	\left| 
		\begin{array}{cc}
			\frac{7}{25} \beta  e^{\frac{\gamma }{20}-5} \gamma +\frac{7}{50} \alpha  x_* e^{\frac{\gamma }{20}-2i \tau _* \omega _*-5}+2i \omega _* &	-7 \alpha /\left(50 e^5\right) \\
			\frac{7}{50} \beta  e^{\frac{\gamma }{20}-5} \gamma +\frac{7}{50} \beta  y_* e^{\frac{\gamma }{20}-2i \tau _* \omega _*-5}	&	\frac{-7}{50 e^5}  \beta  c  e^{-2 i \tau _* \omega _*}  \\
		\end{array}
		\right| 
	\end{equation}

	where $ A =  $
	
	\begin{equation*} \left| 
		\begin{array}{cc}
			\frac{7}{25} \beta  e^{\frac{\gamma }{20}-5} \gamma +\frac{7}{50} \alpha  x_* e^{\frac{\gamma }{20}-2i \tau _* \omega _*-5}+2i \omega _* & \frac{7 \beta ^2 e^{\frac{\gamma }{20}-5} \gamma ^2}{1000 \alpha }+\frac{7 \beta  \gamma  x_* e^{\frac{\gamma }{20}-2i \tau _* \omega _*-5}}{1000} \\
			\frac{7}{50} \beta  e^{\frac{\gamma }{20}-5} \gamma +\frac{7}{50} \beta  y_* e^{\frac{\gamma }{20}-2i \tau _* \omega _*-5} & \frac{7 \beta ^2 e^{\frac{\gamma }{20}-5} \gamma  (\gamma +20)}{1000 \alpha }+\frac{7 \beta ^2 \gamma  y_* e^{\frac{\gamma }{20}-2i \tau _* \omega _*-5}}{1000 \alpha }+2i \omega _* \\
		\end{array} \right|
	\end{equation*}

	Similarly, substituting (\ref{eq:W11}) and (\ref{eq:H11}) into (\ref{eq:inteta1}), we get
	
	\begin{equation*}
		\left(
		\begin{array}{cc}
			\frac{7}{50} \alpha  e^{\frac{\gamma }{20}-5} x_*-\frac{7}{25} \beta  e^{\frac{\gamma }{20}-5} \gamma  & \frac{7 \beta  e^{\frac{\gamma }{20}-5} \gamma  x_*}{1000}-\frac{7 \beta ^2 e^{\frac{\gamma }{20}-5} \gamma ^2}{1000 \alpha } \\
			\frac{7}{50} \beta  e^{\frac{\gamma }{20}-5} y_*-\frac{7}{50} \beta  e^{\frac{\gamma }{20}-5} \gamma  & \frac{7 \beta ^2 e^{\frac{\gamma }{20}-5} \gamma  y_*}{1000 \alpha }-\frac{7 \beta ^2 e^{\frac{\gamma }{20}-5} \gamma  (\gamma +20)}{1000 \alpha } \\
		\end{array}
		\right)
	\end{equation*}
	\begin{equation}
		\times E_2 = 2
		\begin{pmatrix}
			-7 \alpha /\left(25 e^5\right)  \\
			-7 \beta    \text{Re}\{c\} /\left(50 e^5\right)  \\
		\end{pmatrix}
	\end{equation}

	It follows that 
	\begin{equation}
		E_2^{(1)} = \frac{2}{B}	\left| 
		\begin{array}{cc}
			-7 \alpha /\left(25 e^5\right) &  \frac{7 \beta  e^{\frac{\gamma }{20}-5} \gamma  x_*}{1000}-\frac{7 \beta ^2 e^{\frac{\gamma }{20}-5} \gamma ^2}{1000 \alpha } \\
			-7 \beta    \text{Re}\{c\} /\left(50 e^5\right)  & \frac{7 \beta ^2 e^{\frac{\gamma }{20}-5} \gamma  y_*}{1000 \alpha }-\frac{7 \beta ^2 e^{\frac{\gamma }{20}-5} \gamma  (\gamma +20)}{1000 \alpha } \\
		\end{array}
		\right| 
	\end{equation}
	
	and 
	
	\begin{equation}
		E_2^{(2)} = \frac{2}{B}	\left| 
		\begin{array}{cc}
			\frac{7}{50} \alpha  e^{\frac{\gamma }{20}-5} x_*-\frac{7}{25} \beta  e^{\frac{\gamma }{20}-5} \gamma &-7 \alpha /\left(25 e^5\right) \\
			\frac{7}{50} \beta  e^{\frac{\gamma }{20}-5} y_*-\frac{7}{50} \beta  e^{\frac{\gamma }{20}-5} \gamma	&		-7 \beta    \text{Re}\{c\} /\left(50 e^5\right)   \\
		\end{array}
		\right| 
	\end{equation}

	where 
	
	\begin{equation*} B = \left| 
		\begin{array}{cc}
			\frac{7}{50} \alpha  e^{\frac{\gamma }{20}-5} x_*-\frac{7}{25} \beta  e^{\frac{\gamma }{20}-5} \gamma  & \frac{7 \beta  e^{\frac{\gamma }{20}-5} \gamma  x_*}{1000}-\frac{7 \beta ^2 e^{\frac{\gamma }{20}-5} \gamma ^2}{1000 \alpha } \\
			\frac{7}{50} \beta  e^{\frac{\gamma }{20}-5} y_*-\frac{7}{50} \beta  e^{\frac{\gamma }{20}-5} \gamma  & \frac{7 \beta ^2 e^{\frac{\gamma }{20}-5} \gamma  y_*}{1000 \alpha }-\frac{7 \beta ^2 e^{\frac{\gamma }{20}-5} \gamma  (\gamma +20)}{1000 \alpha } \\
		\end{array}
		\right|
	\end{equation*}

	Thus, we can determine $ W_{20}(\vartheta) $ and $ W_{11} (\vartheta) $ from (\ref{eq:W20}) and (\ref{eq:W11}). Furthermore, $ g_{21} $ in (\ref{eq:gcoeff}) can be expressed by the parameters and delay. Thus, we can compute the following values:
	
	\begin{equation}
		\begin{aligned}
			c_1(0) &= \frac{i}{2 \omega_* \tau_*} \left(g_{20}g_{11} - 2 \left|g_{11} \right|^2 -  \frac{ \left|g_{02} \right|^2}{3}  \right) + \frac{g_{21}}{2}, 	\\
			\mu_{2} &= - \frac{\text{Re}\left\lbrace c_1(0) \right\rbrace }{\text{Re}\left\lbrace \lambda^{'} (\tau_*) \right\rbrace},\\
			\beta_2 &= 2 \text{Re}\left\lbrace c_1(0) \right\rbrace ,	\\
			T_2 &= - \frac{\text{Im}\left\lbrace c_1(0) \right\rbrace + \mu_{2} \text{Im}\left\lbrace \lambda^{'} (\tau_*) \right\rbrace}{\omega_* \tau_*}
		\end{aligned}
	\end{equation}

	which determines the qualities of bifurcation periodic solution in the center manifold at the critical value $ \tau_* .$

	Here $\mu_{2}$ determines the direction of the Hopf bifurcation. If $\mu_{2} > 0$, then the bifurcation is supercritical and the bifurcation periodic solutions exist for $\tau > \tau_*.$ $\beta_2$ determines the stability of the bifurcation periodic solutions: it is asymptotically stable if $\beta_2 < 0.$ $T_2$ determines the period of the bifurcation periodic solutions; the period increases if $T_2 > 0.$


	\section{Numerical Simulations}
	\label{numerical}

	In  section \ref{stability}, we derived that the positive equilibrium  $E_{*}(x_{*}, y_{*})$  is asymptotically stable for $0 \leq \tau < \tau_*$ and unstable for $\tau > \tau_*$ and the system (\ref{eq:sys1}) undergoes a Hopf bifurcation when $\tau = \tau_*.$ Here we will give the dynamic behaviors of the system with different values of the parameters $\alpha$ and $\beta$ with different time delay $\tau.$ 	The simulation results of  system  (\ref{eq:sys1}) are plotted using the software Mathematica Version 12.1 \cite{Mathematica12}.
	
	\subsection{The dynamic behavior with $\alpha = 0.5 \textrm{ and } \beta = 0.8$}

	Here we study the dynamic behavior of system (\ref{eq:sys1}) by changing the time delay $\tau$ with $\alpha = 0.5$ and $\beta = 0.8.$ All the simulations have initial conditions of  $x(t) = 35.5$ and $y(t) = 26.5.$ This has a unique positive equilibrium at $(x_*,y_*) = (29.1842, 18.2401).$  We observe that the equilibrium is stable for $\tau < 30.8017$ and unstable for $\tau > 30.8017.$ At a Hopf bifurcation, no new equilibrium arise. A periodic solution emerges at the equilibrium point as $\tau$ passes through the bifurcation value. 
	\begin{figure}[H]
		\centering
		\begin{subfigure}[b]{0.45\textwidth}
			\centering
			\includegraphics[width=1\textwidth]{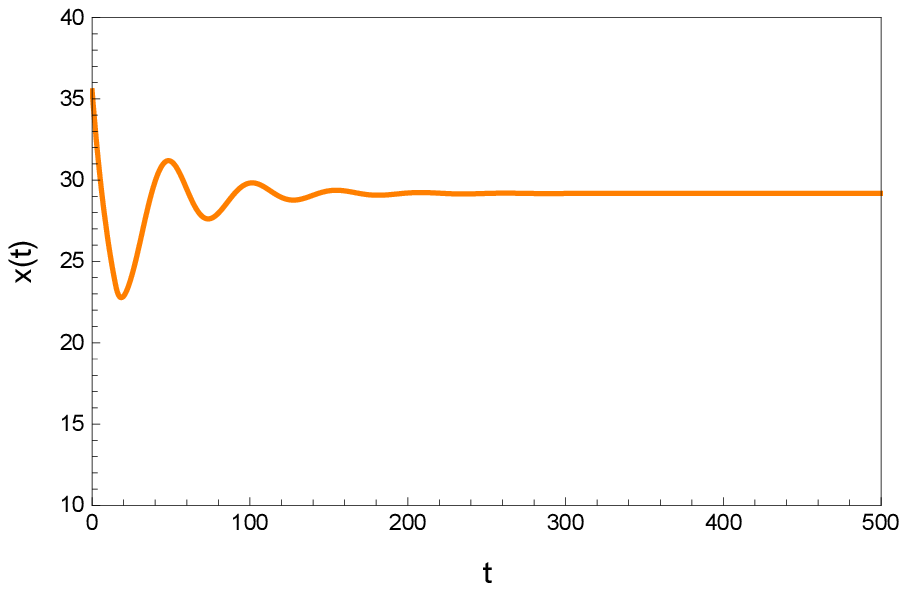}		
			\label{fig:xvst_alpha0p5beta0p8tau15}
		\end{subfigure}
		\hfill
		\begin{subfigure}[b]{0.45\textwidth}
			\centering
			\includegraphics[width=1\textwidth]{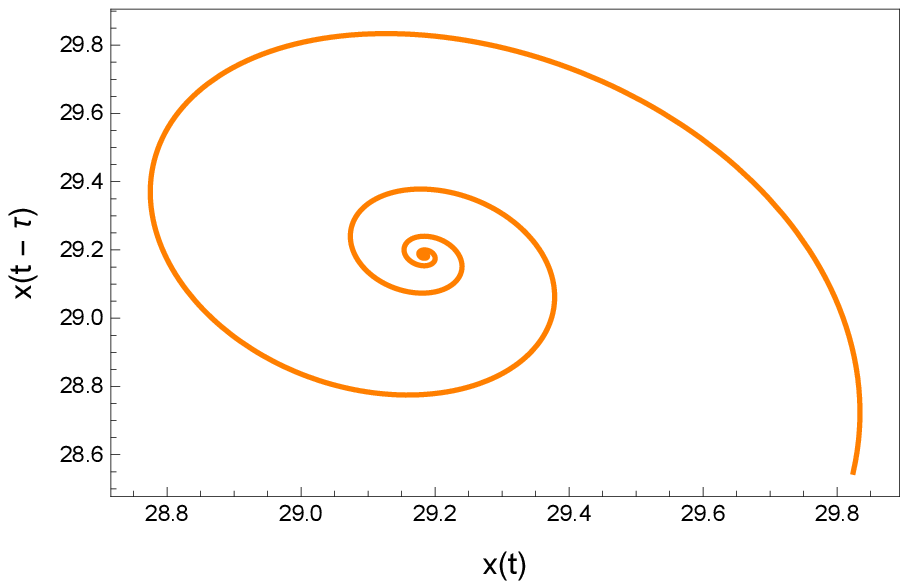}		
			\label{fig:xvsxtaut_alpha0p5beta0p8tau15}
		\end{subfigure}
		\hfill
		\begin{subfigure}[b]{0.45\textwidth}
			\centering
			\includegraphics[width=1\textwidth]{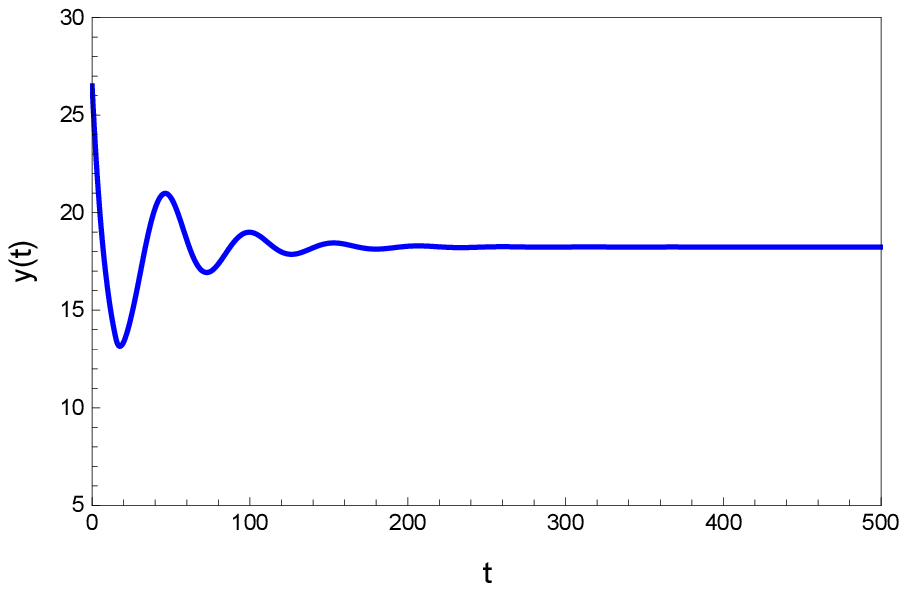}		
			\label{fig:yvst_alpha0p5beta0p8tau15}
		\end{subfigure}
		\hfill
		\begin{subfigure}[b]{0.45\textwidth}
			\centering
			\includegraphics[width=1\textwidth]{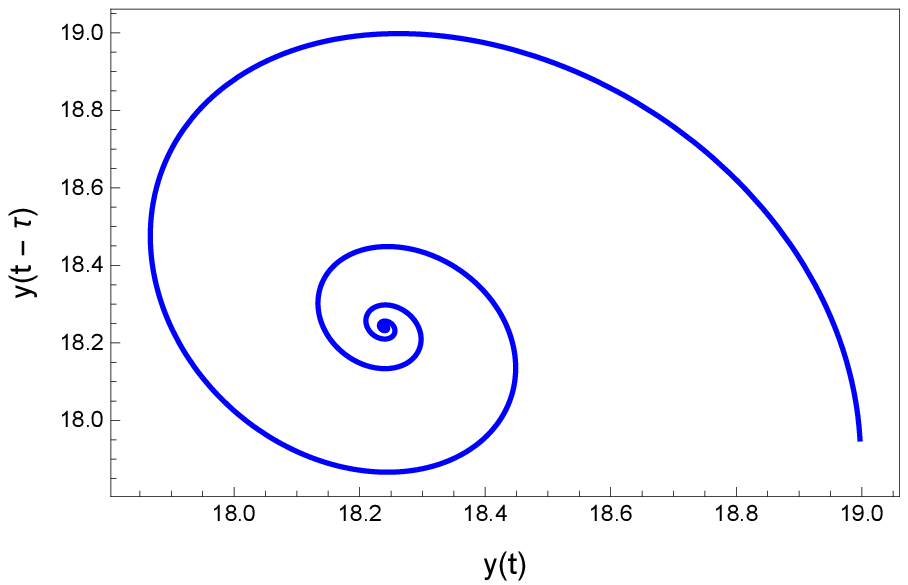}		
			\label{fig:yvsytau_alpha0p5beta0p8tau15}
		\end{subfigure}
		\hfill
		\begin{subfigure}[b]{0.45\textwidth}
			\centering
			\includegraphics[width=1\textwidth]{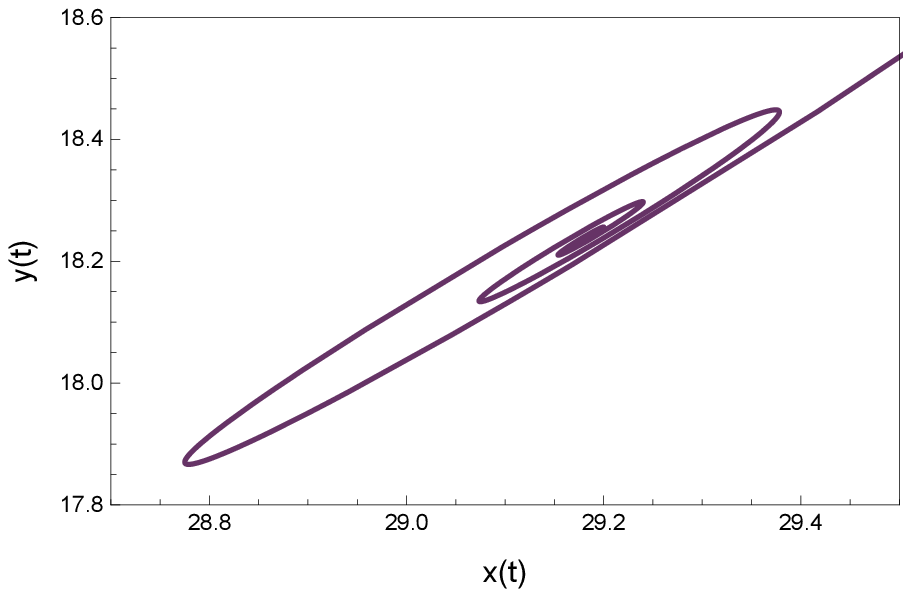}		
			\label{fig:xvsy_alpha0p5beta0p8tau15}
		\end{subfigure}
		\hfill
		\begin{subfigure}[b]{0.45\textwidth}
			\centering
			\includegraphics[width=1\textwidth]{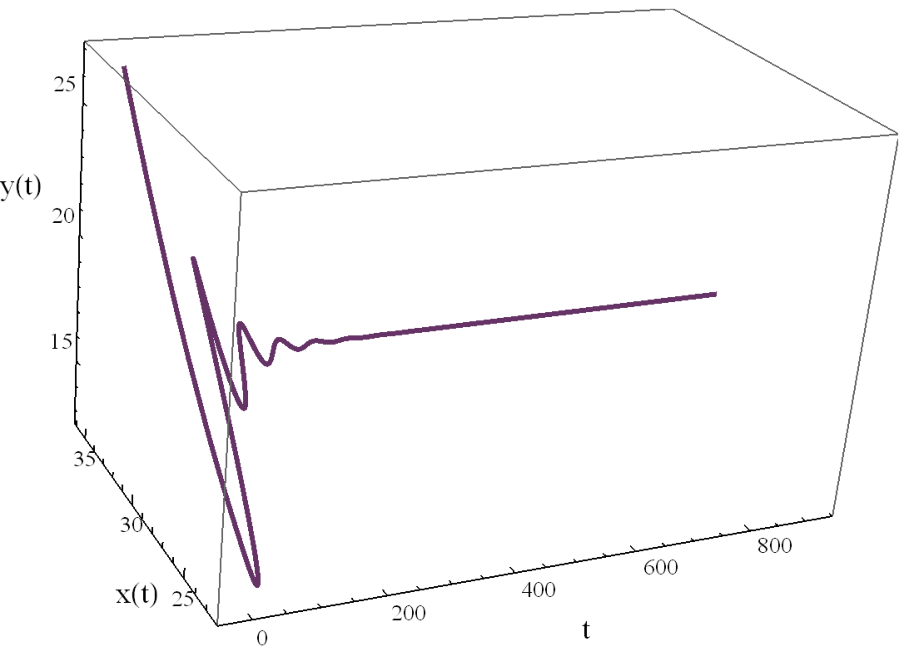}		
			\label{fig:xvsyvst_alpha0p5beta0p8tau15}
		\end{subfigure}
		\hfill
		\begin{subfigure}[b]{0.45\textwidth}
			\includegraphics[width=1\textwidth]{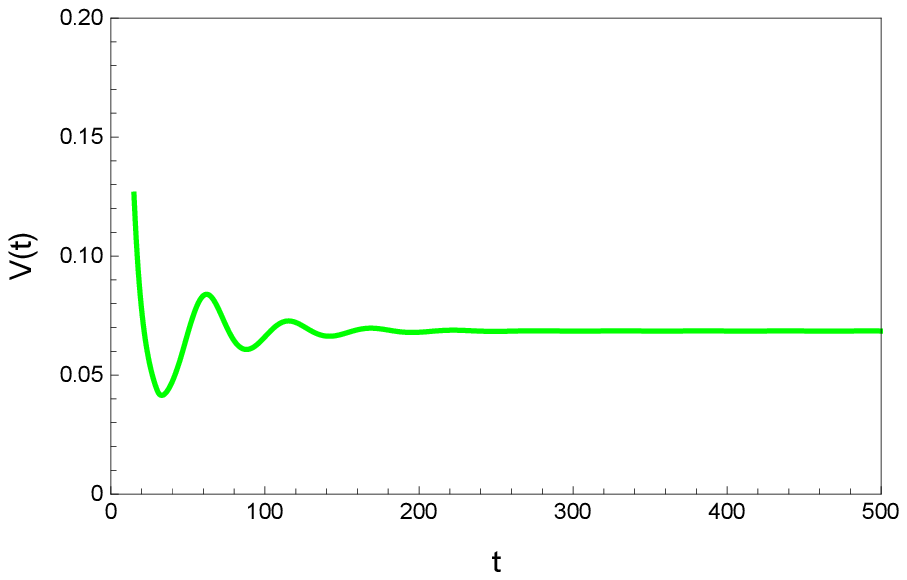}		
			\label{fig:Vvst_alpha0p5beta0p8tau15}
		\end{subfigure}
		\caption{The time series plots, phase plots and the ventilation plot with  $\alpha = 0.5, \; \beta = 0.8 $ and $ \tau = 15.$}
		\label{fig:sys1_alpha0p5_beta0p8_tau15}
		
	\end{figure}

	\begin{figure}[H]
		\centering
		\begin{subfigure}[b]{0.45\textwidth}
			\centering
			\includegraphics[width=1\textwidth]{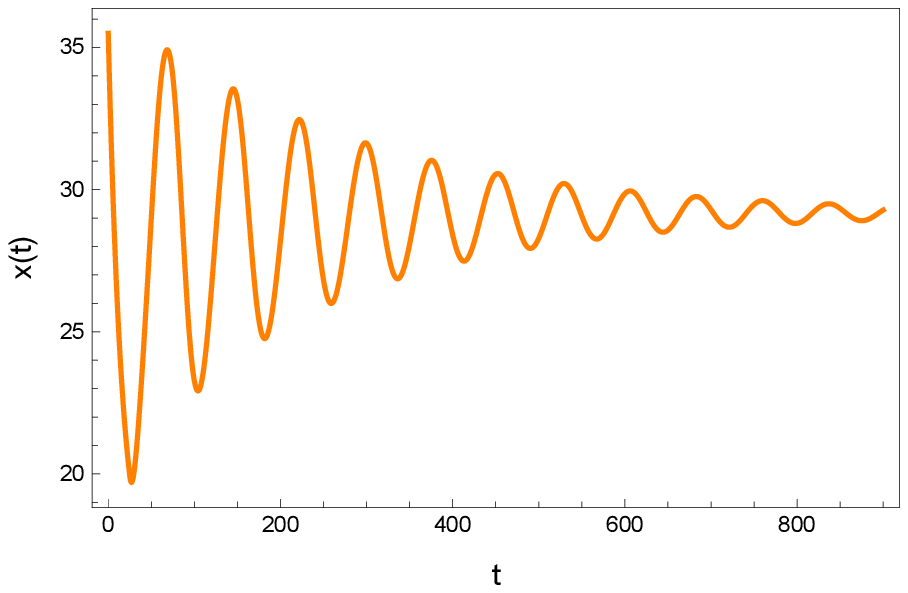}		
			\label{fig:xvst_alpha0p5beta0p8tau25}
		\end{subfigure}
		\hfill
		\begin{subfigure}[b]{0.45\textwidth}
			\centering
			\includegraphics[width=1\textwidth]{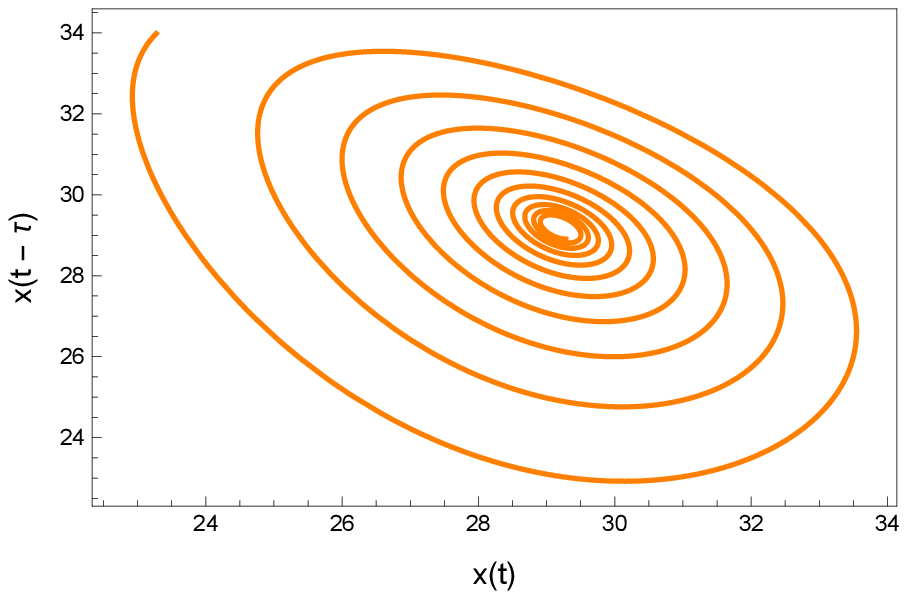}		
			\label{fig:xvsxtaut_alpha0p5beta0p8tau25}
		\end{subfigure}
		\hfill
		\begin{subfigure}[b]{0.45\textwidth}
			\centering
			\includegraphics[width=1\textwidth]{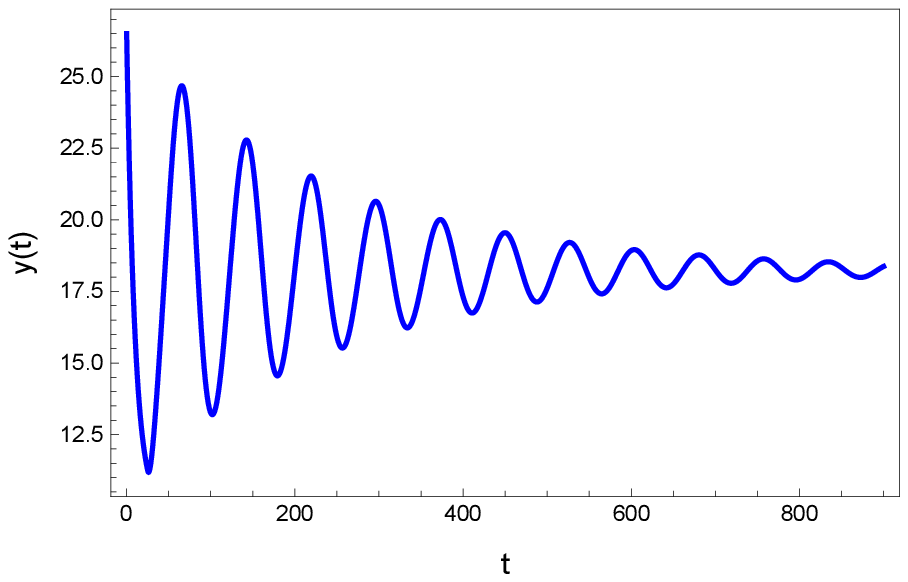}		
			\label{fig:yvst_alpha0p5beta0p8tau25}
		\end{subfigure}
		\hfill
		\begin{subfigure}[b]{0.45\textwidth}
			\centering
			\includegraphics[width=1\textwidth]{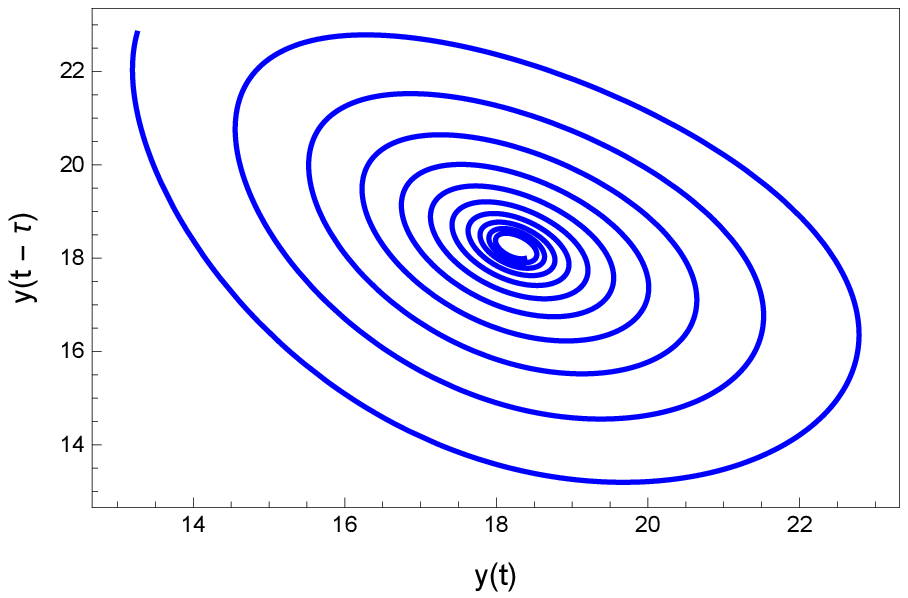}		
			\label{fig:yvsytau_alpha0p5beta0p8tau25}
		\end{subfigure}
		\hfill
		\begin{subfigure}[b]{0.45\textwidth}
			\centering
			\includegraphics[width=1\textwidth]{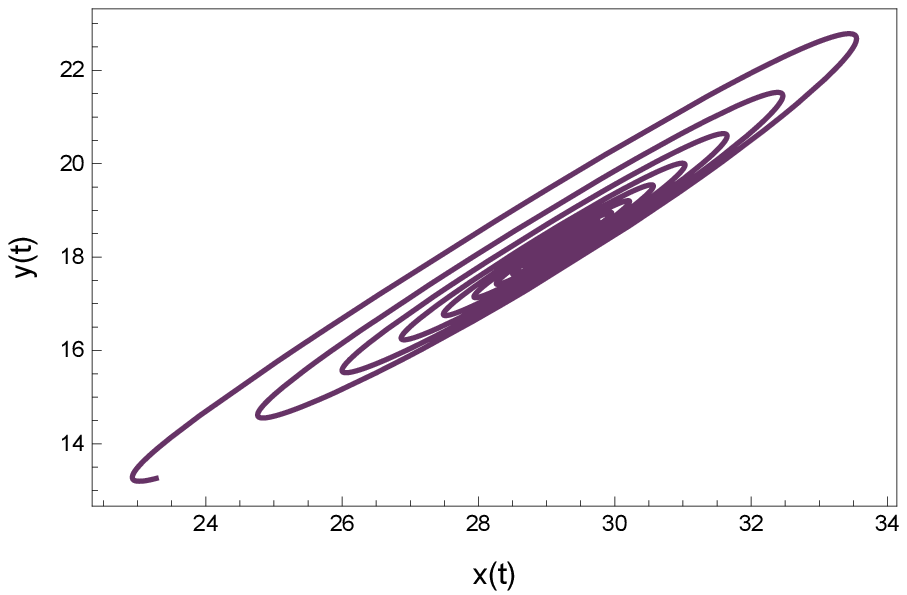}		
			\label{fig:xvsy_alpha0p5beta0p8tau25}
		\end{subfigure}
		\hfill
		\begin{subfigure}[b]{0.45\textwidth}
			\centering
			\includegraphics[width=1\textwidth]{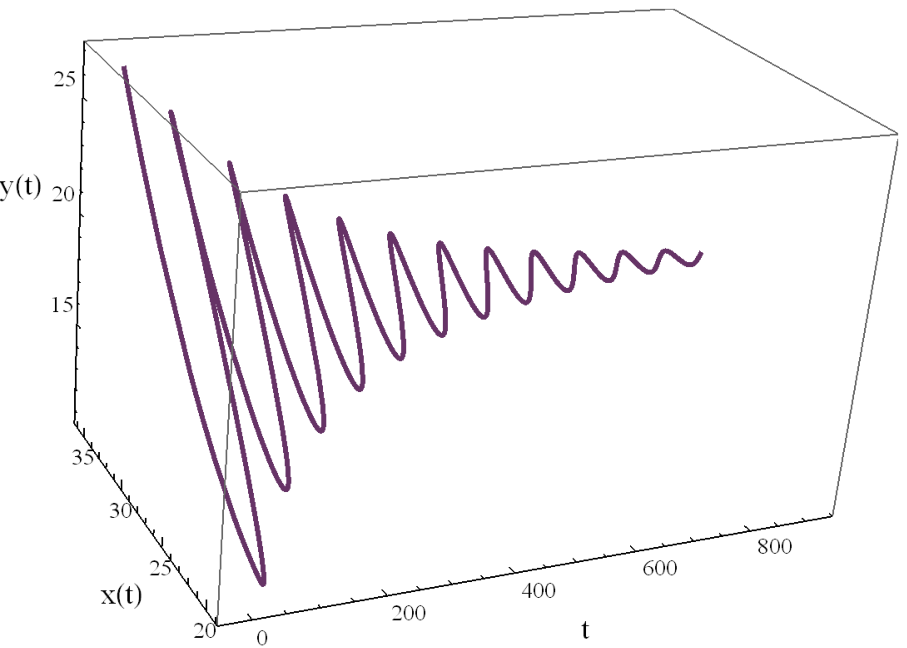}		
			\label{fig:xvsyvst_alpha0p5beta0p8tau25}
		\end{subfigure}
		\hfill
		\begin{subfigure}[b]{0.45\textwidth}
			\includegraphics[width=1\textwidth]{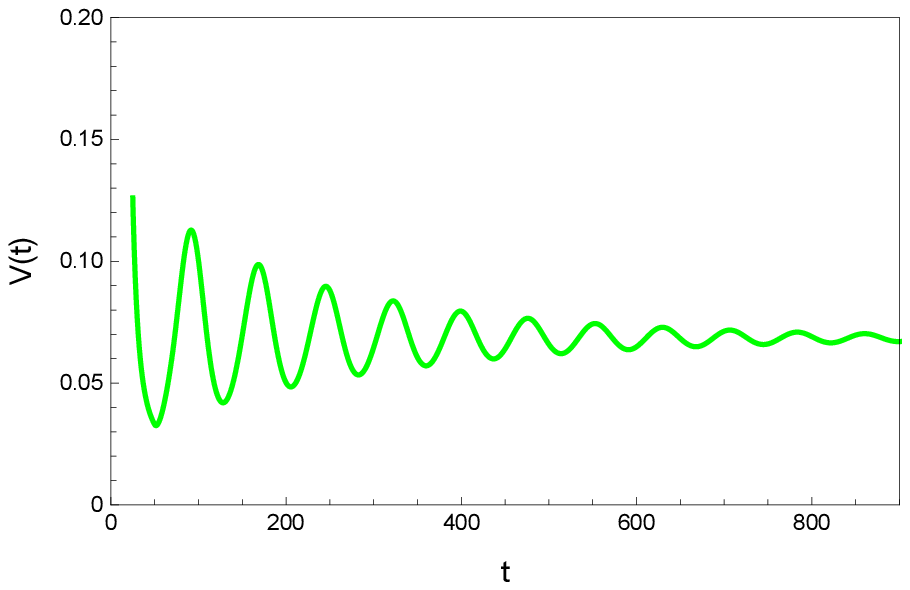}		
			\label{fig:Vvst_alpha0p5beta0p8tau25}
		\end{subfigure}
		\caption{The time series plots, phase plots and the ventilation plot with $\alpha = 0.5, \; \beta = 0.8 $ and $ \tau = 25.  $}
		\label{fig:sys1_alpha0p5_beta0p8_tau25}
		
	\end{figure}
	
	\begin{figure}[H]
		\centering
		\begin{subfigure}[b]{0.45\textwidth}
			\centering
			\includegraphics[width=1\textwidth]{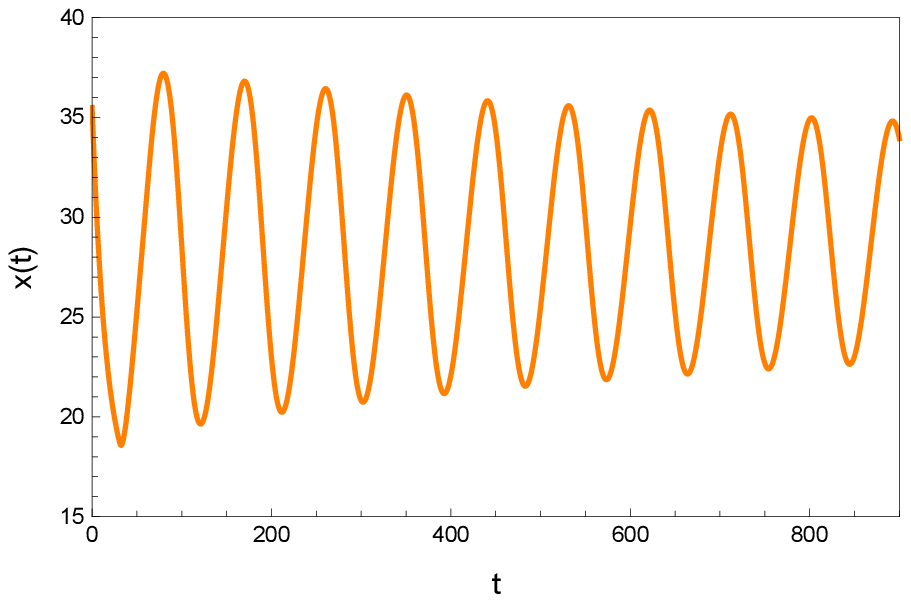}		
			\label{fig:xvst_alpha0p5beta0p8tau30p8}
		\end{subfigure}
		\hfill
		\begin{subfigure}[b]{0.45\textwidth}
			\centering
			\includegraphics[width=1\textwidth]{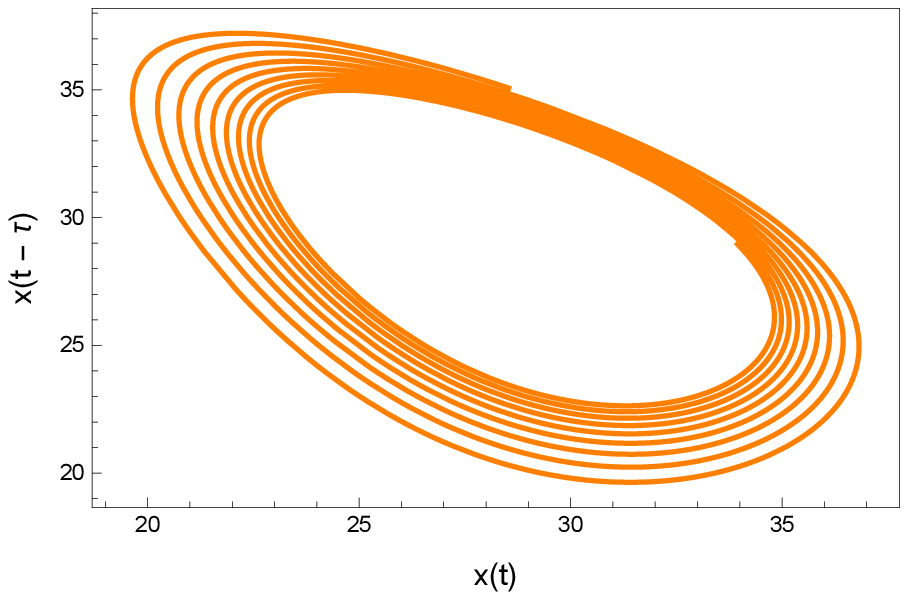}		
			\label{fig:xvsxtaut_alpha0p5beta0p8tau30p8}
		\end{subfigure}
		\hfill
		\begin{subfigure}[b]{0.45\textwidth}
			\centering
			\includegraphics[width=1\textwidth]{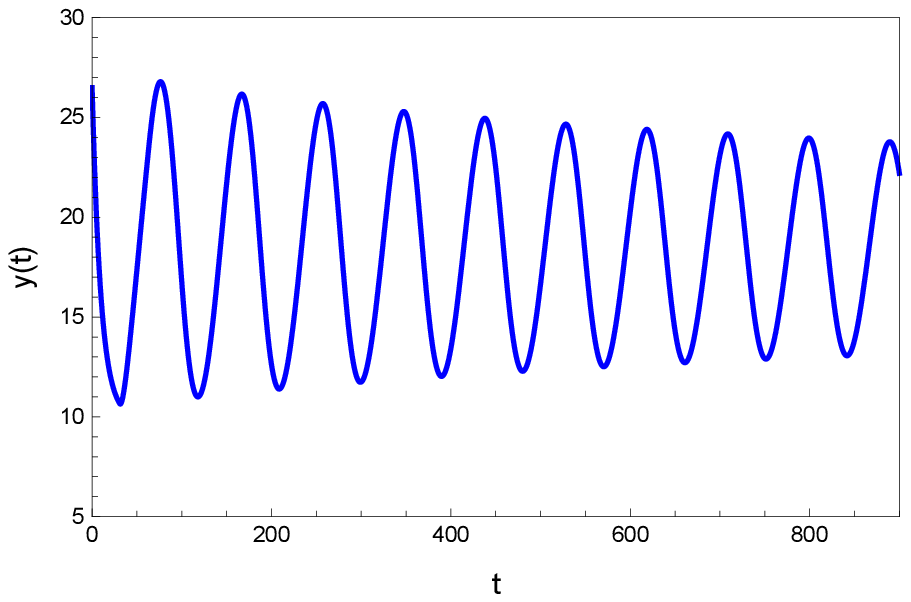}		
			\label{fig:yvst_alpha0p5beta0p8tau30p8}
		\end{subfigure}
		\hfill
		\begin{subfigure}[b]{0.45\textwidth}
			\centering
			\includegraphics[width=1\textwidth]{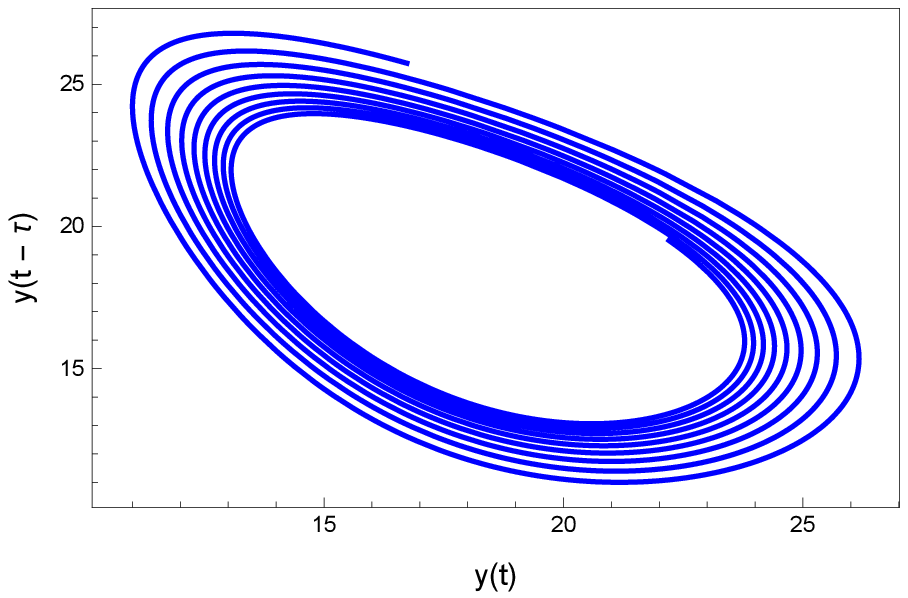}		
			\label{fig:yvsytau_alpha0p5beta0p8tau30p8}
		\end{subfigure}
		\hfill
		\begin{subfigure}[b]{0.45\textwidth}
			\centering
			\includegraphics[width=1\textwidth]{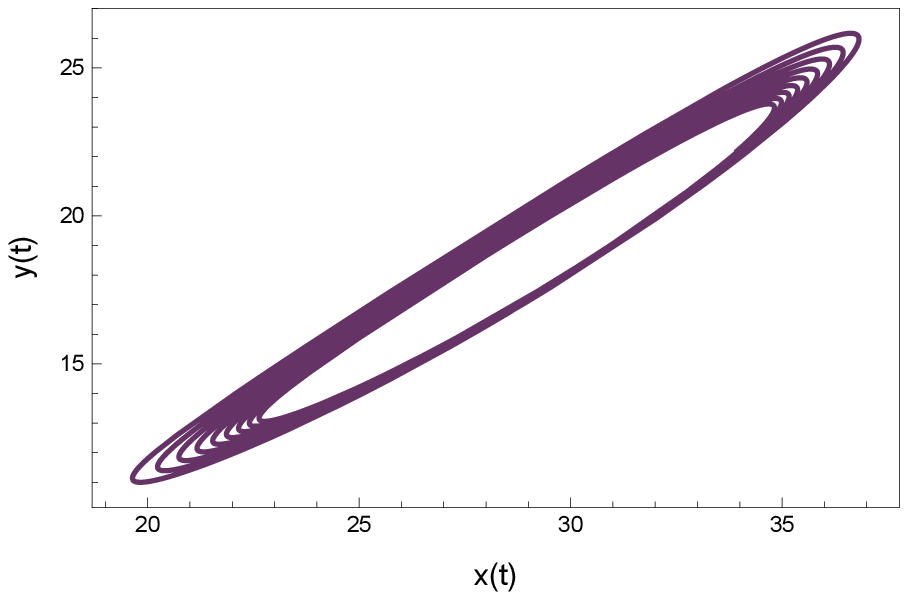}		
			\label{fig:xvsy_alpha0p5beta0p8tau30p8}
		\end{subfigure}
		\hfill
		\begin{subfigure}[b]{0.45\textwidth}
			\centering
			\includegraphics[width=1\textwidth]{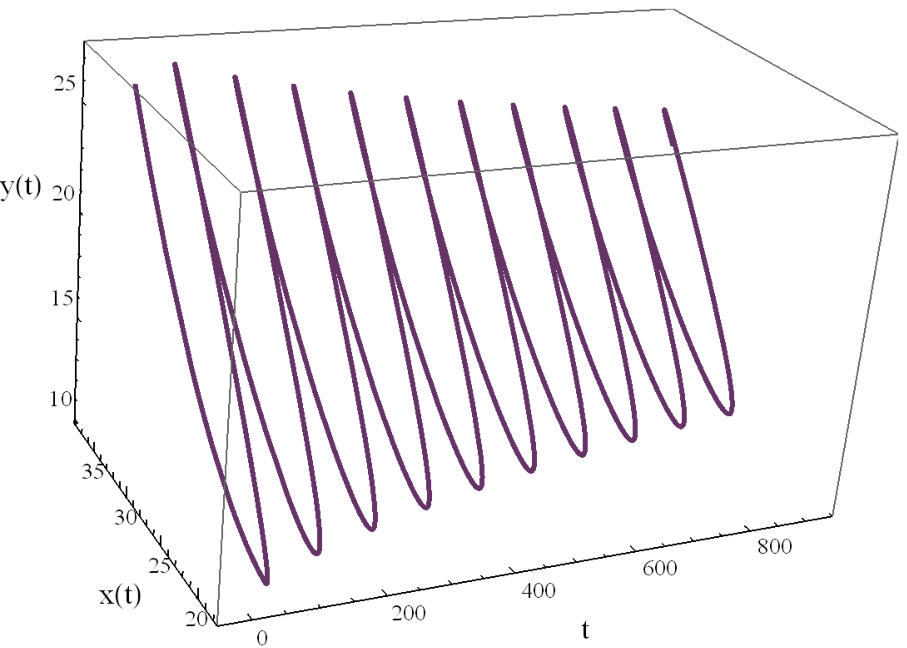}		
			\label{fig:xvsyvst_alpha0p5beta0p8tau30p8}
		\end{subfigure}
		\hfill
		\begin{subfigure}[b]{0.45\textwidth}
			\includegraphics[width=1\textwidth]{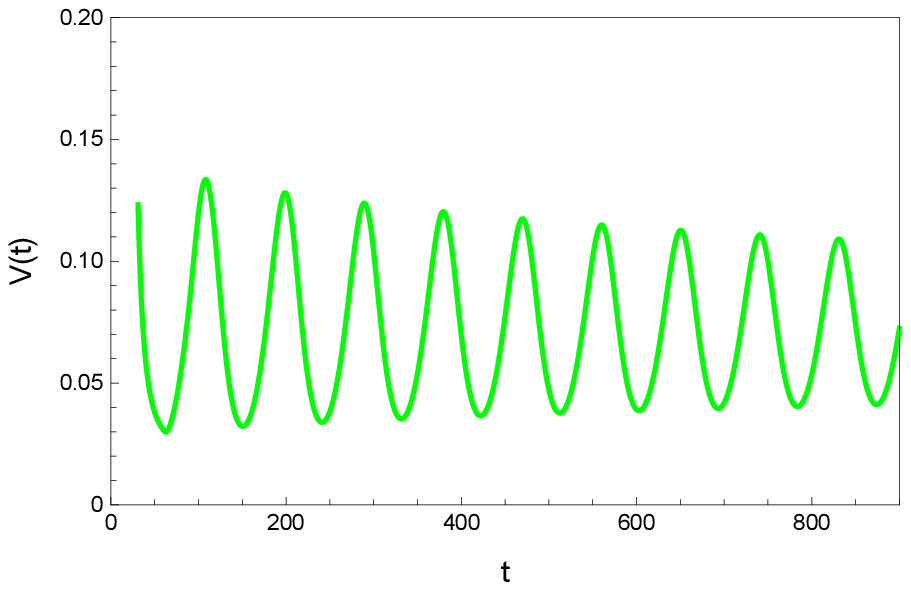}		
			\label{fig:Vvst_alpha0p5beta0p8tau30p8}
		\end{subfigure}
		\caption{The time series plots, phase plots and the ventilation plot with $\alpha = 0.5, \; \beta = 0.8 $ and $ \tau = 30.81.  $}
		\label{fig:sys1_alpha0p5_beta0p8_tau30p8}
		
	\end{figure}

	\begin{figure}[H]
		\centering
		\begin{subfigure}[b]{0.45\textwidth}
			\centering
			\includegraphics[width=1\textwidth]{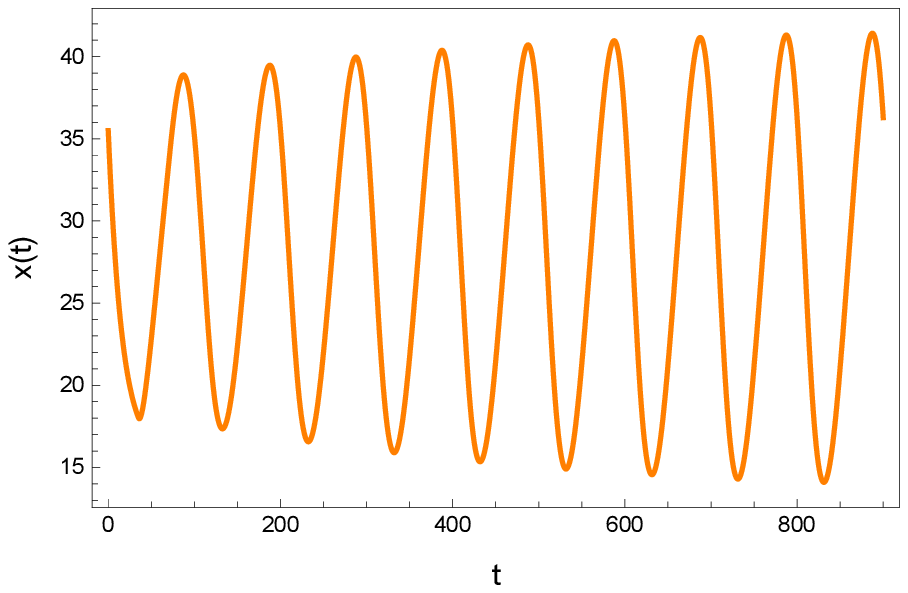}		
			\label{fig:xvst_alpha0p5beta0p8tau35}
		\end{subfigure}
		\hfill
		\begin{subfigure}[b]{0.45\textwidth}
			\centering
			\includegraphics[width=1\textwidth]{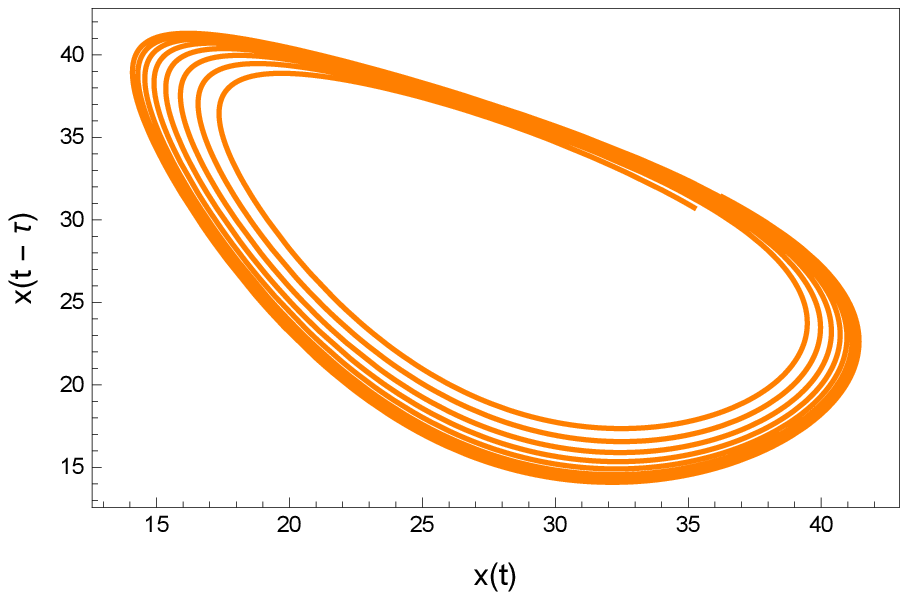}		
			\label{fig:xvsxtaut_alpha0p5beta0p8tau35}
		\end{subfigure}
		\hfill
		\begin{subfigure}[b]{0.45\textwidth}
			\centering
			\includegraphics[width=1\textwidth]{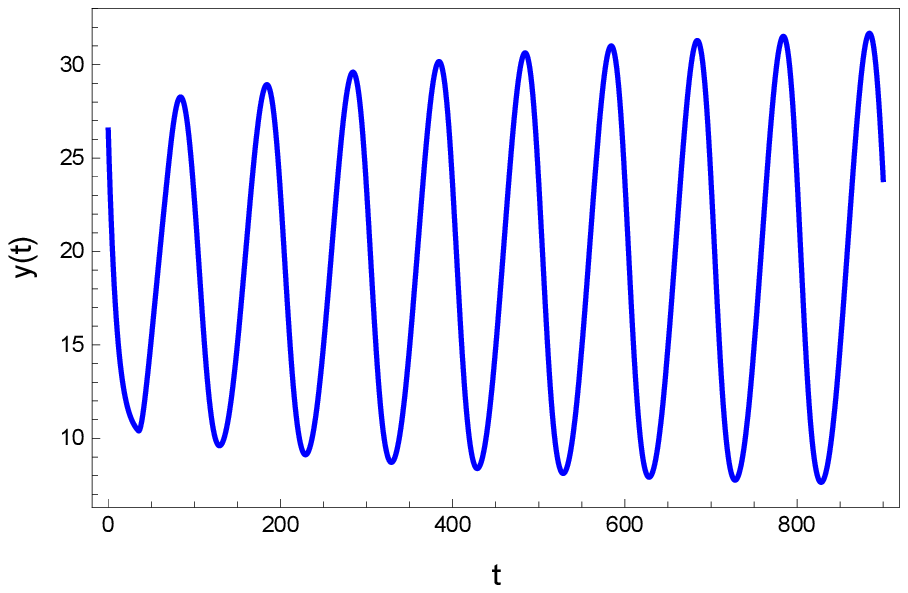}		
			\label{fig:yvst_alpha0p5beta0p8tau35}
		\end{subfigure}
		\hfill
		\begin{subfigure}[b]{0.45\textwidth}
			\centering
			\includegraphics[width=1\textwidth]{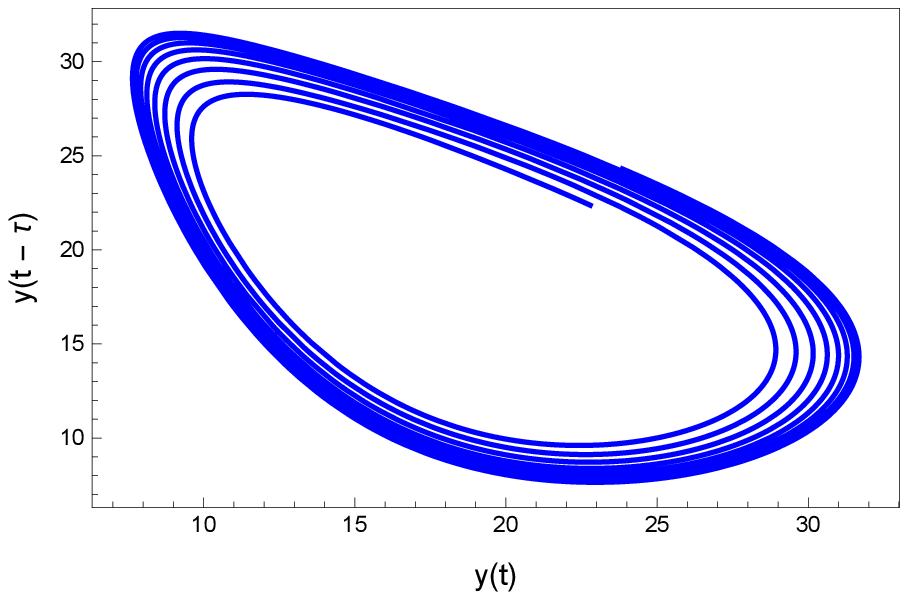}		
			\label{fig:yvsytau_alpha0p5beta0p8tau35}
		\end{subfigure}
		\hfill
		\begin{subfigure}[b]{0.45\textwidth}
			\centering
			\includegraphics[width=1\textwidth]{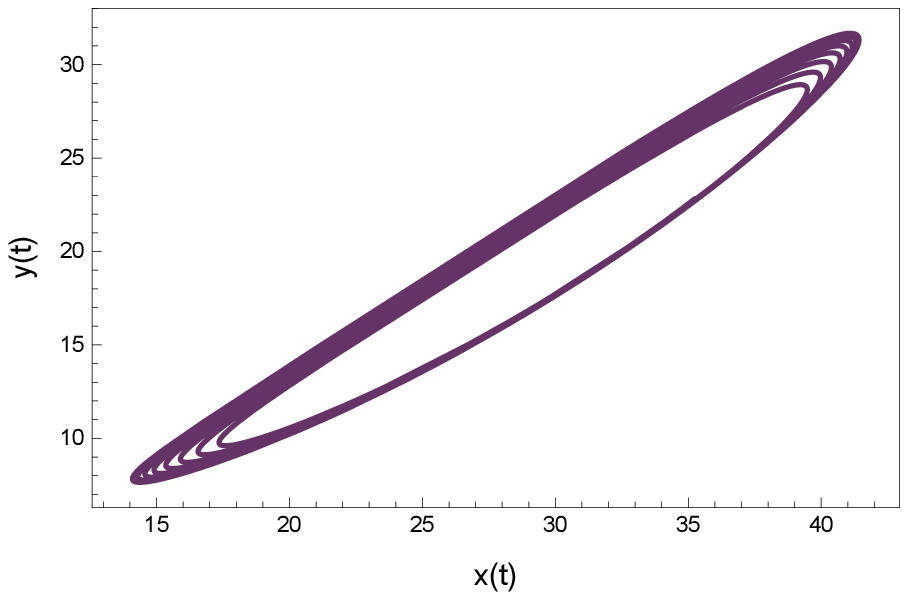}		
			\label{fig:xvsy_alpha0p5beta0p8tau35}
		\end{subfigure}
		\hfill
		\begin{subfigure}[b]{0.45\textwidth}
			\centering
			\includegraphics[width=1\textwidth]{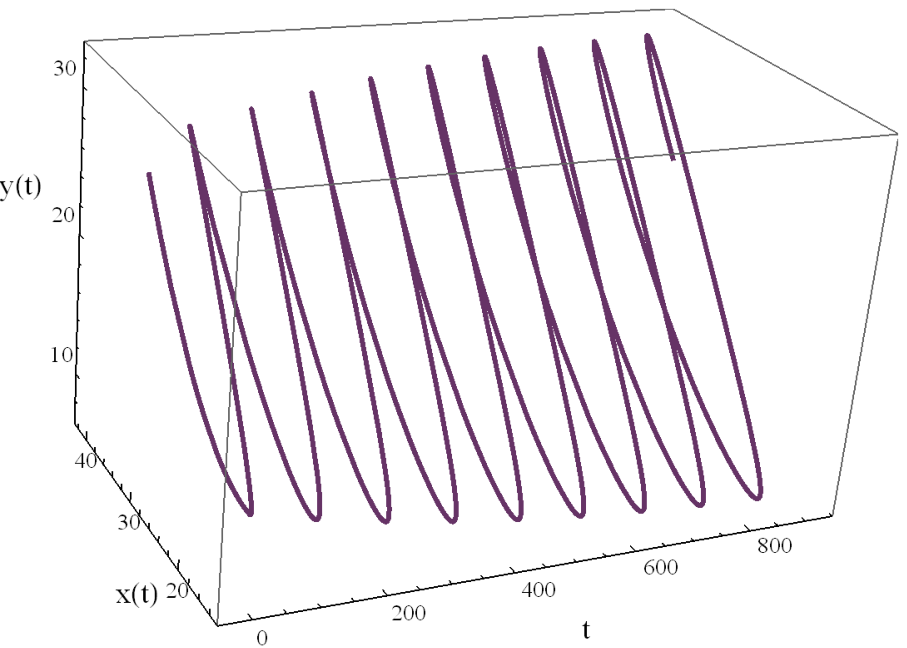}		
			\label{fig:xvsyvst_alpha0p5beta0p8tau35}
		\end{subfigure}
		\hfill
		\begin{subfigure}[b]{0.45\textwidth}
			\includegraphics[width=1\textwidth]{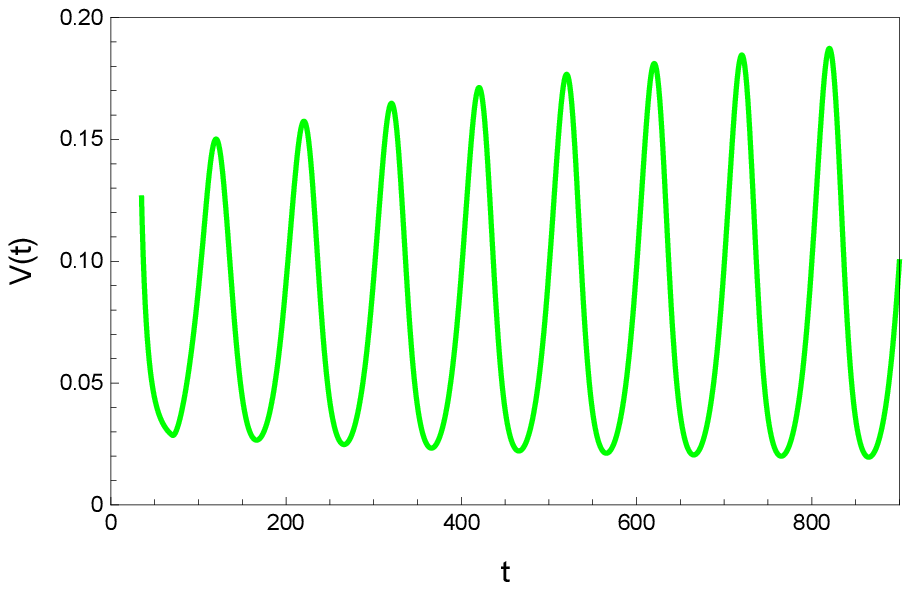}		
			\label{fig:Vvst_alpha0p5beta0p8tau35}
		\end{subfigure}
		\caption{The time series plots, phase plots and the ventilation plot with $\alpha = 0.5, \; \beta = 0.8 $ and $ \tau = 35.  $}
		\label{fig:sys1_alpha0p5_beta0p8_tau35}
		
	\end{figure}

	\begin{figure}[H]
		\centering
		\begin{subfigure}[b]{0.45\textwidth}
			\centering
			\includegraphics[width=1\textwidth]{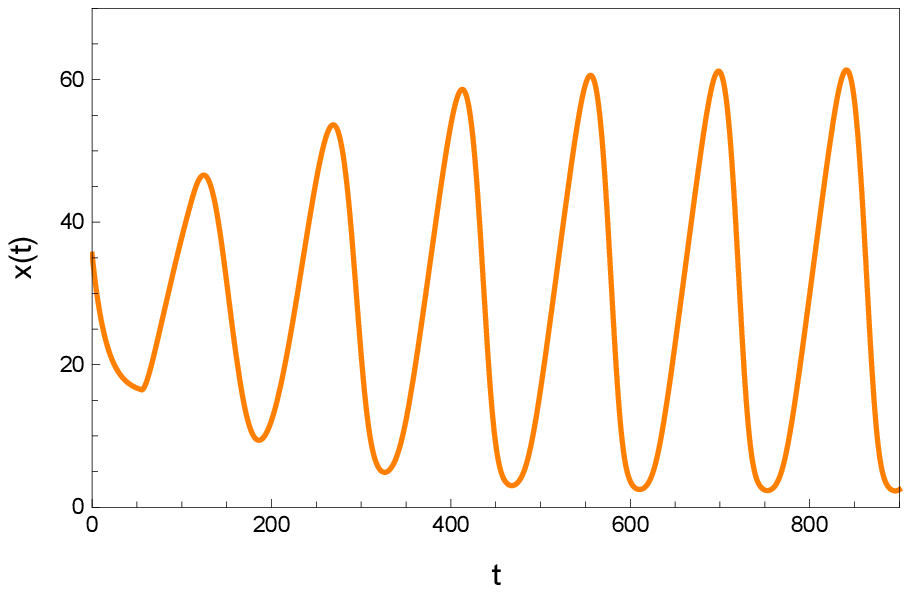}		
			\label{fig:xvst_alpha0p5beta0p8tau55}
		\end{subfigure}
		\hfill
		\begin{subfigure}[b]{0.45\textwidth}
			\centering
			\includegraphics[width=1\textwidth]{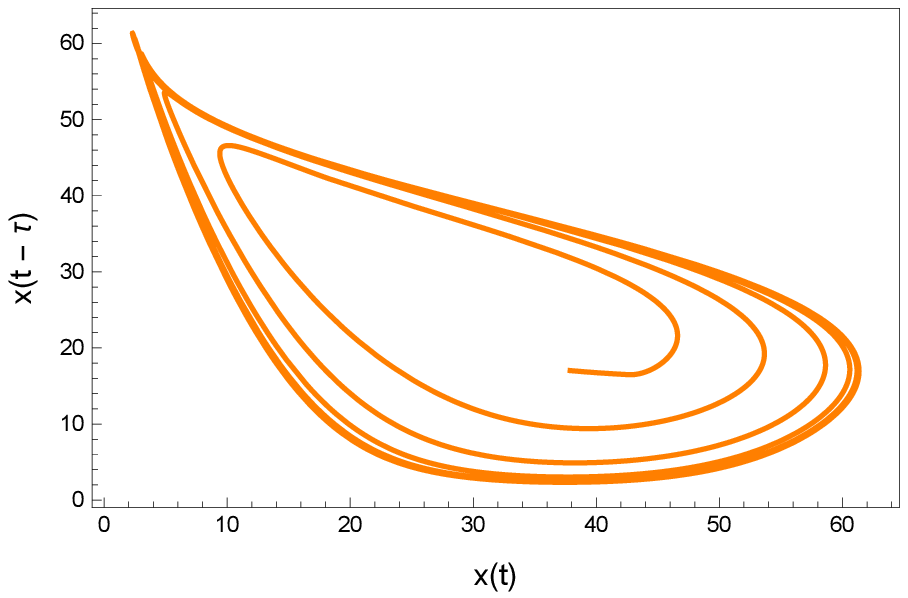}		
			\label{fig:xvsxtaut_alpha0p5beta0p8tau55}
		\end{subfigure}
		\hfill
		\begin{subfigure}[b]{0.45\textwidth}
			\centering
			\includegraphics[width=1\textwidth]{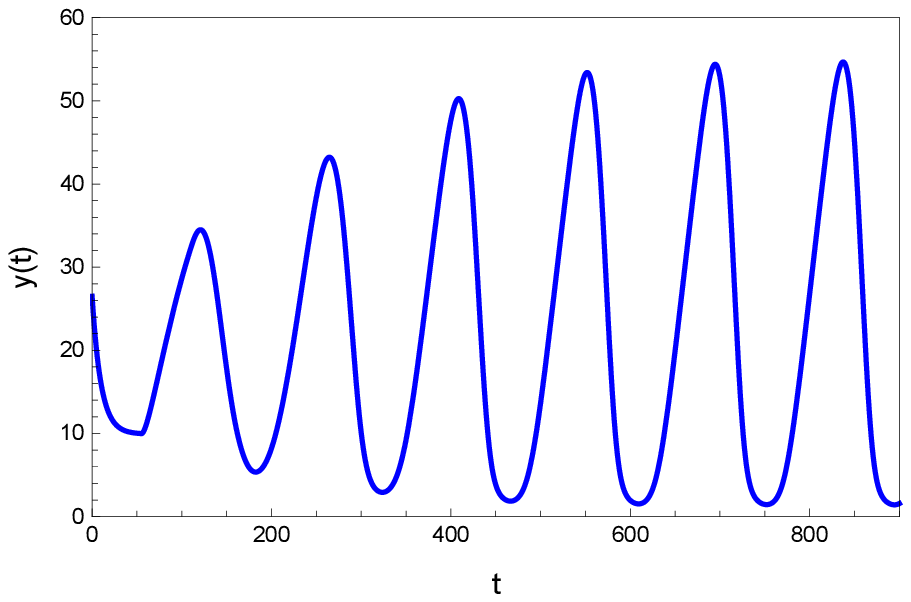}		
			\label{fig:yvst_alpha0p5beta0p8tau55}
		\end{subfigure}
		\hfill
		\begin{subfigure}[b]{0.45\textwidth}
			\centering
			\includegraphics[width=1\textwidth]{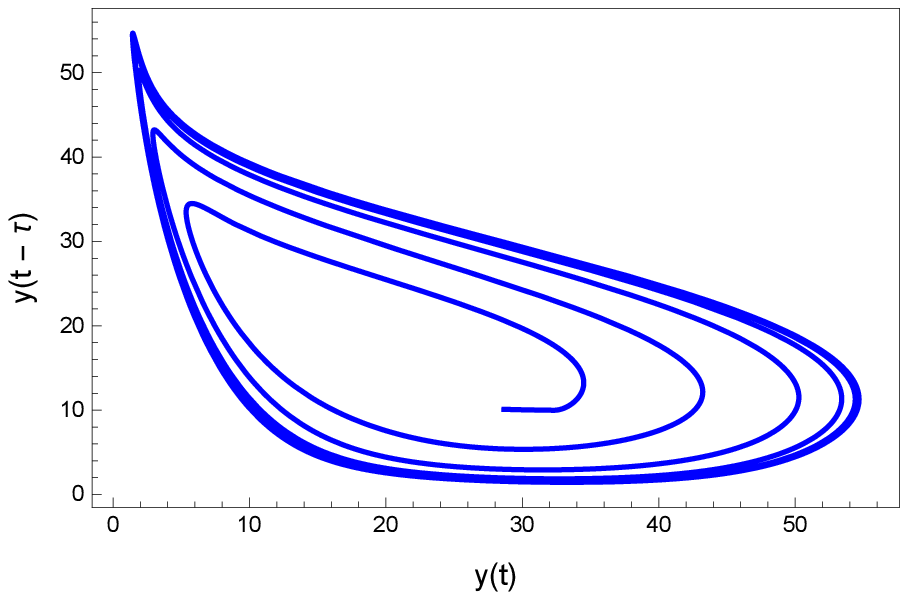}		
			\label{fig:yvsytau_alpha0p5beta0p8tau55}
		\end{subfigure}
		\hfill
		\begin{subfigure}[b]{0.45\textwidth}
			\centering
			\includegraphics[width=1\textwidth]{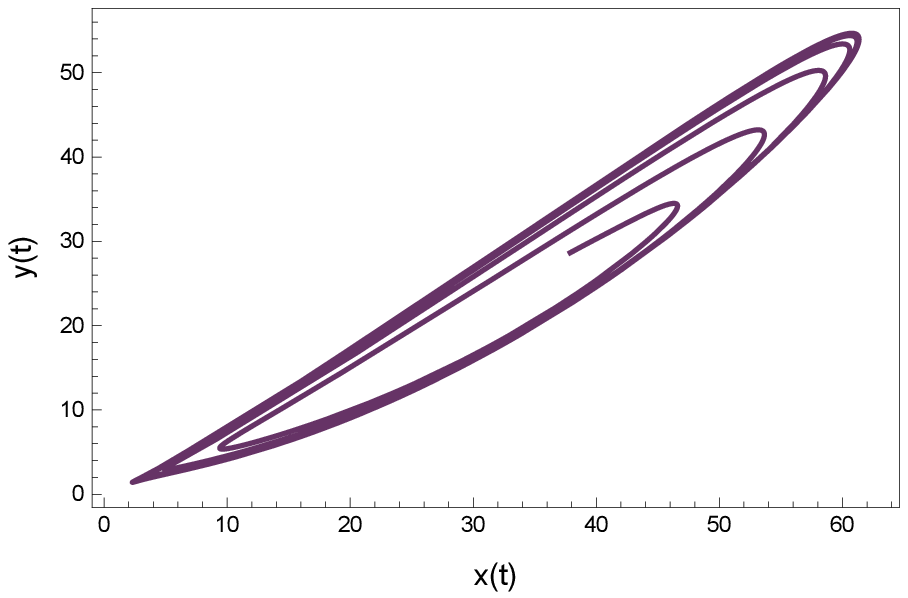}		
			\label{fig:xvsy_alpha0p5beta0p8tau55}
		\end{subfigure}
		\hfill
		\begin{subfigure}[b]{0.45\textwidth}
			\centering
			\includegraphics[width=1\textwidth]{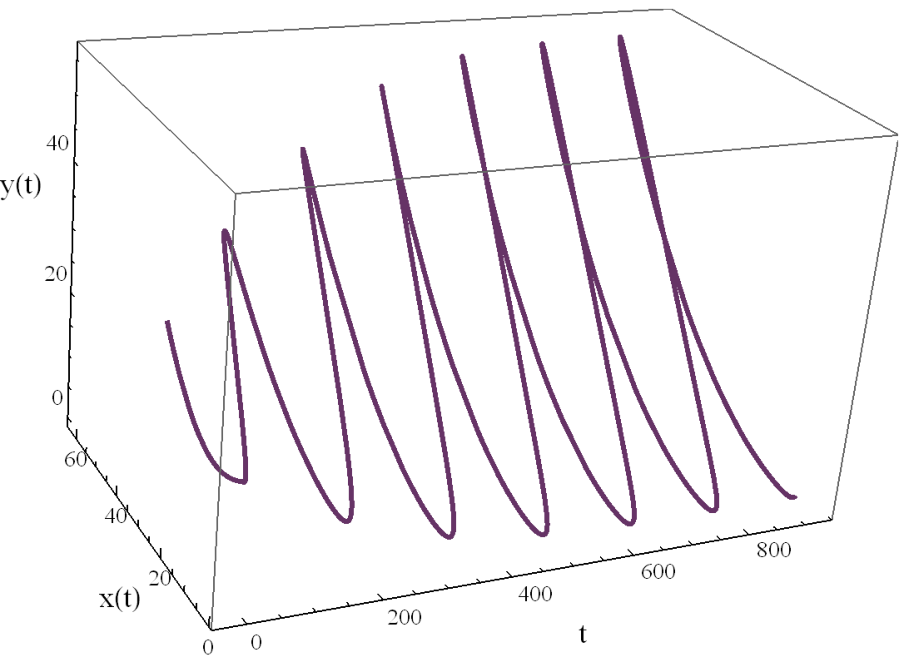}		
			\label{fig:xvsyvst_alpha0p5beta0p8tau55}
		\end{subfigure}
		\hfill
		\begin{subfigure}[b]{0.45\textwidth}
			\includegraphics[width=1\textwidth]{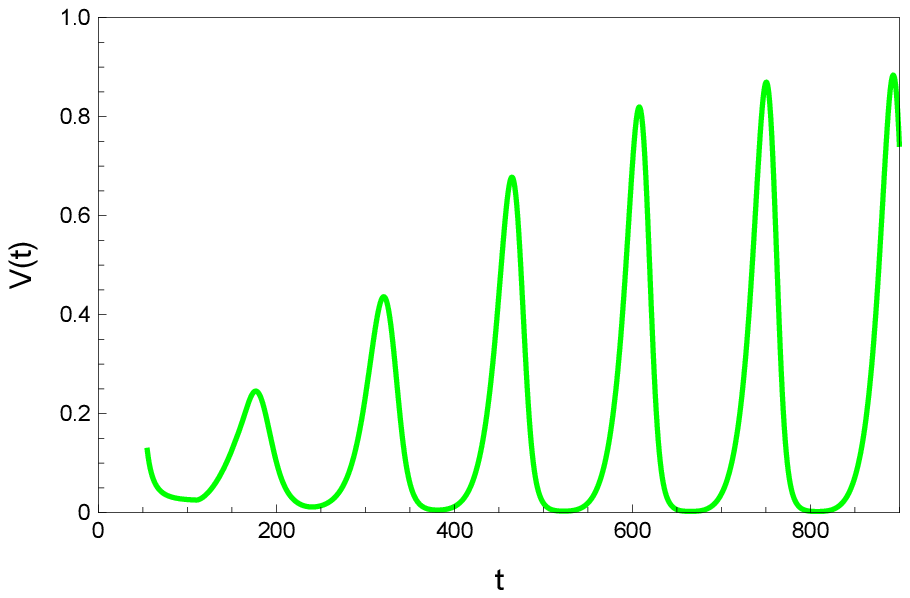}		
			\label{fig:Vvst_alpha0p5beta0p8tau55}
		\end{subfigure}
		\caption{The time series plots, phase plots and the ventilation plot with $\alpha = 0.5, \; \beta = 0.8 $ and $ \tau = 55.  $}
		\label{fig:sys1_alpha0p5_beta0p8_tau55}
		
	\end{figure}

	\subsection{The dynamic behavior with $\alpha = 0.6 \textrm{ and } \beta = 0.6$}

	Here we study the dynamic behavior of system (\ref{eq:sys1}) by changing the time delay $\tau$ with $\alpha = 0.6$ and $\beta = 0.6.$ Again the simulations have initial conditions of  $x(t) = 35.5$ and $y(t) = 26.5.$ This has a unique positive equilibrium at $(x_*,y_*) = (23.4108,   23.4108)$  as listed in \ref{table:eq_delay}.  We observe that the equilibrium is stable for $\tau < 24.9072$ and unstable for $\tau > 24.9072.$ At a Hopf bifurcation, no new equilibrium arise. A periodic solution emerges at the equilibrium point as $\tau$ passes through the bifurcation value. 
	
	\begin{figure}[H]
		\centering
		\begin{subfigure}[b]{0.45\textwidth}
			\centering
			\includegraphics[width=1\textwidth]{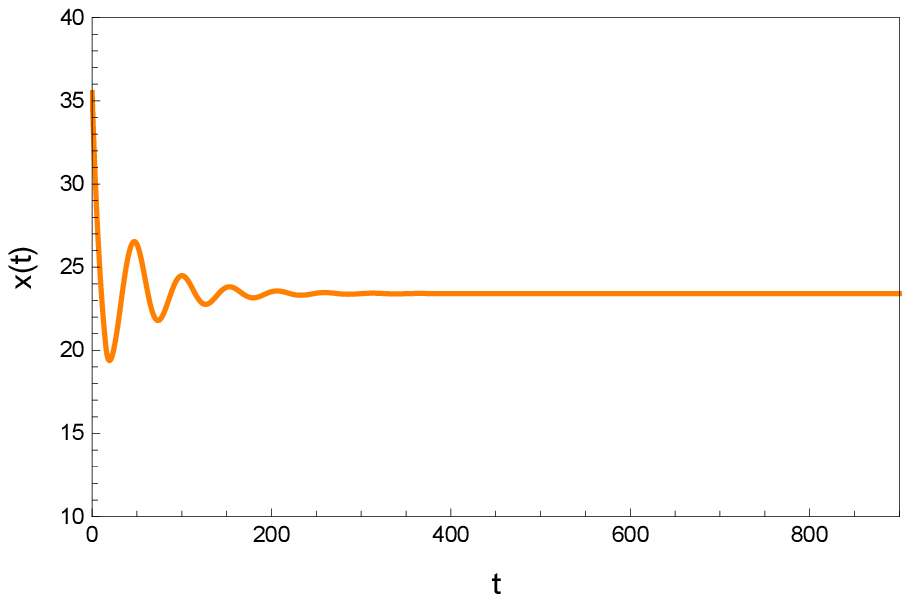}		
			\label{fig:xvst_alpha0p6beta0p6tau15}
		\end{subfigure}
		\hfill
		\begin{subfigure}[b]{0.45\textwidth}
			\centering
			\includegraphics[width=1\textwidth]{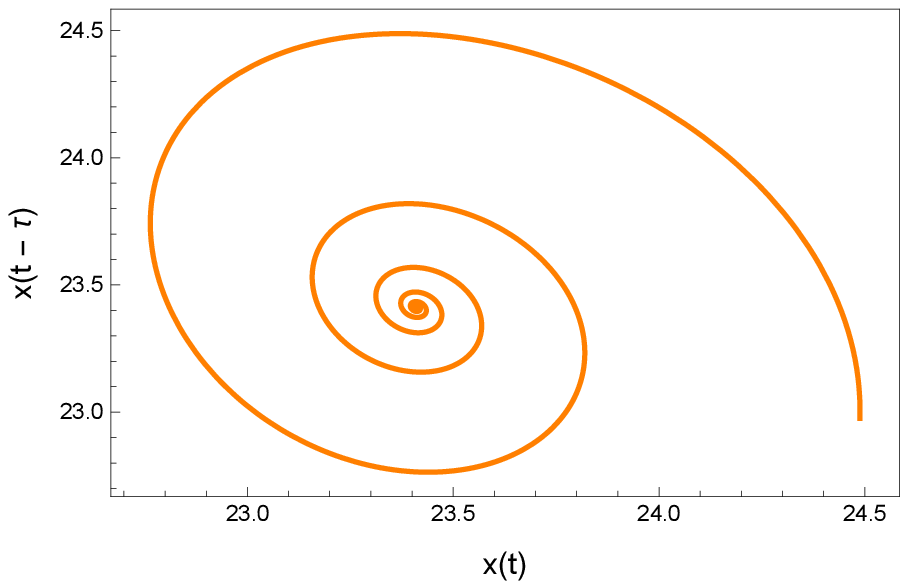}		
			\label{fig:xvsxtaut_alpha0p6beta0p6tau15}
		\end{subfigure}
		\hfill
		\begin{subfigure}[b]{0.45\textwidth}
			\centering
			\includegraphics[width=1\textwidth]{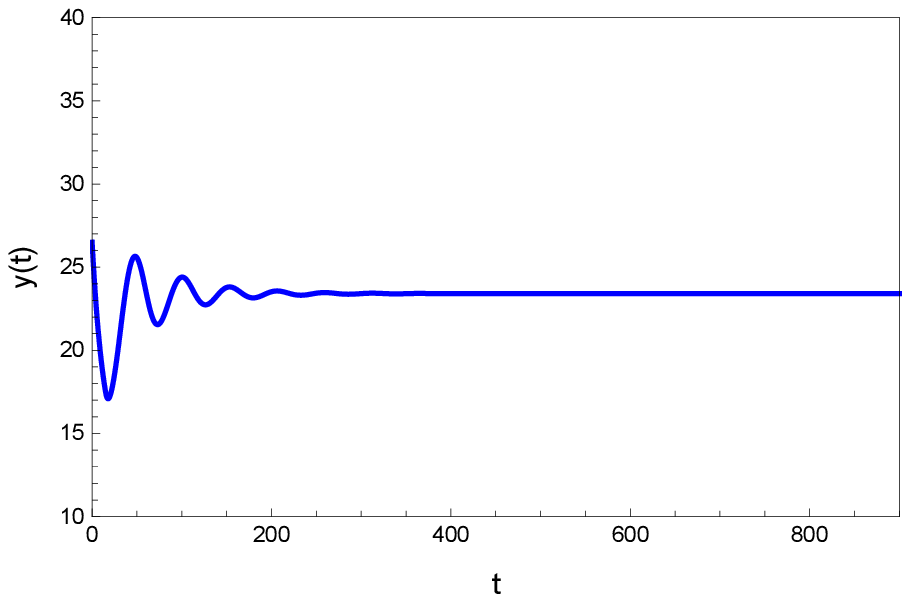}		
			\label{fig:yvst_alpha0p6beta0p6tau15}
		\end{subfigure}
		\hfill
		\begin{subfigure}[b]{0.45\textwidth}
			\centering
			\includegraphics[width=1\textwidth]{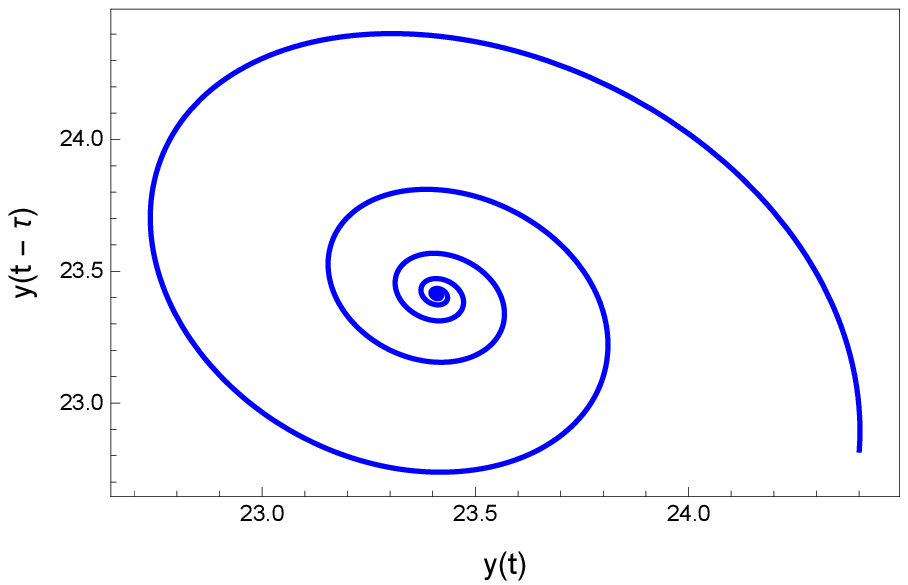}		
			\label{fig:yvsytau_alpha0p6beta0p6tau15}
		\end{subfigure}
		\hfill
		\begin{subfigure}[b]{0.45\textwidth}
			\centering
			\includegraphics[width=1\textwidth]{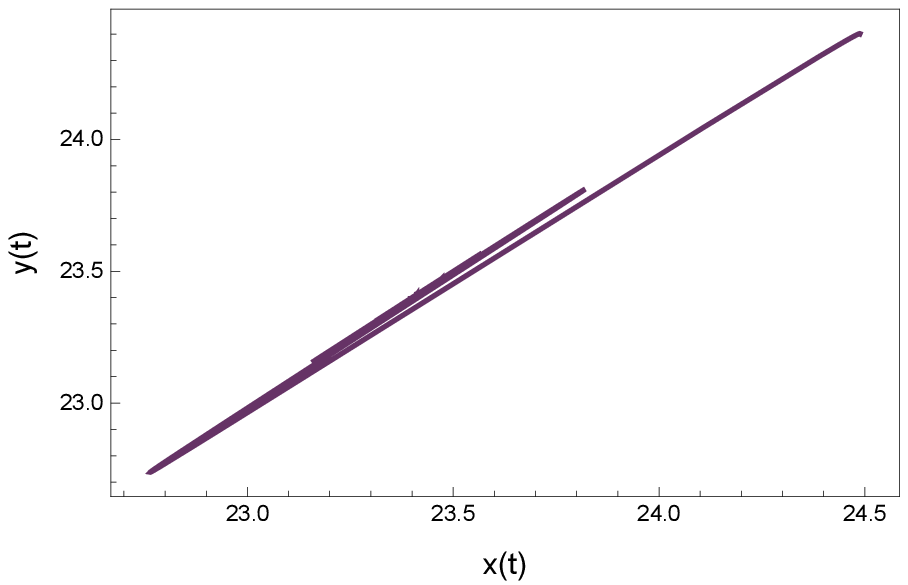}		
			\label{fig:xvsy_alpha0p6beta0p6tau15}
		\end{subfigure}
		\hfill
		\begin{subfigure}[b]{0.45\textwidth}
			\centering
			\includegraphics[width=1\textwidth]{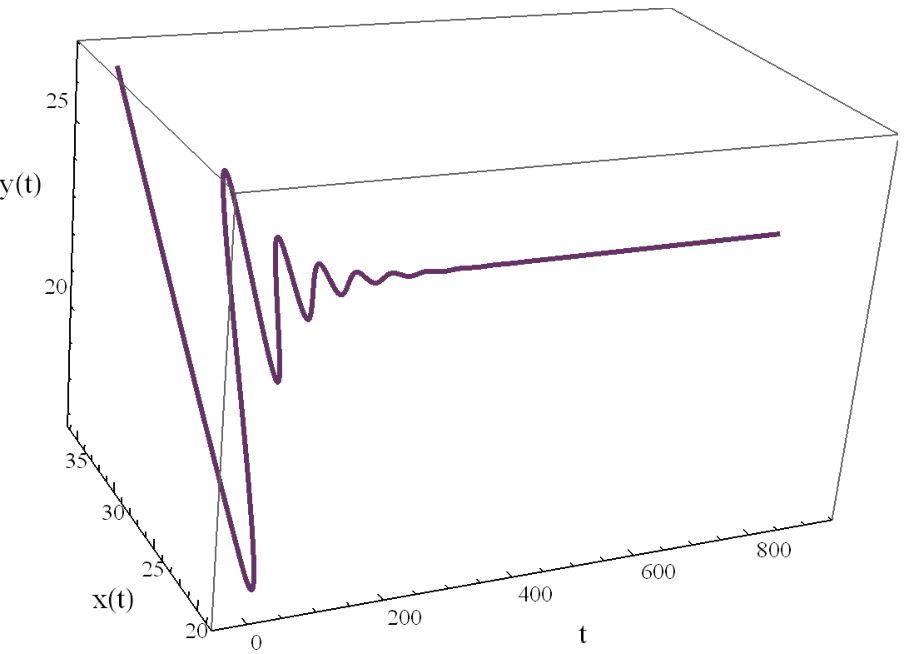}		
			\label{fig:xvsyvst_alpha0p6beta0p6tau15}
		\end{subfigure}
		\hfill
		\begin{subfigure}[b]{0.45\textwidth}
			\includegraphics[width=1\textwidth]{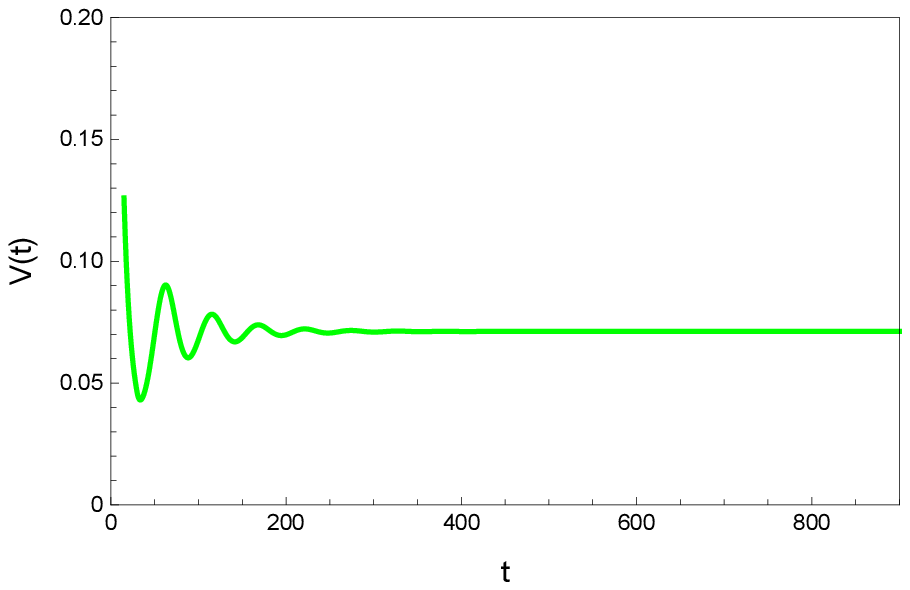}		
			\label{fig:Vvst_alpha0p6beta0p6tau15}
		\end{subfigure}
		\caption{The time series plots, phase plots and the ventilation plot with  $\alpha = 0.6, \; \beta = 0.6 $ and $ \tau = 15.$}
		\label{fig:sys1_alpha0p6_beta0p6_tau15}
		
	\end{figure}
	
	\begin{figure}[H]
		\centering
		\begin{subfigure}[b]{0.45\textwidth}
			\centering
			\includegraphics[width=1\textwidth]{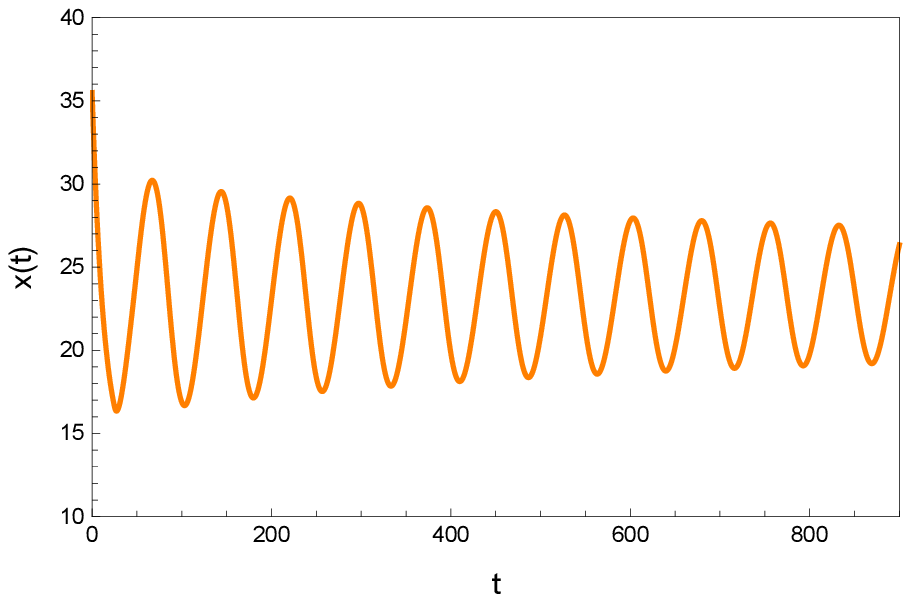}		
			\label{fig:xvst_alpha0p6beta0p6tau24p9}
		\end{subfigure}
		\hfill
		\begin{subfigure}[b]{0.45\textwidth}
			\centering
			\includegraphics[width=1\textwidth]{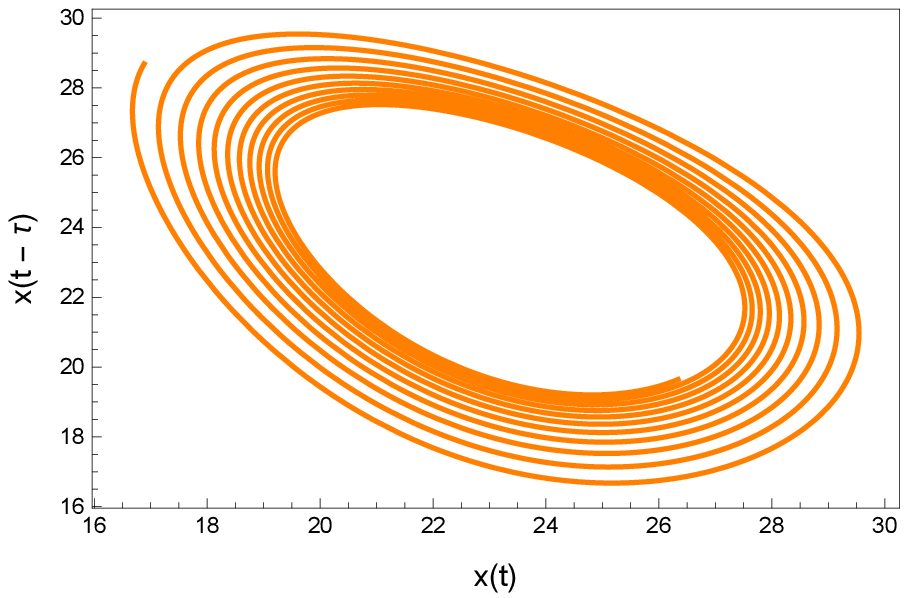}		
			\label{fig:xvsxtaut_alpha0p6beta0p6tau24p9}
		\end{subfigure}
		\hfill
		\begin{subfigure}[b]{0.45\textwidth}
			\centering
			\includegraphics[width=1\textwidth]{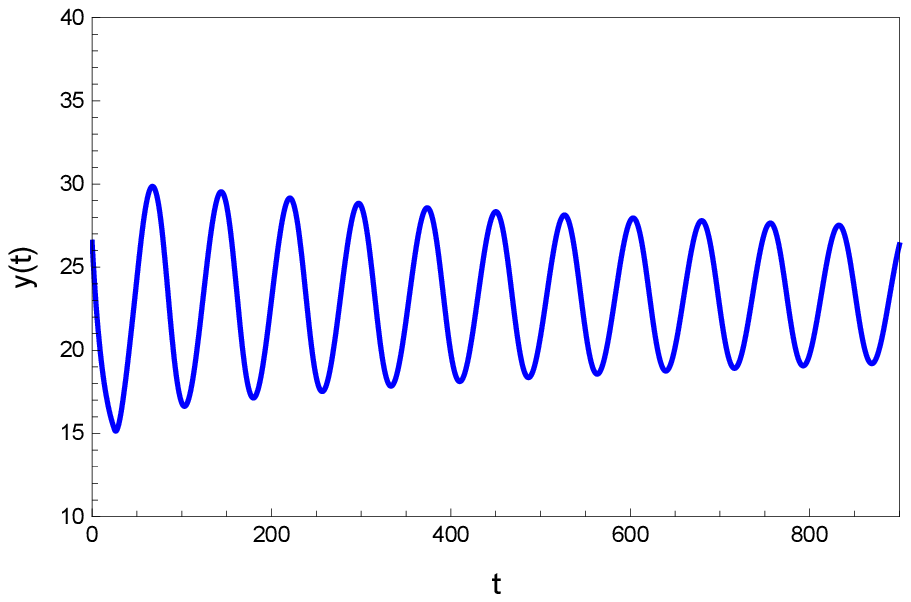}		
			\label{fig:yvst_alpha0p6beta0p6tau24p9}
		\end{subfigure}
		\hfill
		\begin{subfigure}[b]{0.45\textwidth}
			\centering
			\includegraphics[width=1\textwidth]{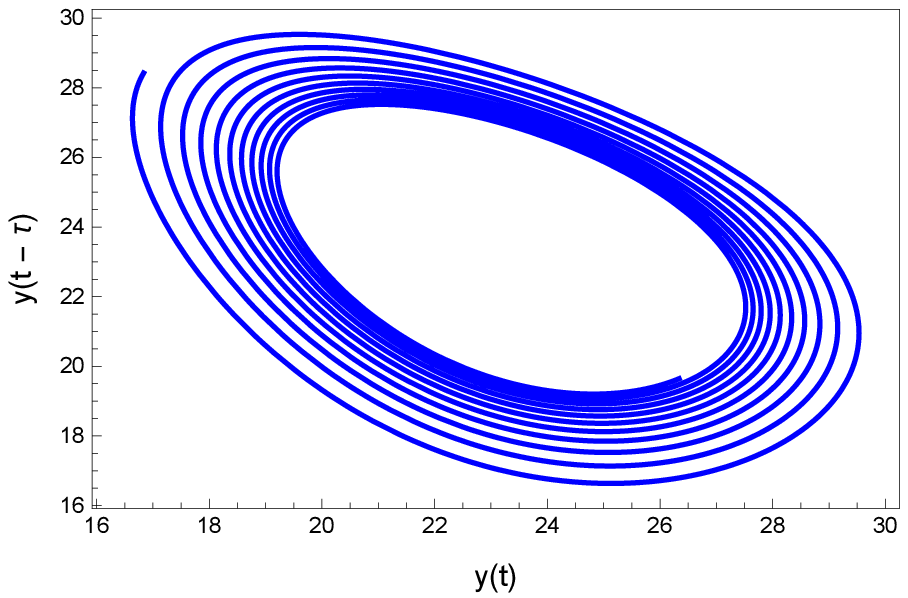}		
			\label{fig:yvsytau_alpha0p6beta0p6tau24p9}
		\end{subfigure}
		\hfill
		\begin{subfigure}[b]{0.45\textwidth}
			\centering
			\includegraphics[width=1\textwidth]{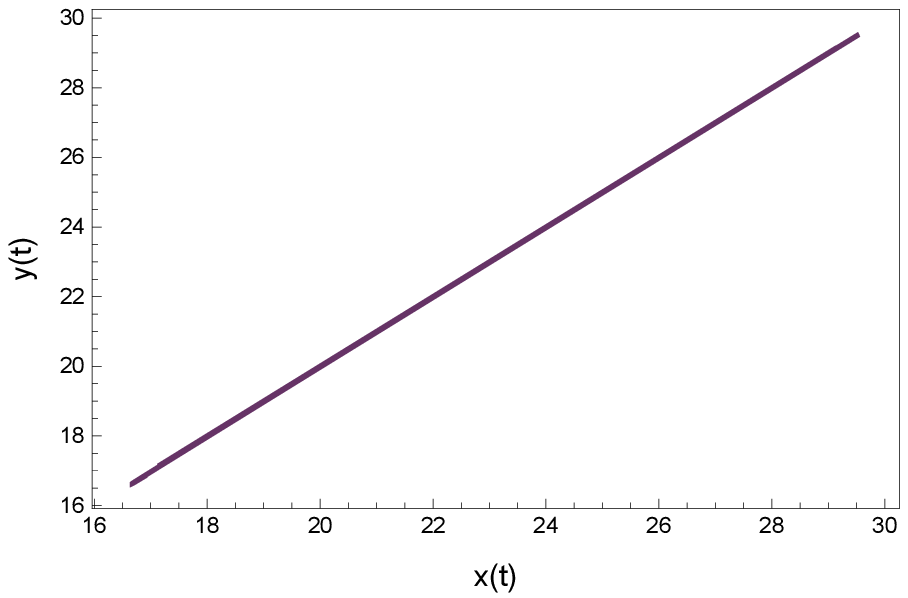}		
			\label{fig:xvsy_alpha0p6beta0p6tau24p9}
		\end{subfigure}
		\hfill
		\begin{subfigure}[b]{0.45\textwidth}
			\centering
			\includegraphics[width=1\textwidth]{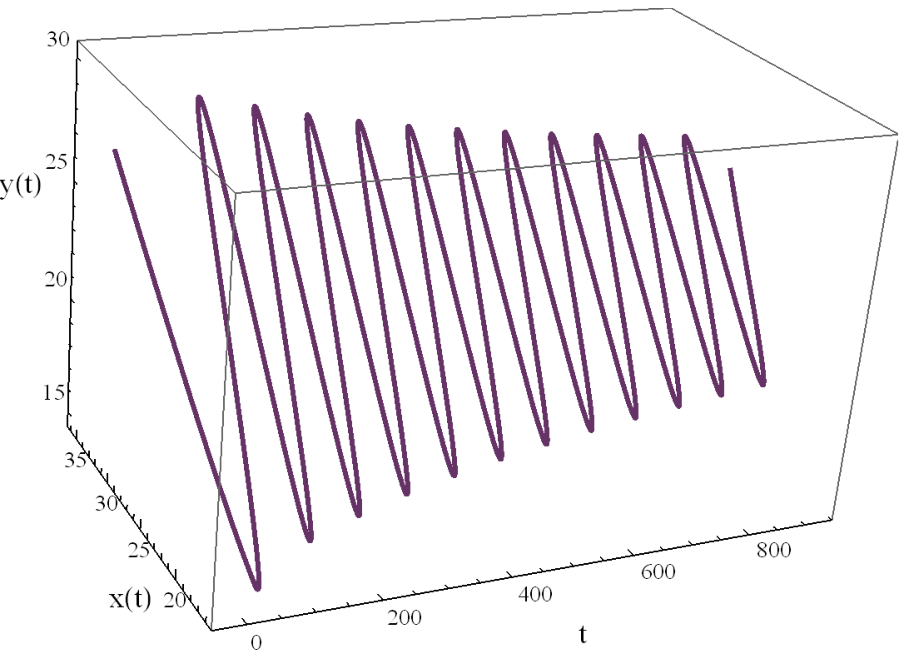}		
			\label{fig:xvsyvst_alpha0p6beta0p6tau24p9}
		\end{subfigure}
		\hfill
		\begin{subfigure}[b]{0.45\textwidth}
			\includegraphics[width=1\textwidth]{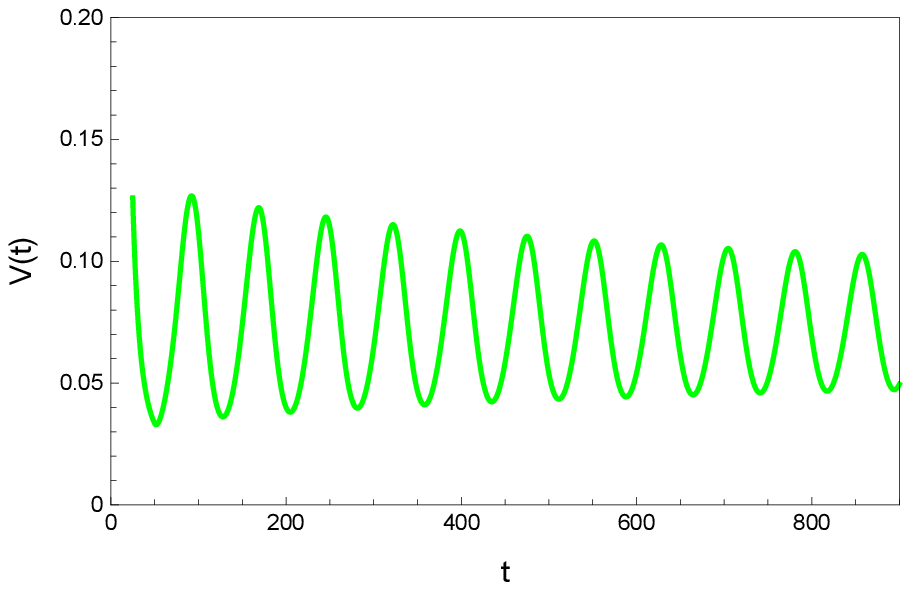}		
			\label{fig:Vvst_alpha0p6beta0p6tau24p9}
		\end{subfigure}
		\caption{The time series plots, phase plots and the ventilation plot with  $\alpha = 0.6, \; \beta = 0.6 $ and $ \tau = 24.9072.$}
		\label{fig:sys1_alpha0p6_beta0p6_tau24p9}
		
	\end{figure}

	\begin{figure}[H]
		\centering
		\begin{subfigure}[b]{0.45\textwidth}
			\centering
			\includegraphics[width=1\textwidth]{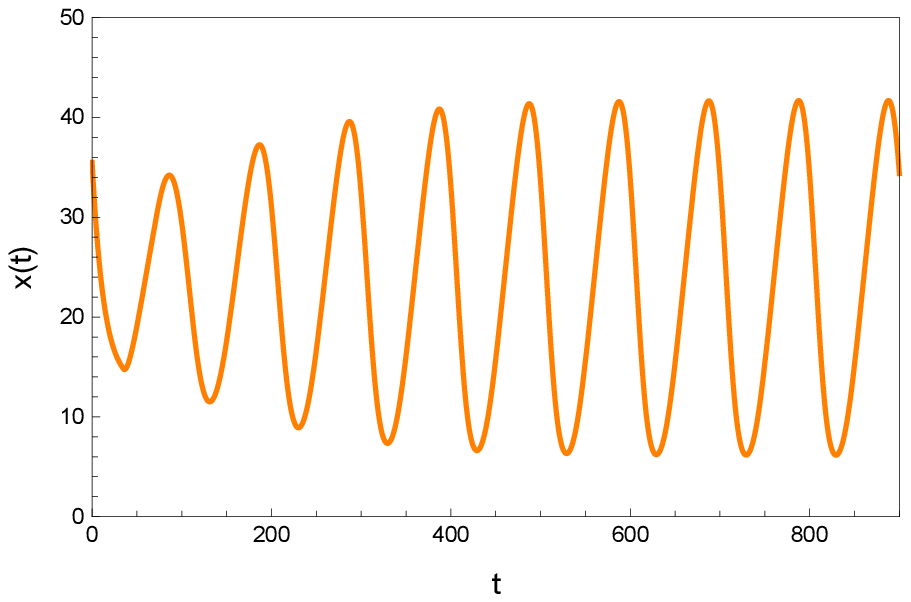}		
			\label{fig:xvst_alpha0p6beta0p6tau35}
		\end{subfigure}
		\hfill
		\begin{subfigure}[b]{0.45\textwidth}
			\centering
			\includegraphics[width=1\textwidth]{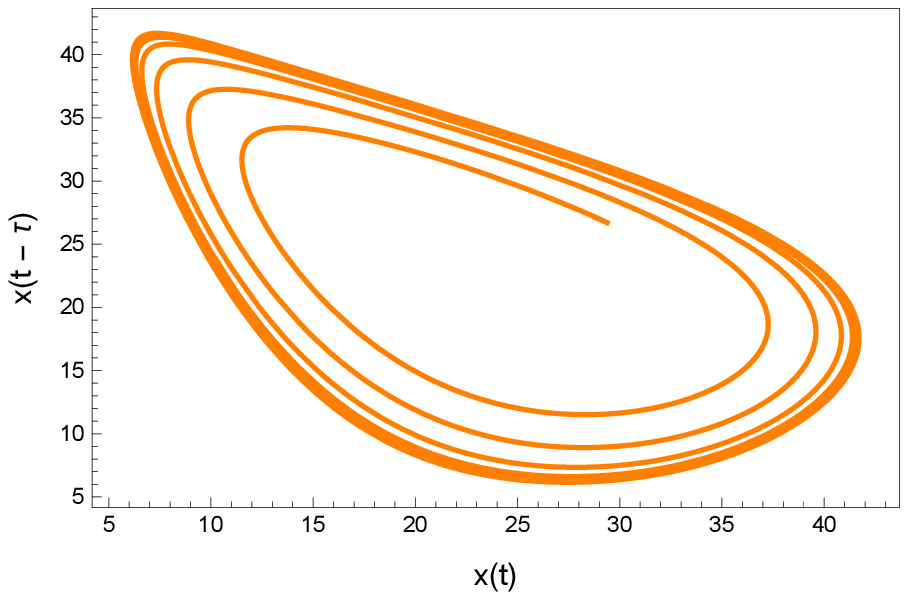}		
			\label{fig:xvsxtaut_alpha0p6beta0p6tau35}
		\end{subfigure}
		\hfill
		\begin{subfigure}[b]{0.45\textwidth}
			\centering
			\includegraphics[width=1\textwidth]{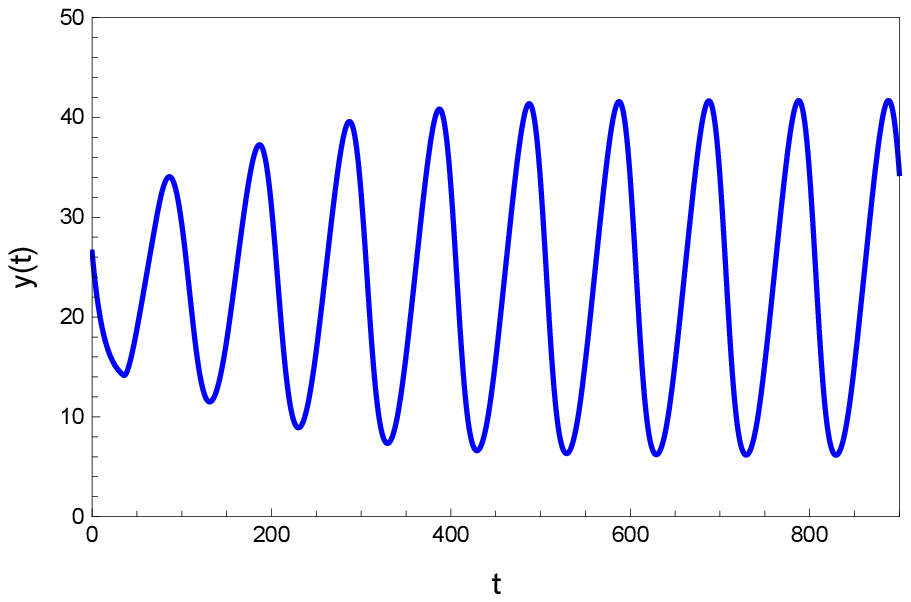}		
			\label{fig:yvst_alpha0p6beta0p6tau35}
		\end{subfigure}
		\hfill
		\begin{subfigure}[b]{0.45\textwidth}
			\centering
			\includegraphics[width=1\textwidth]{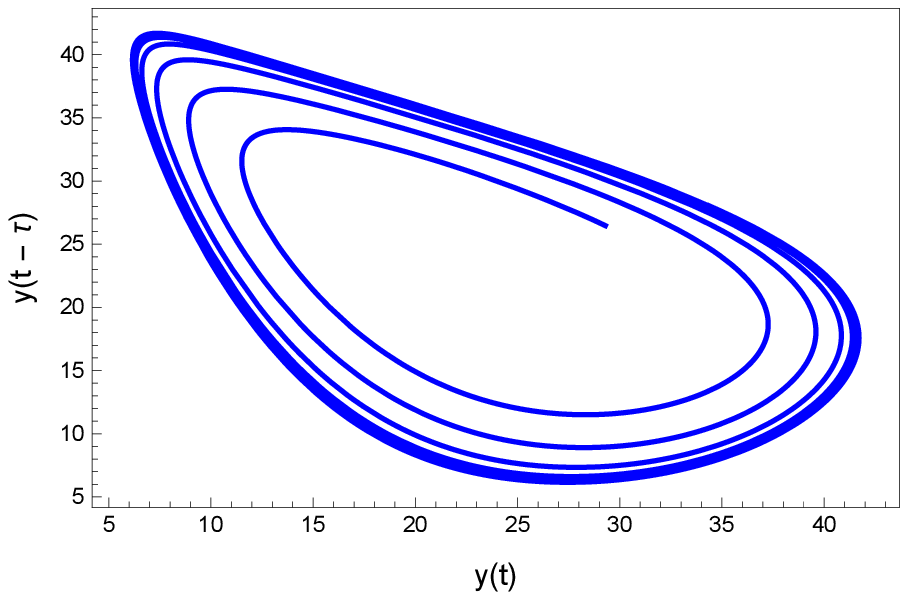}		
			\label{fig:yvsytau_alpha0p6beta0p6tau35}
		\end{subfigure}
		\hfill
		\begin{subfigure}[b]{0.45\textwidth}
			\centering
			\includegraphics[width=1\textwidth]{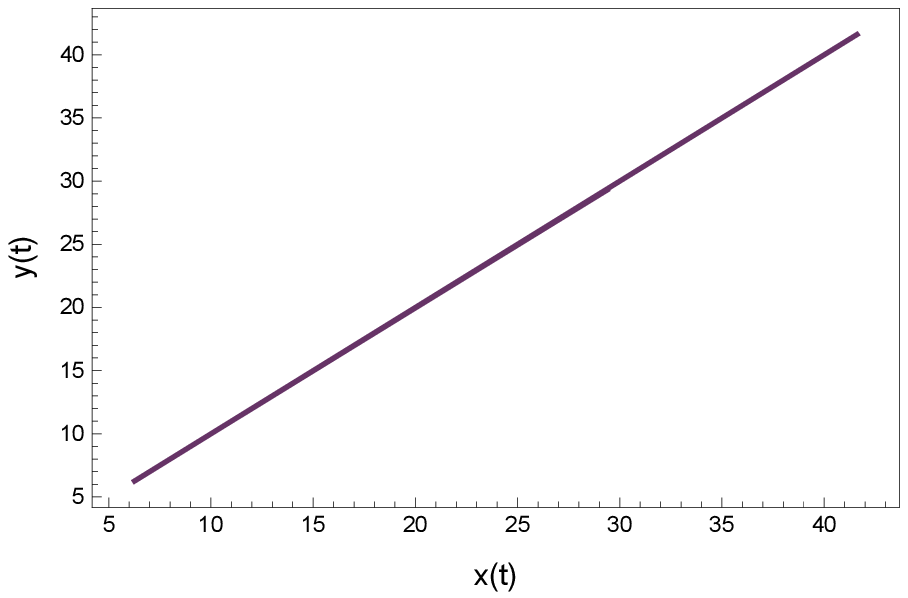}		
			\label{fig:xvsy_alpha0p6beta0p6tau35}
		\end{subfigure}
		\hfill
		\begin{subfigure}[b]{0.45\textwidth}
			\centering
			\includegraphics[width=1\textwidth]{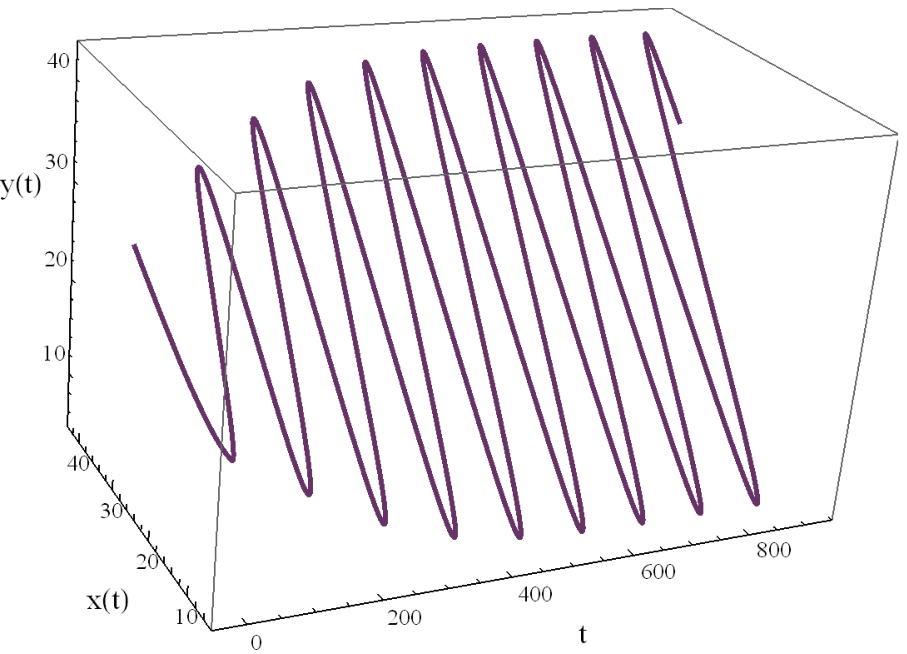}		
			\label{fig:xvsyvst_alpha0p6beta0p6tau35}
		\end{subfigure}
		\hfill
		\begin{subfigure}[b]{0.45\textwidth}
			\includegraphics[width=1\textwidth]{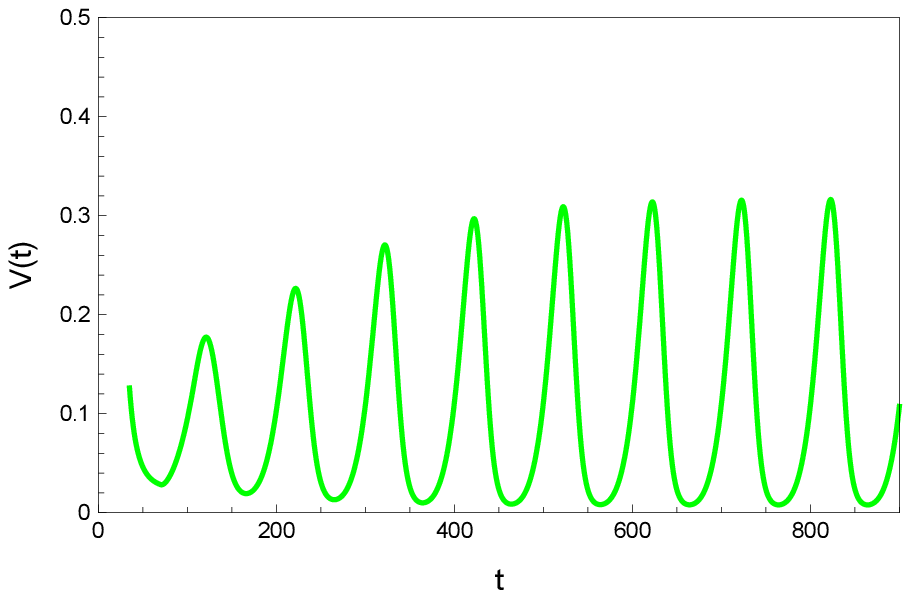}		
			\label{fig:Vvst_alpha0p6beta0p6tau35}
		\end{subfigure}
		\caption{The time series plots, phase plots and the ventilation plot with  $\alpha = 0.6, \; \beta = 0.6 $ and $ \tau = 35.$}
		\label{fig:sys1_alpha0p6_beta0p6_tau35}
		
	\end{figure}


	\section{Conclusion}

	We applied nonlinear delay differential equation in modeling the human respiratory system. The two state model which describes the  balance equation for carbon dioxide and oxygen was studied. This model has three parameters $\alpha,$ $\beta$ and $\tau.$ The parameters $\alpha$ and $\beta$ affect the unique positive equilibrium $(x_*, y_*)$ of the model (see Figures \ref{fig:equi_xyvsalpha_beta0p8} and \ref{fig:equi_xyvsbeta_alpha0p5}) and  the time delay $\tau$ affects the stability of the system (see Figures \ref{fig:cc_beta0p8}, \ref{fig:cc_alpha0p5}, \ref{fig:ccs_tauvsalphavsbeta} and \ref{fig:maxroots_alpha0p5_beta0p8}.) The critical curves (Equation \ref{eq:taustar}) were used in studying the stability of our model. The three dimensional stability chart is constructed as shown in Figure \ref{fig:sr_tauvsalphavsbeta}. There is a region enclosed by $\tau = 0$ and the curve $\tau = \tau_1(0)$ in the $(\tau, \alpha)$ and $(\tau, \beta)$ plane where the equilibrium $(x_*,y_*)$ is stable. We have derived analytical expression for equilibrium point and critical delay as a function of $\alpha$ and $\beta$ and we list some of these numerical values in Table \ref{table:eq_delay}. These values are also verified by plotting bifurcation diagrams. By picking the delay $\tau$ as the bifurcation parameter, the stability of the positive equilibrium $E_*(x_*,y_*)$ and the existence of Hopf bifurcation are derived. The equilibrium is asymptotically stable for $0 \leq \tau \leq \tau_*$, unstable for $\tau > \tau_*.$ The system shows a supercritical Hopf bifurcation giving rise to  stable periodic oscillations. These periodic oscillations may be related to the medical condition we refer as periodic breathing.    It is to be noted that the delay parameter has effect on the stability but not on the equilibrium state. Additionally, 	the explicit derivation of the direction of  Hopf bifurcation and the stability of the bifurcation periodic solutions are determined with the help of normal form theory and center manifold theorem to delay differential equations.
	
	Finally, some numerical example and simulations are carried out to confirm the analytical findings. The numerical simulations verify the theoretical results.

	\bibliography{references1}
\end{document}